%% file: arxiv_v4.tex
\documentclass[9pt,reqno]{amsart}  
\input{packages}
\input{notations}

\usepackage{amsmath,enumitem,amsthm,amsaddr}
\input{bibliography_setup.tex} 
\input{drawGraphs}
\usepgfplotslibrary{external} 
\tikzsetfigurename{eth-figure} 
\bibliography{bibclt}
\usepackage{nameref,zref-xr}
\zxrsetup{toltxlabel=true, tozreflabel=false}
\zexternaldocument*[funcCLT-]{functional_clt_v4}
\bibliography{bibclt}

\newlist{genprop}{enumerate}{1}
\newlist{isoprop}{enumerate}{1}
\newlist{avprop}{enumerate}{1}
\setlist[genprop]{label=(P\arabic*)}
\setlist[isoprop]{label=(P\(^\mathrm{iso}\)\arabic*)}
\setlist[avprop]{label=(P\(^\mathrm{av}\)\arabic*)}

\numberwithin{equation}{section} 
\newtheorem{theorem}{Theorem}[section]
\newtheorem{assumption}[theorem]{Assumption}
\newtheorem{lemma}[theorem]{Lemma}
\newtheorem{proposition}[theorem]{Proposition}
\newtheorem{definition}[theorem]{Definition}

\newtheorem{remark}[theorem]{Remark}

\date{\today}
\author{Giorgio Cipolloni \and L\'aszl\'o Erd\H{o}s}
\address{IST Austria, Am Campus 1, 3400 Klosterneuburg, Austria}
\author{Dominik Schr\"oder\(^{\ast}\)}
\address{Institute for Theoretical Studies, ETH Zurich, Clausiusstr.\ 47, 8092 Zurich, Switzerland}
\email{giorgio.cipolloni@ist.ac.at} 
\email{lerdos@ist.ac.at}
\email{dschroeder@ethz.ch}
\thanks{\(^\ast\)Supported by Dr.\ Max R\"ossler, the Walter Haefner Foundation and the ETH Z\"urich Foundation}
\subjclass[2010]{60B20, 15B52, 58J51, 81Q50} 
\keywords{Eigenstate Thermalization Hypothesis, Feynman diagrams, Local Law, Quantum Unique Ergodicity}
\title{Eigenstate Thermalization Hypothesis for Wigner Matrices}
\date{\today}
\begin{document} 
\thispagestyle{empty}

\begin{abstract}
  We prove that any deterministic matrix is
  approximately the identity in the eigenbasis of a large random Wigner matrix
  with very high probability and with an optimal error inversely proportional
  to  the square root of the dimension. 
  Our theorem thus  rigorously verifies the
  Eigenstate Thermalisation Hypothesis by Deutsch~\cite{9905246}
  for the simplest chaotic
  quantum system, the Wigner ensemble.  In mathematical terms,
  we prove the strong form of Quantum Unique Ergodicity (QUE) with an optimal convergence rate for all eigenvectors simultaneously,
  generalizing previous probabilistic QUE  results in~\cite{MR3606475} and~\cite{MR4156609}.
\end{abstract}

\maketitle

\section{Introduction} 
Since the groundbreaking discovery of E. Wigner~\cite{MR77805} postulating that  Hermitian random matrices
can effectively model the universal  statistics of  gaps between energy levels of large atomic nuclei, simple
random matrices have  been routinely used to replace more complicated quantum
Hamilton operators for many other physically relevant problems, especially in disordered or chaotic quantum systems.
A fundamental phenomenon of such systems is \emph{Quantum Ergodicity (QE)}, stating that the eigenvectors  tend to become 
uniformly distributed  in the phase space.

In this paper we study an enhanced version of  this  question, the  \emph{Quantum Unique Ergodicity (QUE)}, 
for real or complex Wigner matrices and for general observables. We recall
that the Wigner matrix ensemble consists of \(N\times N\)  random Hermitian   matrices \(W=W^*\)
with centred, independent, identically distributed \emph{(i.i.d.)} entries up to the symmetry constraint \(w_{ab} = \ov{w_{ba}}\).
Let \(\set{\bu_i}_{i=1}^N\) be an orthonormal eigenbasis of \(W\). 
Our main  Theorem~\ref{theorem overlap}  asserts that
for any deterministic  matrix \(A\)  with \(\|A\|\le1\) we have the limit \(\braket{\bu_i, A \bu_j}\to\delta_{ij}\braket{A}\) with very high probability (and hence uniform in \(i,j\)) with optimal speed of convergence \(1/\sqrt{N}\), i.e.
\begin{equation}\label{eq:eth}
  \max_{ij}\abs[\Big]{\braket{ \bu_i, A \bu_j}-\braket{A}\delta_{ij}} 
  \lesssim \frac{N^\epsilon}{\sqrt{N}} .
\end{equation}  
Here we introduced the shorthand notation 
\(\langle R \rangle:=\frac{1}{N}\Tr R\) for the normalized trace of any \(N\times N\) matrix. 
In other words,~\eqref{eq:eth} establishes the \emph{QUE in strong form} (i.e.\ uniformly in \(i,j\)) for any Wigner matrix, 
and shows that the action of any bounded traceless deterministic matrix on the eigenbasis \(\set{\bu_i}_{i=1}^N\) makes it  \emph{asymptotically orthogonal} to itself (up to an optimal error \(N^{-1/2}\)).
For genuinely complex Wigner matrices our second main Theorem~\ref{theo:transpoverlap} asserts that
\begin{equation}\label{eq:ethbar}
  \max_{ij} \abs[\Big]{\braket{ \bu_i, \ov{\bu_j}}} \lesssim \frac{N^\epsilon}{\sqrt{N}},
\end{equation}  
again  with very high probability, showing that the eigenbases of \(W\) and \(W^t=\ov{W}\) are asymptotically orthogonal.

The question of ergodicity for general observables is also known as the \emph{Eigenstate Thermalization 
Hypothesis (ETH)} in the physics literature since
the seminal papers of Deutsch~\cite{9905246}
and  Srednicki~\cite{9962049}, 
see also~\cite{Dalessio2015} and~\cite{29862983} for reviews
and further references. Our result thus proves ETH with an optimal speed of
convergence as predicted, e.g., in~\cite[Eqs.~(20)]{Dalessio2015}
for the simplest chaotic quantum system, the Wigner ensemble.

Historically, the most prominent  model for quantum ergodicity
is  the natural quantization of a chaotic classical dynamical system in the semiclassical or
in the high-energy regimes.  The first  mathematical result on QE
was obtained  by Shnirelman~\cite{MR0402834}. It asserts that for \emph{most} 
high energy (normalized) eigenfunctions \(\psi_i\) of the Laplace-Beltrami operator on a surface  with ergodic geodesic flow 
the measures \(|\psi_i(x)|^2 {\rm d} x\) 
become completely flat as \(i\to \infty\). This result was later extended by Colin de Verdi\'ere~\cite{MR818831}
and Zelditch~\cite{MR916129} for much larger classes of observables showing that
if \(A\) is an appropriate pseudodifferential operator with symbol \(\sigma(A)\),  then
\(\langle \psi_i, A \psi_j\rangle \to \delta_{ij} \int_{S^*} \sigma(A)\) 
for \emph{most} index pairs as \(i,j\to \infty\), where \(S^*\) is the unit cotangent bundle of the surface. The analogous result on  large regular graphs was obtained by Anantharaman and Le Masson~\cite{MR3322309}. The celebrated Quantum Unique Ergodicity (QUE)
conjecture, formulated by Rudnick and Sarnak~\cite{MR1266075} in 1994, is a natural strengthening  of these
results  stating that the same limits hold for \emph{all} indices excluding that
exceptional sequences  may exhibit exotic behaviour (\emph{scarring}).
QUE in this form is still an outstanding open question;
only certain special cases have been proven, e.g.\ on arithmetic surfaces for the joint eigenfunctions
of the Laplacian and the Hecke operator by Lindenstrauss~\cite{MR2195133}, with Soundararajan's 
extension~\cite{MR2680500}, see also~\cite{MR3260861,MR2680499}.

The speed of convergence in quantum ergodicity has been a fundamental question in the theory of 
 quantum chaos, see e.g.~\cite{MR2757360} for a review and~\cite{MR2248896} for numerical results. For strongly chaotic (hyperbolic) systems the general physics prediction is
that the variance of \(\langle \psi_i, A \psi_j\rangle\) is proportional with the inverse of the Heisenberg
time, roughly speaking the local eigenvalue spacing
 (see e.g.~\cite[Eq.~(24)]{9964105} %
 building upon earlier results by Feingold and Peres~\cite{9897286}).
For the hyperbolic geodesic flow on general Riemannian manifolds only inverse logarithmic decay 
has been proven  by Zelditch~\cite{MR1262192}  
and Schubert~\cite{MR2267060}
which is even optimal for a special highly degenerate eigenbasis of the quantization of Arnold's cat map~\cite{MR2465731}, see also~\cite{MR3732880} for surfaces of high genus. For similar quantitative QUE results on large deterministic graphs see~\cite{MR3038543,MR3649482,MR3961083}. 
Much stronger polynomial bounds hold for special arithmetic surfaces proven by Luo and Sarnak~\cite{MR1361757}
and for linear  maps on the torus~\cite{MR2150390,MR1810753} and toral eigenfunctions~\cite{MR3662015}. For random \(d\)-regular graph optimal polynomial speed of convergence for QUE with diagonal observables has been obtained in~\cite{MR3962004,MR3688032}.

For   large Wigner matrices, using the Dyson Brownian motion (DBM) for eigenvectors,  Bourgade and Yau~\cite{MR3606475} showed  that for any fixed deterministic unit vector  \(\bq\) and any \(i\)
in the bulk spectrum or close to the edge, the squared overlaps
\(N |\langle \bu_i, \bq\rangle |^2\) converge in distribution  to the square of a standard Gaussian as \(N\to\infty\)
(see also~\cite{MR4164858} for deformed Wigner matrices and~\cite{MR3034787, MR2930379} for
the same result under four moment matching condition in the bulk). 
The corresponding a priori bound, asserting that 
\(N |\langle \bu_i, \bq\rangle |^2\lesssim N^\epsilon\) with very high  probability for any \(\epsilon>0\), 
has been  known beforehand as the \emph{complete delocalisation of eigenvectors}~\cite{MR3183577, MR2481753, MR2871147, MR3103909}. DBM methods allow to obtain optimal delocalisation estimates~\cite{2007.09585}. We mention that~\cite{MR3606475} also obtains asymptotic normality for the  joint distribution of
finitely many eigenvectors tested against one fixed vector \(\bq\)
and for the joint distribution 
of a single eigenvector with finitely many test vectors \(\bq_1, \bq_2, \ldots \bq_K\). Very recently the joint normality of
finitely many eigenvectors and finitely many test vectors has also been achieved~\cite{2005.08425}.

These results based upon DBM establish the universality of Gaussian fluctuation for individual eigenvectors tested against 
\emph{finite rank}  observables, \(A=\frac{N}{K}\sum_{k\le K} a_k |\bq_k\rangle\langle\bq_k|\), with \(K\) being \(N\)-independent
and \(a_k\in[-1,1]\). The key mechanism of QUE for general  observables is
the self-averaging (ergodic) property of this sum as the rank  \(K= K(N)\) tends to infinity.
As a simple corollary of the fluctuation results, QUE in a weak form was also obtained
asserting that  \(\langle \bu_i, A \bu_i\rangle \to  \langle  A\rangle\) in probability for any fixed \(i\) in the bulk
if \(\rank(A)\) grows with \(N\), see~\cite[Corollary 1.4]{MR3606475}
(this result was stated only for diagonal matrices \(A\), but it directly generalizes to any \(A\)
by spectral decomposition).
However, the effective probabilistic estimates in~\cite{MR3606475} were not 
sufficient  to prove the strong form of QUE, i.e.
to guarantee that the limit holds for all eigenvectors simultaneously.
This uniformity was proven in~\cite[Theorem 2.5]{MR4156609}
but only for random matrices with a Gaussian component of size \(t\gg 1/N\)
with an error of order \(1/\sqrt{Nt}\). An off-diagonal version, \(\langle \bu_i, A \bu_j\rangle \to 0\) for \(i\ne j\), coined
as quantum weak mixing, was also obtained  in~\cite{MR4156609} and strengthened in~\cite{1908.10855}. Standard Green function comparison arguments 
may be used to remove the large Gaussian component but only with a considerably suboptimal error 
or under the extra assumption of matching the first several (in fact  more than four) moments of the matrix elements of \(W\) with 
those of the Gaussian GOE/GUE ensemble.

Summarizing,  our Theorem~\ref{theorem overlap}  generalizes the
probabilistic  QUE proven in~\cite[Corollary 1.4]{MR3606475} and in~\cite[Theorem 2.5]{MR4156609}
to general Wigner ensembles
in three  aspects:  (i) the speed of convergence is optimal (up to an \(N^\epsilon\) factor);
(ii) the limit  is controlled  in very high probability, and (iii) it holds uniformly throughout the spectrum
including bulk, edge and the intermediate regime. For any deterministic Hermitian observable written in spectral decomposition \(A=\sum_{k=1}^N a_k |\bq_k\rangle \langle \bq_k|\), our main result,
\begin{equation}\label{eq:etherr}
  \abs*{\langle \bu_i, A \bu_j\rangle -  \delta_{ij}\langle A\rangle} %
  =  \abs*{\frac{1}{N}\sum_{k=1}^N a_k \Bigl( N \langle \bu_i, \bq_k \rangle \langle \bq_k, \bu_j \rangle - \delta_{ij} \Bigr)} \lesssim \frac{N^\epsilon}{\sqrt{N}},
\end{equation} 
shows that the fluctuations of 
\(N \langle \bu_i, \bq_k \rangle \langle \bq_k, \bu_j \rangle \)
are so strongly asymptotically independent for different \(k\)'s that their average 
has the expected \(1/\sqrt{N}\) fluctuation scaling reminiscent to 
the central limit theorem, up to an \(N^\epsilon\) factor. In fact, in our companion paper~\cite[Theorem~\ref{funcCLT-theo:sharpcom}]{2012.13218} 
we also show that the diagonal overlaps \(\braket{\bu_i,A\bu_i}\), after a small averaging in the index \(i\), satisfy a CLT\@.

Next we outline the novel ideas of our  proof.
We consider a spectrally averaged version of   the overlaps
\begin{equation}\label{defL}
  \Lambda^2:=  \max_{i_0, j_0}\frac{1}{(2J)^2}\sum_{\substack{|i-i_0|< J\\ |j-j_0|< J}} N |\langle \bu_i, A\bu_j\rangle|^2
\end{equation}
for  bounded traceless observables, \(\langle A\rangle=0\), \(\|A\|\le 1\), 
where \(J= N^\epsilon\) with some tiny \(\epsilon>0\). Our goal is to show that \(\Lambda\) 
is essentially of order one, with high probability.
Denoting by \(G=G(z)= (W-z)^{-1}\) the resolvent
at \(z\in \HC\), notice that, by spectral decomposition,
\begin{equation}\label{LamIm}
  \Lambda^2 \sim \sup_{E, E'\in [-2,2]} (\rho\rho')^{-1} \langle \Im G(E+\ii \eta) A \Im G (E'+\ii\eta') A\rangle,
\end{equation}
where \(\eta\) is slightly above the local eigenvalue spacing at \(E\) and \(\rho\) is the semicircular density 
at \(E\) smoothed out on scale \(\eta\); the primed quantities defined analogously. The main work consists in
proving a high probability  optimal bound on the quadratic functional of the resolvent \(\langle GAGA\rangle\),
with possible imaginary parts and at different spectral parameters. Note that for overlaps with a rank one observable, \(A= |\bq\rangle\langle \bq|\), it is
sufficient to control \(\langle \bu_i, A\bu_i\rangle= |\langle \bq, \bu_i\rangle|^2\). After a mild local averaging in the index \(i\) this
becomes comparable with \(\langle \bq, ( \Im G)\bq\rangle\) whose control is  equivalent to a conventional 
single-\(G\) isotropic local law. This  served as a natural input for the DBM proofs on eigenvectors in~\cite{MR3606475, MR4156609}. For  traceless observables, however, \(\langle \bu_i, A\bu_i\rangle\) does not have a sign,
so we need to consider \(|\langle \bu_i, A\bu_i\rangle|^2\) to understand its size, hence the relevant quantity 
is \(\Lambda^2\) containing two \(G\) factors~\eqref{LamIm},  i.e.\ single-\(G\) local laws
are not sufficient. 

For estimating~\eqref{LamIm} we face a combination of two serious difficulties. First, we need to gain an additional
cancellation  from the fact that \(A\) is traceless; second, we need to handle local laws for products of several \(G\)'s. The first issue already arises
on the level of a single-\(G\) local law: In Theorem~\ref{pro:1g} we will prove that the resolvent approximation \(G\approx m\) by the Stieltjes transform \(m=m(z)\) of the Wigner semicircular density, commonly referred to as a \emph{local law}, holds to a higher accuracy when tested against a traceless observable. More precisely for the decomposition \(\braket{GA} = \braket{A}\braket{G} + \braket{G(A-\braket{A})}\) and with \(\rho:=\abs{\Im m}/\pi\) we have that 
\begin{equation}\label{eq:Gloc}
  \braket{G}= m+\landauO*{\frac{1}{N\eta}}, \quad \braket{G(A-\braket{A})} = \landauO*{\frac{\rho^{1/2}}{N\eta^{1/2}}},
\end{equation}
with both errors being optimal, in fact they identify the scale of the asymptotic Gaussian fluctuation of  $\braket{G}$ and $ \braket{G(A-\braket{A})}$, respectively~\cite{MR3678478,2012.13218}.
Note that the error term for the traceless part is much smaller than that  for $\braket{G}$ in the relevant small $\eta$ regime.  For \(\braket{ GAG^*A}\) the
discrepancy is even bigger; without zero trace assumption \(\braket{ GAG^*A} \sim 1/\eta\) (e.g.\ for \(A=I\)), while for \(\langle A\rangle=0\)  we will 
show that \(\langle GAG^*A\rangle \sim 1\)  even for very small \(\eta\). 

The second issue touches upon the  basic mechanism of the standard  proof of the local laws. It consists in deriving an approximate self-consistent
equation for the quantity in question, e.g.\ \(\langle GAGA\rangle\), and compare it with the corresponding deterministic
equation \emph{(Dyson equation)} without approximation error. The  main error term  \(\langle \underline{WGAGA}\rangle\),
a renormalized version of  \(\langle WGAGA\rangle\), see~\eqref{eq:defunder},  
is expected to be  smaller than \(\langle GAGA\rangle\), but when estimating its high moments by a cumulant 
expansion  many terms with traces of more than two \(G\)-factors emerge.
Trivial a priori bounds using \(\| G\|\le1/\eta\) are not affordable, so one has to continue expanding, resulting in higher and higher
degree monomials in \(G\); reminiscent
to the notoriously difficult closure problem in the BBGKY hierarchy for the correlation functions 
of  interacting particle dynamics. In the proof of the conventional local law \(\braket{G}=m+\landauO{1/N\eta}\), the expansion is stopped by using the Ward identity \(GG^* = \Im G/\eta\), reducing the number of \(G\) factors by one. However, with a deterministic matrix in between, as in \(GAG^*\), Ward identity is not applicable. A trivial Schwarz bound followed by Ward identity,
\begin{equation}\label{Schw}
  |\langle GAG^*A \rangle| \le \langle GAA^*G^* \rangle = \frac{1}{\eta} \langle (\Im G)AA^* \rangle,
\end{equation}
is available, but at the expense of replacing the traceless matrix \(A\) with the non-zero trace matrix \(AA^*\), hence losing the main cancellation effect that we cannot afford. Our main idea is to use \(\Lambda\) from~\eqref{defL} as
the basic control quantity and  derive a stochastic Gronwall inequality for it. In doing so, we use the spectral decomposition of \(G\) to estimate traces of products of many  \(G's\) and \(A'\)s by the  lower degree term \(\langle GAGA\rangle\). Technically, this requires to extract 
sufficiently many \(\Lambda\)-factors  in the cumulant expansion, which we achieve by a subtle Feynman graph analysis to estimate all  high moments of \(|\langle \underline{WGAGA}\rangle|\).

Feynman diagrams have been systematically used to organize  cumulant expansions and their estimates come
on different levels of sophistication, see e.g.~\cite{MR3941370, MR4134946, 1912.04100}, but also related
expansions in random matrices, e.g.~\cite{MR3119922, MR3302637,MR3485343}. 
For the proof of~\eqref{eq:Gloc} via cumulant expansion (e.g.\ following~\cite{MR3941370}),
it is sufficient to monitor 
the number of \(N\)-factors (from the size of the cumulants and from the summation of intermediate indices)
and the number of \(\rho/\eta\) factors from the Ward identity. 
In the current  analysis we additionally  need to monitor the \(\Lambda\) factors. While the number of traceless \(A\)-factors is preserved along the expansion, but the cancellation effect of some of them may be lost as in~\eqref{Schw}. Our proof has to carefully offset all such  losses by the gains from higher order cumulants that typically accompany the loss of effective \(A\)-factors. In particular, since we are aiming at an optimal bound, in the expansion terms that involve only second order cumulants
we need to gain from \emph{all} \(A\)-factors.  We used a similar but much simpler expansion in our work on CLT for non-Hermitian random matrices, see~\cite[Prop.~5.3, Eq.~(5.10c)]{1912.04100}, where 
the additional smallness came from the large distance between two
(non-Hermitian) spectral parameters \(z_1, z_2\). However, in~\cite{1912.04100}  it was sufficient to gain 
only a small  proportion of all possible smallness factors since we did not aim
at the optimal bound. In the current paper, using a refined combinatorics
we manage to extract the zero trace orthogonality effect to the maximal extent; this is the key to obtain the optimal error bound in~\eqref{eq:etherr}. Similarly, for the proof of~\eqref{eq:ethbar} we manage to extract the \emph{asymptotic orthogonality} effect between the eigenvectors \(\bm u_i\) and their complex conjugates \(\ov{\bm u_i}\) optimally, resulting in the bound \(\abs{\braket{GG^t}}\lesssim 1\), gaining a full power of \(\eta\) over e.g.\ \(\braket{GG^\ast}\sim 1/\eta\).

After this introduction and presenting the main results in the next Section~\ref{sec main results}, we prove the local laws
involving two resolvents in Section~\ref{sec que}. The main inputs  for them are  the  improved bounds on renormalized (``underlined'') monomials in several \(G\)'s in Theorem~\ref{chain G underline theorem} that are proven in Section~\ref{sec cum exp}. Note that even though we are interested in local laws only with two \(G\)'s, due
to the cumulant expansion we need to control 
arbitrary long monomials involving  a product of \(G\)'s and \(A\)'s.

\subsection*{Notations and conventions}

We introduce some notations we use throughout the paper. For integers \(k\in\N \) we use the notation \([k]:= \set{1,\ldots, k}\). We write \(\HC \) for the upper half-plane \(\HC := \set{z\in\C \given \Im z>0}\). For positive quantities \(f,g\) we write \(f\lesssim g\) and \(f\sim g\) if \(f \le C g\) or \(c g\le f\le Cg\), respectively, for some constants \(c,C>0\) which depend only on the constants appearing in~\eqref{eq:momentass}. We denote vectors by bold-faced lower case Roman letters \({\bm x}, {\bm y}\in\C ^k\), for some \(k\in\N\). Vector and matrix norms, \(\norm{\vx}\) and \(\norm{A}\), indicate the usual Euclidean norm and the corresponding induced matrix norm. For any \(N\times N\) matrix \(A\) we use the notation \(\braket{ A}:= N^{-1}\Tr  A\) to denote the normalized trace of \(A\). Moreover, for vectors \({\bm x}, {\bm y}\in\C^N\) we define
\[ \braket{ {\bm x},{\bm y}}:= \sum \overline{x}_i y_i, \qquad A_{\vx\vy}:=\braket{\vx,A\vy},\]
with \(A\in\C^{N\times N}\). We will use the concept of ``with very high probability'' meaning that for any fixed \(D>0\) the probability of the \(N\)-dependent event is bigger than \(1-N^{-D}\) if \(N\ge N_0(D)\). Moreover, we use the convention that \(\xi>0\) denotes an arbitrary small constant which is independent of \(N\).

\section{Main results}\label{sec main results}

We consider real symmetric or complex Hermitian \(N\times N\) Wigner matrices \(W\). We formulate the following assumptions on the entries of \(W\). 

\begin{assumption}\label{ass:entr}
  The matrix elements \(w_{ab}\) are independent up to the Hermitian symmetry \(w_{ab}=\overline{w_{ba}}\). 
  We assume identical distribution in the sense that \(w_{ab}\stackrel{\mathrm{d}}{=} N^{-1/2}\chi_{\mathrm{od}}\), for \(a<b\), \(w_{aa}\stackrel{\mathrm{d}}{=}N^{-1/2} \chi_{\mathrm{d}}\), with \(\chi_{\mathrm{od}}\) being a real or complex random variable and \(\chi_{\mathrm{d}}\) being a real random variable such that \(\E \chi_{\mathrm{od}}=\E \chi_{\mathrm{d}}=0\) and \(\E |\chi_{\mathrm{od}}|^2=1\). 
  In the complex case we also assume that \(\E \chi_{\mathrm{od}}^2\in\R\). In addition, we assume the existence of the high moments of \(\chi_{\mathrm{od}}\), \(\chi_{\mathrm{d}}\), i.e.\ that there exist constants \(C_p>0\), for any \(p\in\N \), such that
  \begin{equation}\label{eq:momentass}
    \E |\chi_{\mathrm{d}}|^p+\E |\chi_{\mathrm{od}}|^p\le C_p.
  \end{equation}
  In this paper we use the notations \(w_2:=\E \chi_{\mathrm{d}}^2\), \(\sigma:=\E \chi_{\mathrm{od}}^2\) and their commonly occurring combination \(\widetilde{w_2}:=w_2-1-\sigma\),  
  and note that \(w_2,\widetilde{w_2},\sigma\in \R\).
\end{assumption}

Our first main result is the proof of the \emph{Eigenstate Thermalisation Hypothesis}, that in mathematical terms is the proof of an optimal convergence rate of the strong \emph{Quantum Unique Ergodicity (QUE)} for general observables uniformly in the spectrum of \(W\). 

\begin{theorem}[Eigenstate Thermalization Hypothesis]\label{theorem overlap}
  Let \(W\) be a Wigner matrix satisfying Assumption~\ref{ass:entr}, and denote by \(\bm u_1,\dots,\bm u_N\) its orthonormal
  eigenvectors. Then for any deterministic matrix \( A\) with \(\norm{A}\lesssim 1\) it holds
  \begin{equation}\label{eq:eth1}
    \max_{i,j\in [N]}\abs{\braket{\bm u_i, A \bm u_j}-\braket{A}\delta_{ij}}+\max_{i,j\in [N]}\abs{\braket{\bm u_i, A \overline{\bm u_j}}-\braket{A}\braket{\bm u_i,\overline{\bm u_j}}} \le  \frac{N^\xi}{\sqrt{N}},
  \end{equation}
  with very high probability for any arbitrary small \(\xi>0\).
\end{theorem}

The first relation in Theorem~\ref{theorem overlap} states that any deterministic observable is essentially diagonal in the eigenbasis of \(W\), in other words the eigenvectors remain \emph{asymptotically orthogonal}  when tested against any traceless 
observable \(\braket{A}=0\). The second relation shows the same phenomenon between the eigenbasis \(\set{{\bm u}_i}_{i\in [N]}\) of \(W\) and the eigenbasis \(\set{\overline{{\bm u}_i}}_{i\in [N]}\) of \(W^t\). The scalar products \(\braket{\bm u_i,\overline{\bm u_j}}\) appearing 
in~\eqref{eq:eth1}
when \(\braket{A}\ne 0\) can also be identified. Indeed, the next theorem shows that these two eigenbases are also essentially orthogonal apart from the extreme cases \(\sigma=\pm 1\) (see Remark~\ref{rem:extrcas} below).

\begin{theorem}\label{theo:transpoverlap}
  Let \(W\) be a Wigner matrix satisfying Assumption~\ref{ass:entr}, and denote by \(\bm u_1,\dots,\bm u_N\) its 
orthonormal eigenvectors. Recall \(\sigma:=\E \chi_{\mathrm{od}}^2\) and assume \(|\sigma|<1\), then there is
a constant $C_\sigma<\infty$ such that
  \begin{equation}
    \max_{i,j\in [N]}\abs{\braket{\bm u_i, \overline{\bm u_j}}} \le C_\sigma \frac{N^\xi}{\sqrt{N}},
  \end{equation}
  with very high probability for any arbitrary small \(\xi>0\).
\end{theorem}

\begin{remark}\label{rem:extrcas}
  Theorem~\ref{theo:transpoverlap} does not hold for \(\sigma=\pm 1\). Indeed, for \(\sigma=1\) the matrix \(W\) is real symmetric hence the eigenvectors can be chosen real so \(\abs{\braket{\bm u_i, \overline{\bm u_j}}}=\delta_{ij}\). On the other hand for \(\sigma=-1\) and \(w_2=0\) the spectrum is symmetric, i.e.\ the  eigenvalues $\lambda_1\le\lambda_2\le\dots$ and the corresponding eigenvectors ${\bm u}_1,{\bm u}_2,\dots$ come in pairs, \(\lambda_{N-i+1} = -\lambda_i\) and \(\bu_{N-i+1}=\ov{\bu_{i}}\) (up to phase) and thus \(\abs{\braket{\bm u_i, \overline{\bm u_j}}}=\delta_{i,N-j+1}\).
\end{remark}

The main inputs to prove Theorems~\ref{theorem overlap}--\ref{theo:transpoverlap} are  the 
local laws for one and two resolvents (and their transposes) tested against matrices \(A\) with \(\braket{A}=0\). 
We recall that in the limit \(N\to \infty\) the resolvent \(G=G(z)=(W-z)^{-1}\) becomes approximately deterministic. Its deterministic approximation is given by \(m=m_{\mathrm{sc}}\), the Stieltjes transform of the Wigner semicircular law, which is given by the unique solution of the 
quadratic equation
\begin{equation}\label{eq:mdesc}
  -\frac{1}{m}=z+m, \qquad \Im m(z) \Im z>0.
\end{equation}
We note that \(|m|\le 1\) for any \(z\).
In this paper we allow spectral parameters with \(\Im z<0\), in order to conveniently account for possible adjoints of the resolvent since \(G(z)^*=G(\overline{z})\). Therefore, in contrast with most papers on local laws, \(\Im m_{\mathrm{sc}}\) may be negative and we define \(\rho=\rho_{\mathrm{sc}}(z):=\pi^{-1}|\Im m_{\mathrm{sc}}|\) and $\eta:= |\Im z|$.  

The classical local law (see e.g.\ in~\cite{MR2871147,MR3103909,MR3183577})
 for a single resolvent in \emph{averaged} and \emph{isotropic form} states that 
in the spectral regime $\set{z \given N\rho\eta\ge 1}$ we have
 \begin{equation}\label{eq:oldlocal}
    |\braket{(G-m)A}|\prec \frac{1}{N\eta}, \qquad   \abs*{\braket{{\bm x}, (G-m){\bm y}}}\prec \sqrt{\frac{\rho}{N\eta}}
  \end{equation}
for any deterministic matrix $A$ and vectors \({\bm x}, {\bm y}\), with \(\|A\|, \norm{\bm x}, \norm{\bm y}\lesssim 1\).
Here $\prec$ indicates the commonly used concept of stochastic domination (see, e.g.~\cite{MR3068390}) indicating
a bound with very high probability up to a factor $N^\epsilon$ for any small $\epsilon>0$, uniformly in $A, {\bm x}, {\bm y}$ 
and in the spectral parameter $z$ 
as long as $N\rho\eta\ge 1$.
The precise
definition is as follows:
\begin{definition}[Stochastic Domination]\label{def:stochDom}
  If \[
  X=\tuple*{ X^{(N)}(u) \given N\in\N, u\in U^{(N)} }\quad\text{and}\quad Y=\tuple*{ Y^{(N)}(u) \given N\in\N, u\in U^{(N)} }\] are families of non-negative random variables indexed by \(N\), and possibly some parameter \(u\), then we say that \(X\) is stochastically dominated by \(Y\), if for all \(\epsilon, D>0\) we have \[\sup_{u\in U^{(N)}} \Prob\left[X^{(N)}(u)>N^\epsilon  Y^{(N)}(u)\right]\leq N^{-D}\] for large enough \(N\geq N_0(\epsilon,D)\). In this case we use the notation \(X\prec Y\) or $X= \landauOprec*{Y}$.
\end{definition}

Our key new insight is that
whenever the deterministic matrix $A$ in~\eqref{eq:oldlocal} is traceless, then 
 $\braket{GA}$ is considerably 
smaller (by a factor of $\sqrt{\rho\eta}$, in the interesting small $\eta$ regime) than 
the general bound~\eqref{eq:oldlocal} predicts. There is no such improvement for the isotropic law. 

\begin{theorem}[Traceless single \(G\) local law]\label{pro:1g} 
Fix $\epsilon>0$,  let \(W\) be a Wigner matrix satisfying Assumption~\ref{ass:entr}, let \(z\in \C \setminus \R \), and let \(G(z)=(W-z)^{-1}\). We use the notation \(\eta:=|\Im z|\), \(\rho=\rho_{\mathrm{sc}}(z)\), \(m=m_{\mathrm{sc}}(z)\). Then for $N\eta\rho\ge N^{\epsilon}$ and  for any deterministic matrix \(A\), with \(\braket{A}=0\) and \(\norm{A}\lesssim 1\), we have 
  \begin{equation}\label{eq:avll1G}
  |\braket{(G-m)A}| =  |\braket{GA}|\prec \frac{\sqrt{\rho}}{N\sqrt{\eta}}.
  \end{equation}
\end{theorem}

We prove a similar drastic improvement owing to the traceless observables for 
local laws involving two resolvents, like $\braket{GAGA}$,  as well as local laws involving a resolvent and its
transpose, $\braket{GG^t}$. The isotropic laws are also improved in this case.
The precise statements will be given  in Remark~\ref{rmk:2G} in Section~\ref{sec que}.
We close the current section by a remark indicating the optimality of the new local law~\eqref{eq:avll1G}.

\begin{remark}
  The local law for \(\braket{GA}\) in~\eqref{eq:avll1G} is optimal for \(G,G^\ast\), as well as \(\Im G\). Indeed a simple calculation from~\cite[Theorem~\ref{funcCLT-CLT theorem}]{2012.13218} shows, for \(\Im z>0\), that 
  \begin{equation}\label{eq:optimality}
    \E \abs{\braket{\Im G A}}^2 \approx \frac{\braket{AA^\ast}}{2N^2}\Bigl(\frac{\Im m}{\eta}-\Re \partial_z m\Bigr) 
    \sim \frac{\Im m}{N^2\eta}.
  \end{equation}
  In fact, in our companion paper we prove that
   \(\braket{\Im GA}\) is asymptotically Gaussian with zero expectation and variance given in~\eqref{eq:optimality} (see~\cite[Eq.~\eqref{funcCLT-eq:variancesres}]{2012.13218}). This variance is much smaller 
   than the one without traceless observable,  $\Var \braket{\Im G}\sim (N\eta)^{-2}$ (see~\cite{MR3678478}).
\end{remark}

\section{Quantum unique ergodicity: Proof of Theorems~\ref{theorem overlap}--\ref{theo:transpoverlap}}\label{sec que}
For integers \(J\in\N\) and self-adjoint matrices \(B=B^\ast\) we introduce the \(J\)-averaged observables
\begin{align}\label{eq:deflam}
  \Xi_J^B&:= \biggl(\max_{i_0,j_0}\frac{N}{(2J)^2}\sum_{\abs{i-i_0}<J}\sum_{\abs{j-j_0}<J} \abs{\braket{\bm u_i, B \bm u_j}}^2 \biggr)^{1/2},\\\label{eq:defpi}
  \bar\Xi_J^B&:= \biggl(\max_{i_0,j_0}\frac{N}{(2J)^2}\sum_{\abs{i-i_0}<J}\sum_{\abs{j-j_0}<J} \abs{\braket{\bm u_i, B \ov{\bm u_j}}}^2 \biggr)^{1/2}.
\end{align}

The following theorem shows that both averaged overlaps \(\Xi_J^B,\bar\Xi_J^B\)are essentially bounded 
if $B$ is traceless, and that \(\bar\Xi_J^I\) is essentially bounded for \(\abs{\sigma}<1\) and for choosing the identity matrix \(B=I\).

\begin{theorem}\label{theo:addchan}
Fix $\epsilon>0$, let $J\ge N^\epsilon$, and let $A=A^*$ be a deterministic Hermitian matrix\footnote{We use the notational convention that the letter $A$ denotes traceless matrices, while $B$ denotes arbitrary 
matrices.}
 such that $\braket{A}=0$, $\norm{A}\lesssim 1$, then it holds
\begin{equation}
\label{eq:boundlamnewtheo}
\Lambda_J^A\prec 1 \qquad \text{for}\qquad \Lambda_J^A:=\Xi_J^A+\ov\Xi_J^A.
\end{equation}
Similarly, for any $|\sigma|<1$ it holds
\begin{equation}
\label{eq:boundpinewtheo}
\Pi_J\prec 1  \qquad\text{for}\qquad \Pi_J:=\bar \Xi_J^I,
\end{equation}
    where the error is uniform in \(|\sigma|\le 1-\epsilon'\), for any fixed \(\epsilon'>0\).
\end{theorem}
Hence, up to a $J$-averaging, 
we have the \emph{asymptotic orthogonality} of \(\bm u_i\) and \(A\bm u_j\), \(A\ov{\bm u_j}\) 
for any \(i,j\) and for any  traceless $A$.
  Similarly, for \(\abs{\sigma}<1\) we have the \emph{asymptotic orthogonality} of \(\bm u_i\) and \(\ov{\bm u_j}\).
Note that in case \(\sigma=\pm1\) the bound~\eqref{eq:boundpinewtheo}
 does not hold since then \(\abs{\Pi_J}\gtrsim \sqrt{N/J}\), as easily seen.
Using Theorem~\ref{theo:addchan} we immediately conclude Theorems~\ref{theorem overlap}--\ref{theo:transpoverlap}
by removing  $J$-averaging with a small $J$.

\begin{proof}[Proof of Theorems~\ref{theorem overlap}--\ref{theo:transpoverlap}]
For the proof of Theorem~\ref{theorem overlap} we may assume by linearity that \(A\) is traceless and self-adjoint.
For any \(i,j\in [N]\), by~\eqref{eq:boundlamnewtheo}, we have that
  \begin{equation}\label{eq:eb}
    |\braket{\bm u_i, A \bm u_j}|^2+|\braket{\bm u_i, A \overline{\bm u_j}}|^2\le \frac{J^2}{N} (\Lambda_J^A)^2\prec \frac{J^2}{N}.
  \end{equation}
  Since \(J=N^\epsilon\) with \(\epsilon>0\) arbitrary small, together with the definition of \(\prec\) in Definition~\ref{def:stochDom}, the bound in~\eqref{eq:eb} implies that \(  |\braket{\bm u_i, A \bm u_j}|^2+|\braket{\bm u_i, A \overline{\bm u_j}}|^2\prec N^{-1}\), concluding the proof of Theorem~\ref{theorem overlap}. The proof of Theorem~\ref{theo:transpoverlap} is completely analogous and so omitted.
\end{proof}

As a first step towards the proof of Theorem~\ref{theo:addchan},
we first show that \(\Xi_J^B,\bar\Xi_J^B\) are comparable with \(\braket{\Im G_1 B \Im G_2 B},\braket{\Im G_1 B \Im G_2^t B}\) for suitably chosen spectral parameters in 
the resolvents \(G_i=G(z_i)\).  For any 
 \(J\in\N\) and \(E\in[-2,2]\) we define \(z=z(E,J)=E+\ii\eta(E,J)\in\HC\) implicitly via the equation
  \(N\eta(E,J)\rho(z(E,J))=J\). Note that this equation has a unique solution $\eta(E,J)>0$ 
   since the function \(\eta\to \eta\Im m(E+\ii\eta)\) is strictly increasing from \(0\) to \(1\). 
   The following simple lemma will be proven at the end of this section.
\begin{lemma}\label{lem:compa}
Let \(\epsilon>0\),  fix  some $J\ge N^\epsilon$ and let \(B=B^\ast\) be an arbitrary deterministic self-adjoint matrix. For $E_1,E_2\in [-2,2]$ let $z_i=z(E_i,J)$ and  set \(G_i=G(z_i),\rho_i=\rho(z_i)\), then we have
\begin{equation}
\label{eq:proportional}
(\Xi_J^B)^2\lesssim \sup_{E_1,E_2}\frac{\braket{\Im G_1B \Im G_2 B}}{\rho_1\rho_2}\lesssim (\Xi_J^B)^2, \quad (\bar\Xi_J^B)^2\lesssim \sup_{E_1,E_2}\frac{\braket{\Im G_1B \Im G_2^t B}}{\rho_1\rho_2}\lesssim (\bar \Xi_J^B)^2,
\end{equation}
with very high probability.
\end{lemma}

As the main  inputs for Theorem~\ref{theo:addchan},
we now state the various  local laws and bounds for products of \(G\)'s, their transposes
 and deterministic matrices in the following Propositions~\ref{pro:lambll1}--\ref{pro:lambll}. 
  These bounds
 still involve  the key control quantities $\Lambda$ and $\Pi$. 
  Using these bounds, we will  prove Theorem~\ref{theo:addchan} by a Gronwall argument on $\Lambda$ and $\Pi$
  that will immediately imply Theorem~\ref{pro:1g}.
  Finally, for completeness, we also  state a few representative local laws involving two resolvents in Remark~\ref{rmk:2G}.
  The key technical 
  Propositions~\ref{pro:lambll1}--\ref{pro:lambll}  will  be 
  proven in Section~\ref{sec:loclaws}.

\begin{proposition}\label{pro:lambll1}
Let \(A=A^\ast\) be a deterministic matrix with \(\braket{A}=0\). Fix $\epsilon>0$ and consider $z\in \C\setminus \R$ such that $L:=N\eta\rho\ge N^\epsilon$. Then for \(G=G(z)\) we have that
\begin{equation}\label{eq:singGA}
    |\braket{GA}|\prec \frac{\sqrt{\rho}\Lambda_+^A}{N\sqrt{\eta}},
  \end{equation}
  with \(\Lambda_+^A:=\Lambda_L^A+\|A\|\).
\end{proposition}

\begin{proposition}\label{pro:lambll} 
  Let \(A=A^\ast,A'=(A')^\ast\) be a deterministic matrix with \(\braket{A}=0=\braket{A'}\). Fix $\epsilon>0$, let \(W\) be a Wigner matrix satisfying Assumption~\ref{ass:entr}, let \( z_1,z_2\in \C \setminus \R \), and let \(G_i=G(z_i)\), for \(i\in \set{1,2}\). We use the notations \(\eta_i:=|\Im z_i|\), \(\rho_i=\rho_{\mathrm{sc}}(z_i)\), 
  \(m_i=m_{\mathrm{sc}}(z_i)\), and set \(L:=N\min_i(\eta_i\rho_i)\), \(\eta_*:=\eta_1\wedge\eta_2\) and \(\rho^*=\rho_1\vee \rho_2\). 
  Then for $L\ge N^\epsilon$  and setting \(\Lambda_+^A=\Lambda_L^A+\|A\|\), \(\Pi_+:=\Pi_L+1\), 
   we have the averaged local laws 
  \begin{align}
    &|\braket{G_1G_2^{(t)}A}|\prec\frac{\sqrt{\rho^*}\Lambda_+^A}{N\eta_*^{3/2}},\quad |\braket{\Im G_1 A G_2^{(t)}}|\prec \frac{\rho_1\Lambda_+^A}{L\sqrt{\eta_*}}, \quad |\braket{\Im G_1 A \Im G_2^{(t)}}|\prec \frac{\rho_1\rho_2\Lambda_+^A}{L\sqrt{\eta_*}},\label{eq:boundimGGA} \\
    &\braket{G_1AG_2^{(t)}A'}=m_1m_2\braket{AA'}+\landauOprec*{\Lambda_+^A\Lambda_+^{A'}\sqrt{\frac{\rho^\ast}{N\eta_\ast}}},\label{eq:llGAGAtranspose} \\
    &\braket{\Im G_1A\Im G_2^{(t)} A'}=\Im m_1 \Im m_2\braket{AA'}+\mathcal{O}_\prec\left(\frac{\rho_1 \rho_2\Lambda_+^A\Lambda_+^{A'}}{\sqrt{L}} \right),\label{eq:llimGAimGAtranspose} 
  \end{align}
  where \(G^{(t)}\) indicates that the bounds are valid for both choices \(G\) or \(G^t\). 
   Moreover, for any deterministic vectors \({\bm x}, {\bm y}\) such that \(\lVert {\bm x}\rVert+\lVert {\bm y}\rVert\lesssim 1\) we have the isotropic law
\begin{equation}\label{eq:isll2Gpro}
  |\braket{{\bm x}, G_1AG_2 {\bm y}}|\prec \Lambda_+^A\sqrt{\frac{\rho^*}{\eta_*}}.
\end{equation}
  Additionally, for \(|\sigma|<1\)  we have that
  \begin{equation}\label{eq:g1g2t}
    \braket{G_1G_2^t}=\frac{m_1m_2}{1-\sigma m_1 m_2}+\landauOprec*{\Pi_+^2\sqrt{\frac{\rho^\ast}{N\eta_\ast}}},
  \end{equation}
  and
  \begin{equation}\label{eq:lltrantran}
    \braket{\Im G_1\Im G_2^t}=\frac{\Im m_1 \Im m_2 (1-\sigma^2|m_1|^2|m_2|^2)}{|1-\sigma m_1m_2|^2|1-\sigma \overline{m_1}m_2|^2}+\mathcal{O}_\prec\left(\frac{\rho_1 \rho_2\Pi_+^2}{\sqrt{L}} \right),
  \end{equation}
    where the error is uniform in \(|\sigma|\le 1-\epsilon'\), for any fixed \(\epsilon'>0\).
\end{proposition}

Using Lemma~\ref{lem:compa} and Propositions~\ref{pro:lambll1}--\ref{pro:lambll} as an input, we now conclude the proof of Theorem~\ref{theo:addchan}.

\begin{proof}[Proof of Theorem~\ref{theo:addchan}] 
We start with the proof of~\eqref{eq:boundlamnewtheo}.
Choose \(J=N^\epsilon\) with a fixed arbitrary small \(\epsilon>0\), and $E_1,E_2\in [-2,2]$. Then by the definition of \(z(E_i,J)=E_i+\ii\eta(E_i,J)\) above Lemma~\ref{lem:compa} it follows that
  \[
  J=N\eta(E_1,J)\rho(z(E_1,J))=N\eta(E_2,J)\rho(z(E_2,J)),
  \]
  and thus we obtain from~\eqref{eq:llimGAimGAtranspose} 
  that 
  \begin{equation}\label{eq zJ local law}
    \frac{\abs{\braket{\Im G(z(E_1,J))A\Im G(z(E_2,J))^{(t)} A }}}{\rho(z(E_1,J))\rho(z(E_2,J)) }\prec 1+\frac{(\Lambda_+^A)^2}{J^{1/2}}.
  \end{equation}
  By a standard grid argument using the Lipschitz continuity of the resolvent we conclude that~\eqref{eq zJ local law} remains valid after taking the supremum over \(E_1,E_2\in[-2,2]\) and therefore from the lower bound in Lemma~\ref{lem:compa} we conclude 
  \begin{equation}\label{eq:boot}
    (\Lambda_J^A)^2\prec 1+\frac{(\Lambda^A_+)^2}{J^{1/2}}. 
  \end{equation}
  Finally, by~\eqref{eq:boot} it follows that $(\Lambda_J^A)^2\prec 1$, concluding the proof of~\eqref{eq:boundlamnewtheo}. The proof of~\eqref{eq:boundpinewtheo} is completely analogous to the proof of~\eqref{eq:boundlamnewtheo} above %
  using the local law~\eqref{eq:lltrantran}.
  This concludes the proof of Theorem~\ref{theo:addchan}.
\end{proof}

\begin{proof}[Proof of Theorem~\ref{pro:1g}]
This theorem immediately  follows from Propositions~\ref{pro:lambll1} 
  together with \(\Lambda_L^A\prec 1\) obtained in Theorem~\ref{theo:addchan}.
\end{proof}

\begin{remark}\label{rmk:2G}
Proposition~\ref{pro:lambll}  combined with the 
  \(\Lambda_+^A+\bm1(|\sigma|<1)\Pi_+\prec 1\) obtained in Theorem~\ref{theo:addchan}
  also provides local laws involving two resolvents 
  as counterparts of the single $G$ local law stated in Theorem~\ref{pro:1g}.
  For example, with the notations of Proposition~\ref{pro:lambll}, we have
    \begin{align}
    \braket{G_1AG_2^{(t)}A'}&=m_1m_2\braket{AA'}+\landauOprec*{\sqrt{\frac{\rho^*}{N\eta_\ast}}},
   \qquad  |\braket{{\bm x}, G_1AG_2 {\bm y}}|\prec \sqrt{\frac{\rho^*}{\eta_*}},
   \label{eq:avll2G2}
  \end{align}
  and for \(|\sigma|< 1\) we also have 
    \begin{equation}\label{eq:avll2G3}
    \braket{G_1G_2^t}=\frac{m_1m_2}{1-\sigma m_1m_2}+\landauOprec*{\sqrt{\frac{\rho^*}{N\eta_\ast}}}.
  \end{equation}
  We stated only the local laws for two resolvents where the asymptotic orthogonality mechanism is detected, i.e.\ if a traceless deterministic matrix is present or if \(G\) and \(G^t\) appear next  to each other
  and \(|\sigma|<1\). Note that when both mechanisms are simultaneously present, as in the terms with transposes in~\eqref{eq:boundimGGA}, one may gain from both effects  simultaneously, but we refrain from doing this here. 
  
  For comparison, we also list some local laws without exploiting this mechanism:
    \begin{equation}\label{eq:GG}
    \braket{G_1G_2}=\frac{m_1m_2}{1- m_1m_2}+\mathcal{O}_\prec\left(\frac{1}{N\eta_1\eta_2}\right), \qquad
     \braket{{\bm x}, G_1G_2 {\bm y}}=\frac{m_1m_2\braket{{\bm x},  {\bm y}}}{1-m_1m_2}+\mathcal{O}_\prec\left(\frac{\sqrt{\rho^*}}{\sqrt{N\eta_*}\eta^*}\right),
  \end{equation}
  and  for any $|\sigma|\le 1$
    \begin{equation}\label{eq:GGt}
    \braket{G_1G_2^t}=\frac{m_1m_2}{1-\sigma m_1m_2}+\mathcal{O}_\prec\left(\frac{1}{N\eta_1\eta_2}\right);
  \end{equation}
 these relations  are proven in our companion paper~\cite{2012.13218} using  Theorem~\ref{chain G underline theorem}   of the present paper.
In the most interesting critical case $z:=z_1=\bar z_2$ with $\eta=|\Im z|\ll 1$,  the leading term in~\eqref{eq:GG} is of order $1/\eta$
with a large error $1/N\eta^2$, 
while the leading term in~\eqref{eq:avll2G2} is bounded (even zero in the isotropic case) with a negligble error term.
The leading terms in~\eqref{eq:avll2G3} and~\eqref{eq:GGt} are of course the same, but the error
term in~\eqref{eq:GGt} is much bigger since it
ignores the asymptotic orthogonality mechanism.

Note that using the decomposition \(B = \braket{B} + \mathring{B}\), where \(\mathring{B}\) is the traceless part of \(B\), a combination of~\eqref{eq:avll2G2}--\eqref{eq:GGt} trivially  gives   local laws for any product of the form \(GBG^{(t)}B'\) for arbitrary deterministic matrices \(B\) and \(B'\).
\end{remark}

  Finally, we close this section by proving Lemma~\ref{lem:compa}.
  
\begin{proof}[Proof of Lemma~\ref{lem:compa}]
We  only consider $\braket{\Im G_1 B \Im G_2 B}$, the proof of the bounds for $\braket{\Im G_1 B \Im G_2^t B}$ is completely analogous and so omitted.

We recall that by the averaged local law for single resolvent in~\eqref{eq:avll1G}  it follows the rigidity of the eigenvalues (see e.g.~\cite[Theorem 7.6]{MR3068390} or~\cite{MR2871147}):
\begin{equation}\label{eq:rig}
  |\lambda_i-\gamma_i|\prec \frac{1}{N^{2/3}\widehat{i}^{1/3}},
\end{equation}
where \(\widehat{i}:=i\wedge (N+1-i)\), and \(\gamma_i\) are the classical eigenvalue locations (\emph{quantiles}) defined by
\begin{equation}\label{eq:quantin}
  \int_{-\infty}^{\gamma_i} \rho(x)\, \dif x=\frac{i}{N}, \qquad i\in [N],
\end{equation}
where we recall \(\rho(x)=\rho_\mathrm{sc}(x)=(2\pi)^{-1}\sqrt{(4-x^2)_+}\).

For $E_1,E_2\in [-2,2]$, we recall that by the definition of \(z(E_i,J)=E_i+\ii\eta(E_i,J)\) above Lemma~\ref{lem:compa} it follows that \(J=N\eta(E_1,J)\rho(z(E_1,J))=N\eta(E_2,J)\rho(z(E_2,J))\) and thus, together with~\eqref{eq:quantin} we conclude that there is constant \(C\) such that for any \(a, a_0\) it holds that
\begin{equation}\label{eq:levspac1a}
  |\gamma_a-\gamma_{a_0}|\le\eta( \gamma_{a_0}) \Rightarrow |a-a_0|\le CJ, \qquad |a-a_0|\le J \Rightarrow
  |\gamma_a-\gamma_{a_0}|\le C\eta( \gamma_{a_0}).
\end{equation}
With a slight abuse of notation we will write this relation as
\begin{equation}\label{eq:levspaca}
  |\gamma_a-\gamma_{a_0}|\lesssim \eta( \gamma_{a_0}) \Leftrightarrow |a-a_0|\lesssim J,
\end{equation}
since the implicit constants in the \(\lesssim\) relation  will be  irrelevant for our analysis. With the short-hand notations $\Xi=\Xi_J^B$, \(z_i=z(E_i,J)\), $G_i=G(z_i)$, $\eta_i=\eta(E_i,J)$, and \(\rho_i=\rho(z_i)\), then by~\eqref{eq:rig}, and writing \(\braket{\Im G_1 B \Im G_2 B}\) in spectral decomposition
\begin{equation}\label{eq:speccdec}
  \braket{\Im G_1 B \Im G_2 B}=\sum_{ab} \frac{R_{ab}S_{ab}}{N},\quad R_{ab}:=\frac{\eta_1\eta_2}{|\lambda_a-z_1|^2|\lambda_b-z_2|^2},\quad S_{ab}:=|\braket{\bm u_a,B \bm u_b}|^2,
\end{equation}
it follows that
\begin{equation}\label{eq:resba}
  \Xi^2\lesssim \sup_{E_1, E_2\in [-2,2]}  \frac{\braket{\Im G_1 B \Im G_2 B}}{\rho_1\rho_2}\lesssim \Xi^2,
\end{equation}
with very high probability on the set where the rigidity bound~\eqref{eq:rig} holds. The lower bound in~\eqref{eq:resba} is trivial by choosing \(E_1=\gamma_{a_0}\) and \(E_2=\gamma_{b_0}\). To prove the upper bound in~\eqref{eq:resba} we use the local averaging formula 
\begin{equation}\label{tricka}
  \sum_{ab} R_{ab} S_{ab} \sim \sum_{a_0b_0} \frac{1}{(2J)^2} \sum_{\substack{\abs{a-a_0}< J, \\ \abs{b-b_0}< J}}  R_{ab} S_{ab} \sim \sum_{a_0b_0}  R_{a_0b_0} \frac{1}{(2J)^2} \sum_{\substack{\abs{a-a_0}< J, \\ \abs{b-b_0}< J}}   S_{ij}
\end{equation}
for general non-negative \(R_{ab},S_{ab}\) such that \(R_{ab}\sim R_{a_0b_0}\) whenever \(\abs{a-a_0}\vee\abs{b-b_0}\le J\) which is applicable for \(R_{ab}\) in~\eqref{eq:speccdec} as a consequence of the rigidity bound in~\eqref{eq:rig}, the relation in~\eqref{eq:levspaca}, and the fact that we can choose the \(\xi>0\) so that \(N^\xi\), coming from the rigidity high probability bound, is much smaller than \(J\ge N^\epsilon\). Finally we note that \(N^{-2}\sum_{ab}R_{ab}=\braket{\Im G_1}\braket{\Im G_2} \sim \rho_1\rho_2\) by \(|\braket{\Im G_i-\Im m_i}|\prec \rho_i\) from~\eqref{eq:avll1G}, concluding the proof of the upper bound in~\eqref{eq:resba}. 
\end{proof}

\section{Local laws for one and two resolvents}\label{sec:loclaws}
In this section we prove the local laws in Propositions~\ref{pro:lambll1}--\ref{pro:lambll}. By the self consistent equation for \(m\) in~\eqref{eq:mdesc}, and by \(G(W-z)=I\), we have 
\begin{equation}\label{eq:Gnounder}
  G=m-mWG - m \braket{G} G + m\braket{G-m}G.
\end{equation}
For any given functions \(f,g\) of the Wigner matrix \(W\) we define the renormalisation of the product \(g(W)Wf(W)\) (denoted by underline) as follows:
\begin{equation}\label{eq:defunder}
  \underline{g(W)Wf(W)}:=g(W)Wf(W)-\widetilde{\E }g(W)\widetilde{W}(\partial_{\widetilde{W}}f)(W)-\widetilde{\E }(\partial_{\widetilde{W}}g)(W)\widetilde{W}f(W),
\end{equation}
where \(\partial_{\widetilde W} f(W)\) denotes the directional derivative of the function \(f\) in the direction \(\widetilde{W}\) at the point \(W\), and \(\widetilde{W}\) is an independent copy of \(W\). The definition is chosen such that it subtracts the second order term in the cumulant expansion, in particular if all entries of \(W\) were  Gaussian then we  had
\(\E  \underline{g(W)Wf(W)}=0\). Note that the definition~\eqref{eq:defunder} only makes sense if it is clear to which \(W\) the underline refers, i.e.\ it would be ambiguous if \(f(W)=W\). In our applications, however, each underlined term contains exactly a single \(W\) factor, and hence such ambiguities will not arise. As a special case we have that
\begin{equation}\label{eq:newunder}
  \underline{WG}=WG+\braket{G}G  + \sigma \frac{G^t G}{N} + \frac{\widetilde{w_2}}{N} \diag(G)G,
\end{equation}
recalling the notation $\widetilde{w_2}=w_2-1-\sigma$ from Assumption~\ref{ass:entr}. Then by~\eqref{eq:Gnounder} and and~\eqref{eq:newunder}, it follows that
\begin{equation}\label{eq:G}
  G=m-m\underline{WG}+\frac{m\sigma}{N}G^t G+\frac{\widetilde{w_2}}{N}\diag(G)G+m\braket{G-m}G.
\end{equation}

From~\eqref{eq:G} one can already see that in order to get a local law for \(G\) it is essential to estimate the underlined term \(\underline{WG}\) in averaged and isotropic sense. In order to prove Proposition~\ref{pro:lambll} we need bounds for underlined terms involving not only one \(G\) but also two \(G\)'s (see e.g.~\eqref{eq:2G} below). We now state the bound for these terms and for longer products of resolvents and deterministic matrices both in an averaged and in an isotropic sense, since the proof for products with more than two resolvents is very similar to the cases we need.

For \(l\in\N\) we consider renormalised alternating products of resolvents \(G_1,\ldots,G_l\) and deterministic matrices \(B_1,\ldots,B_l\) in averaged and isotropic form, 
\begin{equation}\label{general prod resolvents}
  \braket{\un{WG_1B_1G_2B_2\cdots G_l B_l}}, \quad \braket{\vx,\un{W G_1 B_1 G_2\cdots B_{l-1}G_l}\vy}.
\end{equation}
Each resolvent \(G_k\) is evaluated at a (potentially) different spectral parameter \(z_k\in\C\setminus\R\) and other than \(G_k=G(z_k)\) we allow each resolvent to be transposed and/or being the imaginary part, i.e.\ 
\begin{equation}\label{G list} G_k\in\set{G(z_k),G(z_k)^t,\Im G(z_k),(\Im G(z_k))^t}.\end{equation}
Note that adjoints of resolvents can be included in the products by conjugating the spectral parameter since \(G(z)^\ast=G(\ov z)\). For a given product of the form~\eqref{general prod resolvents} we consider three sets \(\mathfrak{i},\mathfrak a,\mathfrak t\) of indices, recording special structural properties of \(G_k,B_k\).  By definition, the set \(\mathfrak{i}\subset[l]\) collects those indices \(k\in [l]\) for which \(G_k\in\set{\Im G(z_k),(\Im G(z_k))^t}\).  For the choice of the sets \(\mathfrak a,\mathfrak t\) we allow a certain freedom described in the theorem.

\begin{theorem}\label{chain G underline theorem}
Fix $\epsilon>0$,  let \(l\in\N\), \(z_1,\dots, z_l\in \C\setminus\R\) and for \(k\in[l]\) let \(G_k\) be as in~\eqref{G list} and \(B_k\) be deterministic \(N\times N\) matrices, and \(\vx,\vy\) be deterministic vectors with bounded norms \(\norm{B_k}\lesssim1\), \(\norm{\vx}+\norm{\vy}\lesssim1\). Set 
  \begin{equation}
    L := N\min_{k} (\eta_k\rho_k),\quad \rho^\ast := \max_{k}\rho_k,\quad \eta_\ast:=\min_k \eta_k, 
  \end{equation}
  with \(\eta_k:=\abs{\Im z_k}\), \(\rho_k:=\rho(z_k)=\abs{\Im m(z_k)}/\pi\) and assume \(L\ge N^\epsilon\) and \(\eta_\ast\lesssim 1\). Let \(\mathfrak a,\mathfrak t\) denote disjoint sets of indices, \(\mathfrak a\cap \mathfrak t=\emptyset\), such that for each \(k\in\mathfrak{a}\) we have \(\braket{B_k}=0\), and for each \(k\in\mathfrak{t}\) exactly one of \(G_k,G_{k+1}\) is transposed, where in the averaged case and \(k=l\) it is understood that \(G_{l+1}=G_1\). Recall the notations $\Pi_+:=\Pi_L+1$, \(\Lambda_{+}^{B}:=\Lambda_L^B+\norm{B}\). Then with \(a:=\abs{\mathfrak a}, t:=\abs{\mathfrak t}\), we have the following bounds: 
  \begin{itemize}
    \item[(av1)] For \(\mathfrak{a}=\mathfrak{t}=\emptyset\) we have
    \begin{equation}\label{eq:under B}
      \begin{split}
        \abs{\braket{\un{WG_1B_1G_2B_2\cdots G_l B_l}}} &\prec \frac{\rho^{\ast}}{ N \eta_\ast^{l} }.
      \end{split}   
    \end{equation}
    \item[(av2)] For \(\mathfrak a,\mathfrak t\subset[l]\), \(\abs{\mathfrak a\cup \mathfrak t}\ge 1\) we have the bound
    \begin{equation}\label{eq:under} 
      \begin{split}
        \abs{\braket{\un{WG_1B_1G_2B_2\cdots G_l B_l}}} &\prec \frac{(\sqrt{N}\eta_\ast)^{a+t}}{N\eta_\ast^l} \sqrt{\frac{\rho^{\ast}}{N\eta_\ast}}  \Pi_+^t\prod_{k\in\mathfrak{a}}\Lambda_+^{B_k}.
      \end{split} 
    \end{equation} 
    \item[(iso)] For \(\mathfrak a,\mathfrak t\subset[l-1]\) and for any \(0\le j<l\) we have the bound 
    \begin{equation}\label{eq:underIso}
        \abs{\braket{\vx,\un{G_1 B_1\cdots G_{j} B_j W G_{j+1}B_{j+1}\cdots B_{l-1}G_l }\vy}}  \prec \frac{(\sqrt{N}\eta_\ast)^{a+t}}{\eta_\ast^{l-1}} \sqrt{\frac{\rho^{\ast}}{N\eta_\ast}}   \Pi_+^t\prod_{k\in\mathfrak{a}}\Lambda_+^{B_k},
    \end{equation}
    where the \(j=0\) case is understood as \(\braket{\vx,\un{W G_1 B_1\cdots B_{l-1}G_l}\vy}\). 
  \end{itemize}
  In case \(\prod_{k\in\mathfrak i} \rho_k\lesssim (\rho^\ast)^{b+1}\), the bounds~\eqref{eq:under B}--\eqref{eq:underIso} remain valid if the rhs.\ are multiplied by the factor \((\rho^\ast)^{-b-1}\prod_{k\in\mathfrak i} \rho_k\), where \(b:=l\) in case of~\eqref{eq:under B}, \(b:=l-a-t\) in case of~\eqref{eq:under}, and \(b:=l-a-t-1\) in case of~\eqref{eq:underIso}. Moreover, for any \(\eta_\ast\ge 1\) we have the bounds 
  \begin{equation}\label{eq large eta bounds}
    \begin{split}
      \abs{\braket{\un{WG_1B_1G_2B_2\cdots G_l B_l}}} &\prec \frac{1}{N\eta_\ast^l},\\  
      \abs{\braket{\vx,\un{G_1 B_1\cdots G_{j} B_j W G_{j+1}B_{j+1}\cdots B_{l-1}G_l }\vy}} &\prec \frac{1}{N^{1/2}\eta_\ast^l}.
    \end{split}
  \end{equation}
\end{theorem}

\begin{remark}[Asymptotic orthogonality effect]
  The main result of Theorem~\ref{chain G underline theorem} are~\eqref{eq:under} and its isotropic counterpart~\eqref{eq:underIso}. The essential part is the factor \((\sqrt{N}\eta_\ast)^{a+t}\) in~\eqref{eq:under} since the additional factors \(\Pi_+\) and \(\Lambda_+\) are a posteriori  shown to be essentially \(\landauO{1}\), c.f.\ Theorem~\ref{theo:addchan}. Compared with the robust bound~\eqref{eq:under B} in the relevant small \(\eta_\ast\) regime the bound~\eqref{eq:under} represents an improvement of \(\sqrt{N} \eta_\ast\) for each occurrence when the asymptotic orthogonality
  can be exploited, either due to a traceless matrix $B$ or to a switch between
  a resolvent  and its transpose. In addition, compared to the robust averaged bound~\eqref{eq:under B} there is a further improvement of \(\sqrt{\rho^\ast/N\eta_\ast}\) in~\eqref{eq:under} if at least one orthogonality effect is exploited, enabling the optimal \(GA\) local law in Theorem~\ref{pro:1g}. We note that in case when \(\sqrt{N}\eta_\ast\gg 1\) the robust bounds~\eqref{eq:under B} and~\eqref{eq:underIso} with \(a+t=0\) are always available also in the presence of traceless deterministic matrices and alternating \(G,G^t\) simply by choosing the sets \(\mathfrak a,\mathfrak t\) to be empty.
\end{remark}
\begin{remark}[Alternative renormalisation]\label{remark alt underline}
    In~\eqref{eq:defunder} we defined the renormalisation with respect to an independent copy of \(W\) while in some previous papers~\cite{1912.04100} the same notation was used to denote the renormalisation with respect to a suitable reference ensemble (e.g.\ the GUE-ensemble in the present paper or the complex Ginibre ensemble in case of~\cite{1912.04100}). However, mostly these two possible definitions only differ in some sub-leading terms. For example, denoting the renormalisation with respect to an independent GUE-matrix by 
    \[
        \un{Wf(W)}_\mathrm{GUE}:=Wf(W)-\wt\E_\mathrm{GUE} \wt W (\partial_{\wt W} f)\] 
    we have trivially
    \[\braket{\un{WG}}-\braket{\un{WG}_\mathrm{GUE}}= \sigma\frac{\braket{G^t G}}{N} + \wt w_2 \frac{\braket{\diag(G)G}}{N} = \landauOprec*{\frac{\rho}{N\eta}}.\]
    The difference between the two renormalisations becomes relevant in Theorem~\ref{chain G underline theorem} only whenever at least one transposed resolvent occurs since then for example
    \[\braket{\un{WG^t}}-\braket{\un{WG^t}_\mathrm{GUE}}= \sigma\braket{G}^2 + \wt w_2 \frac{\braket{\diag(G)G}}{N} \sim 1.\]
    However, this is the only relevant case and the statement of Theorem~\ref{chain G underline theorem} holds true verbatim if \(\un{Wf(W)}\) is replaced by \(\un{Wf(W)}_\mathrm{GUE}\) in case no resolvents are transposed, c.f.\ Remark~\ref{remark alt underline expl} in Section~\ref{sec cum exp}. 
\end{remark}

Using the bounds for the underlined terms in Theorem~\ref{chain G underline theorem} we conclude the proof of  Proposition~\ref{pro:lambll1}--\ref{pro:lambll}. We start with the proof of the local law for \(\braket{GA}\) and then we prove the local laws and bounds for two \(G\)'s.

\begin{proof}[Proof of Proposition~\ref{pro:lambll1}]
    
    Using the equation for \(G\) in~\eqref{eq:G}, we start writing the equation for \(GA\):
    \begin{equation}\label{eq:1GA}
      GA=mA-m\underline{WG}A+m\braket{G-m}GA+\frac{m\sigma}{N}G^t GA+\frac{\widetilde{w_2}}{N}\diag(G)GA,
    \end{equation}
    where we recall the definition of \(\underline{WG}\) in~\eqref{eq:defunder}. Then, taking the average in~\eqref{eq:1GA}, using that \(\braket{A}=0\) and that \(|\braket{G-m}|\prec (N\eta)^{-1}\) by the first bound in~\eqref{eq:avll1G}, we conclude
    \begin{equation*}
      \left[1+\mathcal{O}_\prec\left(\frac{1}{N\eta}\right)\right]\braket{GA}=-m\braket{\underline{WG}A}+\frac{m\sigma}{N}\braket{G^t GA}+\frac{\widetilde{w_2}}{N}\braket{\diag(G)GA}.
    \end{equation*}
    Next we notice that 
    \begin{equation}\label{eq:aux1}
      \frac{1}{N}|\braket{G^t GA}|\le \frac{1}{N}\braket{|G|}^{1/2}\braket{GA|G^t|AG^*}^{1/2}\prec \frac{\Lambda_+\sqrt{\rho}}{N\sqrt{\eta}},
    \end{equation}
    where in the last inequality we used Lemma~\ref{degree two lemma}, and the notation $\Lambda_+:=\Lambda_L^A+\|A\|$. Additionally, we also have that
    \begin{equation}\label{eq:aux2}
      \frac{1}{N}|\braket{\diag(G)GA}|=\left|\frac{1}{N^2}\sum_i G_{ii} (GA)_{ii}\right|\prec \frac{1}{N},
    \end{equation}
    where we used that \(|G_{ii}|+|(GA)_{ii}|\prec 1\) by~\eqref{eq:oldlocal}. Combining~\eqref{eq:aux1}--\eqref{eq:aux2} we finally conclude that
    \begin{equation}\label{eq:fin1ga}
      \braket{GA}=-m\braket{\underline{WG}A}+\mathcal{O}_\prec\left(\frac{\Lambda_+\sqrt{\rho}}{N\sqrt{\eta}}\right)=\mathcal{O}_\prec\left(\frac{\Lambda_+\sqrt{\rho}}{N\sqrt{\eta}}\right),
    \end{equation}
    where we used \(|\braket{\underline{WG}A}|\prec \Lambda_+\rho^{1/2} N^{-1}\eta^{-1/2}\) by~\eqref{eq:under}. This concludes the proof of~\eqref{eq:singGA}.
  \end{proof}
  
  Next, using the local law for single resolvent proven above, we proceed with the proof of the bounds for for products of two resolvents and deterministic traceless matrices.

  \begin{proof}[Proof of Proposition~\ref{pro:lambll}]

  We start writing the equation for generic products of two \(G\)'s \(G_1B_1G_2B_2\), where \(G_i=(W-z_i)^{-1}\) and \(B_1,B_2\) are deterministic matrices. Using the equation~\eqref{eq:1GA} for \(G_1B_1\), writing \(G_2=m_2+(G_2-m_2)\), we obtain
  \begin{equation}\label{eq:2G}
    \begin{split}
      G_1B_1G_2B_2&=m_1m_2 B_1B_2+m_1B_1(G_2-m_2)B_2-m_1\underline{WG_1B_1G_2}B_2 \\
      &\quad+m_1\braket{G_1-m_1}G_1B_1G_2B_2 + m_1\braket{G_1B_1G_2}G_2B_2 \\
      &\quad+\frac{m_1\sigma}{N}G_1^t G_1B_1G_2B_2+\frac{m_1\sigma}{N} (G_1B_1G_2)^t G_2B_1 \\
      &\quad+ \frac{m_1\widetilde{w_2}}{N}\diag(G_1)G_1B_1G_2B_2+\frac{m_1 \widetilde{w_2}}{N}\diag(G_1B_1G_2)G_2B_2,
    \end{split}
  \end{equation}
  where we used that
  \begin{equation}\label{eq:doubunder}
    \underline{WG_1B_1G_2}=\underline{WG_1}B_1G_2+\braket{G_1B_1G_2}G_2+\frac{\sigma}{N}(G_1B_1G_2)^t G_2+\frac{\widetilde{w_2}}{N}\diag(G_1B_1G_2)G_2,
  \end{equation}
  with \(\underline{WG_1}\) from~\eqref{eq:newunder}. The identity in~\eqref{eq:doubunder} follows by the definition of underline in~\eqref{eq:defunder}.
  
  \begin{remark}\label{rmk Lambda plus}
For notational simplicity, throughout the proof of Proposition~\ref{pro:lambll} we use the notations
\[
\Lambda_+:= \Lambda_+^A\vee\Lambda_+^{A'},
\]
rather than distinguishing the different $\Lambda$'s. However, the proof naturally yields in fact a factor of \(\Lambda_+^A\) for each traceless \(A\) giving the bounds in Propositions~\ref{pro:lambll}.
  \end{remark}

  \begin{proof}[Proof of the bounds~\eqref{eq:boundimGGA}]
    We focus only on the proof of the bound for \(\braket{G_1G_2A}\), the bounds for \(\braket{G_1^t G_2A}\), \(\braket{\Im G_1 A G_2}\), \(\braket{\Im G_1^t A G_2}\), \(\braket{\Im G_1 A \Im G_2}\), and \(\braket{\Im G_1^t A \Im G_2}\) are completely analogous and so omitted, modulo the bound for the underlined term. In particular, the bound in~\eqref{eq:specunder} has to be replaced by
    \[
    |\braket{\underline{WG_2\Im G_1}A}|\prec \frac{\rho_1\Lambda_+}{\sqrt{NK}\eta_*},\quad |\braket{\underline{W\Im G_2 \Im G_1}A}|\prec \frac{\rho_1\rho_2\Lambda_+}{\sqrt{NK}\eta_*},
    \]
    for \(\braket{\Im G_1 A G_2}\), \(\braket{\Im G_1^t A G_2}\) and \(\braket{\Im G_1 A \Im G_2}\), \(\braket{\Im G_1^t A \Im G_2}\), respectively, where $K:=N\eta_*\rho^*$.
    
    Choosing \(B_1=I\) and \(B_2=A\) in~\eqref{eq:2G}, with \(\braket{A}=0\), using \(|\braket{G_1-m_1}|\prec (N\eta_1)^{-1}\) we find that
    \begin{align}\nonumber
        \left[1+\mathcal{O}_\prec\left(\frac{1}{N\eta_2}\right)\right]\braket{G_1G_2A}&=-m_1\braket{\underline{WG_1G_2}A}+m_1\braket{G_2A}+m_1\braket{G_1G_2}\braket{G_2A}\\\label{eq:12ggia}
        &\quad+\frac{m_1\sigma}{N}\braket{G_1^t G_1G_2A}+\frac{m_1\sigma}{N}\braket{(G_1G_2)^t G_2A} \\\nonumber
        &\quad+\frac{m_1 \widetilde{w_2}}{N}\braket{\diag(G_1)G_1G_2A}+\frac{m_1 \widetilde{w_2}}{N}\braket{\diag(G_1G_2)G_2A}.
    \end{align}
    Then using Cauchy-Schwarz we have that
    \begin{equation}\label{eq:1need}
      \begin{split}
        \frac{1}{N}|\braket{G_1G_2AG_1^t}|&\le \frac{1}{N}\braket{G_1G_1^*}^{1/2}\braket{G_2AG_1^t(G_1^t)^*AG_2^*}^{1/2}\prec \frac{\sqrt{\rho_1}}{N\eta_1\sqrt{\eta_2}}\braket{\Im G_2 A \Im G_1^t A}^{1/2} \\
        &\quad\prec \frac{\rho_1\sqrt{\rho_2}\Lambda_+}{N\eta_1\sqrt{\eta_2}}\le \frac{\rho^*\sqrt{\rho_*}\Lambda_+}{\sqrt{NK\eta_1\eta_2}},
      \end{split}
    \end{equation}
    where we used the Ward identity, that \(\braket{\Im G_1}\prec \rho_1\), and that $K=N\eta_*\rho^*$. In penultimate inequality of~\eqref{eq:1need} we also used Lemma~\ref{degree two lemma} to prove that \(\braket{\Im G_2 A \Im G_1^t A}\prec \rho_1\rho_2\Lambda_+^2\). Using exactly the same computations we conclude the same bound for \(\braket{(G_1G_2)^t G_2A}\) as well. Now we show that the terms with a pre-factor \(\widetilde{w_2}\) are negligible. We start with
    \begin{equation}\label{eq:2need}
      \frac{1}{N}|\braket{\diag(G_1)G_1G_2A}|=\left|\frac{1}{N^2}\sum_i G_{ii} (G_1G_2A)_{ii}\right|\prec \frac{\sqrt{\rho_1\rho_2}}{N\sqrt{\eta_1\eta_2}},
    \end{equation}
    obtained using that \(|G_{ii}|\prec 1\), by the isotropic law~\eqref{eq:oldlocal}, and \(|(G_1G_2A)_{ii}|\prec \sqrt{\rho_1\rho_2/\eta_1\eta_2}\) by a simple Schwarz inequality. The bound for \(|\braket{\diag(G_1G_2)G_2A}|\) is analogous and so omitted. Combining~\eqref{eq:12ggia} with~\eqref{eq:1need}--\eqref{eq:2need}, using that \(|\braket{G_2A}|\prec \sqrt{\rho_2}N^{-1}\eta_2^{-1/2}\) by~\eqref{eq:singGA}, and dividing by the factor in the lhs.\ of~\eqref{eq:12ggia}, we conclude that
    \begin{equation}\label{eq:12gginew}
      \begin{split}
        \braket{G_1G_2A}&=-m_1\braket{\underline{WG_1G_2}A}+m_1\braket{G_1G_2}\braket{G_2A}+\mathcal{O}_\prec\left(\frac{\rho^*\Lambda_+}{\sqrt{NK}\eta_*}\right) =\mathcal{O}_\prec\left(\frac{\rho^*\Lambda_+}{\sqrt{NK}\eta_*}\right),
      \end{split}
    \end{equation}
    where to go from the first to the second line we used that \(|\braket{G_2A}|\Lambda_+ \prec \rho_2^{1/2}N^{-1}\eta_2^{-1/2}\) by~\eqref{eq:fin1ga}, that \(|\braket{G_1G_2}|\prec \sqrt{\rho_1\rho_2/(\eta_1\eta_2)}\) by a Schwarz inequality, and that 
    \begin{equation}\label{eq:specunder}
      \abs*{\braket{\underline{WG_1G_2}A}}\prec  \frac{\rho^*\Lambda_+}{\sqrt{NK}\eta_*},
    \end{equation}
    by~\eqref{eq:under}. This concludes the proof of the bound of \(|\braket{G_1G_2A}|\).
  \end{proof}

  \begin{proof}[Proof of the bound in~\eqref{eq:isll2Gpro} for \(\braket{{\bm x}, G_1AG_2{\bm y}}\)]
    Using the bound for \(\braket{G_1G_2A}\) and the estimates in Lemma~\ref{degree two lemma} below as an input, the proof of the bound
    \begin{equation}\label{eq:isbgagnew}
      |\braket{{\bm x}, G_1AG_2{\bm y}}|\prec \Lambda_+\sqrt{\frac{\rho^*}{\eta_*}},
    \end{equation}
    follows by exactly the same computations as in the proof of the bound for \(|\braket{G_1G_2A}|\).
  \end{proof}

  \begin{proof}[Proof of the local laws for two resolvents in~\eqref{eq:llGAGAtranspose} and~\eqref{eq:g1g2t}]
    
    We focus only on the proof of the local law for \(\braket{G_1AG_2A'}\), the proof of the local law for \(\braket{G_1^t AG_2A'}\) is exactly the same. The prof of the local law for \(\braket{G_1G_2^t}\) is also analogous to the proof of the local law for \(\braket{G_1AG_2A'}\) with the only difference that the multiplicative factor in the rhs.\ of~\eqref{eq:12ggiaa} has to be replaced by
    \[
    1-\sigma m_1m_2+\mathcal{O}_\prec\left(\frac{1}{N\eta_*}\right).
    \]
    This difference does not create any change since for \(|\sigma|<1\) the stability factor \(1-\sigma m_1m_2\) is bounded from below by \(1-|\sigma|\).
    
    Choosing \(B_1=A\) and \(B_2=A'\) in~\eqref{eq:2G}, with \(\braket{A}=\braket{A'}=0\), and using that \(|\braket{G_1-m_1}|\prec (N\eta_1)^{-1}\), we find that
    \begin{equation}\label{eq:12ggiaa}
      \begin{split}
        &\Bigl(1+\landauO[\Big]{\frac{1}{N\eta_2}}\Bigr)\braket{G_1AG_2A'}\\
        &\;=m_1m_2\braket{AA'}-m_1\braket{\underline{WG_1AG_2}A'} +\braket{G_1AG_2}\braket{G_2A'}+\frac{m_1\sigma}{N}\braket{G_1^t G_1AG_2A'} \\
        &\quad+\frac{m_1\sigma}{N}\braket{(G_1AG_2)^t G_2A'}+\frac{m_1\widetilde{w_2}}{N}\big[\braket{\diag(G_1)G_1AG_2A'}   +\braket{\diag(G_1AG_2)G_2A'}\big].
      \end{split}
    \end{equation}
    We start with the bound
    \begin{equation}\label{eq:bda1}
      \begin{split}
        \frac{1}{N}|\braket{G_1^t G_1AG_2A'}|&\le \frac{1}{N}\braket{G_1A|G_2|AG_1^*}^{1/2}\braket{(G_1^t)^*A'|G_2|A'(G_1)^t}^{1/2}\\
        &=\frac{1}{N\eta_1} \braket{\Im G_1A|G_2|A}^{1/2}\braket{\Im G_1^t A'|G_2|A'}^{1/2}\prec\frac{\rho_1\Lambda_+^2}{N\eta_1},
      \end{split}
    \end{equation}
    where we used a Schwarz inequality and that \(|\braket{\Im G_1A|G_2|A}|\prec  \rho_1\Lambda_+^2\) by Lemma~\ref{degree two lemma} below. Following exactly the same computations we conclude that \(|\braket{(G_1AG_2)^t G_2A'}|\prec  \Lambda_+^2\rho_2\eta_2^{-1}\). Similarly, we also bound
    \begin{equation}\label{eq:bda2}
      \frac{1}{N}|\braket{\diag(G_1)G_1AG_2A'}|=\left|\frac{1}{N^2}\sum_i(G_1)_{ii}(G_1AG_2A')_{ii}\right|\prec\frac{\sqrt{\rho_1\rho_2}}{N\sqrt{\eta_1\eta_2}},
    \end{equation}
    where we used that \(|(G_1)_{ii}|\prec 1\) and that \(|(G_1AG_2A')_{ii}|\prec \sqrt{\rho_1\rho_2/(\eta_1\eta_2)}\) by a Schwarz inequality. The bound for \(|\braket{\diag(G_1AG_2)G_2A'}|\) is completely analogous and so omitted.
    
    Combining~\eqref{eq:12ggiaa} with~\eqref{eq:bda1}--\eqref{eq:bda2}, and using that
    \[
    |\braket{G_2A'}|\prec \frac{ \sqrt{\rho_2}\Lambda_+}{N\sqrt{\eta_2}}, \qquad |\braket{G_2G_1A}|\prec \frac{ \rho^*\Lambda_+}{\sqrt{NK}\eta_*},
    \]
    by~\eqref{eq:singGA} and~\eqref{eq:boundimGGA}, respectively, we conclude that
    \begin{equation}\label{eq:fingaga}
      \begin{split}
        \braket{G_1AG_2A'}&=m_1m_2\braket{AA'}-m_1\braket{\underline{WG_1AG_2}A'}+\mathcal{O}_\prec\left(\frac{\rho^*\Lambda_+^2}{K^2}\right)\\
        &=m_1m_2\braket{AA'}+\mathcal{O}_\prec\left(\frac{\rho^*\Lambda_+^2}{\sqrt{K}}\right).
      \end{split} 
    \end{equation}
    To go from the first to the second line of~\eqref{eq:fingaga} we used that \(|\braket{\underline{WG_1AG_2}A'}|\prec  \Lambda_+^2\rho^*K^{-1/2}\) by~\eqref{eq:under}. This concludes the proof of the local law for \(\braket{G_1AG_2A'}\).
  \end{proof}
  
  In order to conclude the proof of Proposition~\ref{pro:lambll} we are only left with the averaged local laws for \(\Im G_1 A \Im G_2 A'\) and \(\Im G_1A \Im G_2^t A'\) in~\eqref{eq:llimGAimGAtranspose} and for \(\Im G_1^t\Im G_2\) in~\eqref{eq:lltrantran}.

  \begin{proof}[Proof of the local laws in~\eqref{eq:llimGAimGAtranspose}]
    
    We present only the proof of the local law for \(\braket{\Im G_1 A \Im G_2 A'}\), the proof for \(\braket{\Im G_1 A \Im G_2^t A'}\) is identical and so omitted. We start with the formula analogous to~\eqref{eq:2G} but with \(\Im G\)'s instead of \(G\)'s, generating altogether twelve terms with a \(1/N\) pre-factor.
    Ten of them can be estimated by \(\Lambda_+^2\rho_1\rho_2L^{-1}\) 
    exactly as in~\eqref{eq:bda1}--\eqref{eq:bda2}  by writing out \(2\ii \Im G_i=G_i-G_i^*\). Note that whenever  the analogue of~\eqref{eq:bda1} is used, but with  \(G_1^{(t)}G_2\) instead of \(G_1^{(t)}G_1\), we could gain the necessary
    factor \(\sqrt{\rho_1\rho_2/(\eta_1\eta_2)}\) instead of only \(\rho_1/\eta_1\)
    in the first Schwarz inequality in~\eqref{eq:bda1}. Keeping  the two special \(1/N\) terms, this gives the 
    expansion
    \begin{equation}\label{eq:bigeq}
      \begin{split} 
        &\braket{\Im G_1 A\Im G_2 A'}+\landauOprec*{\frac{\Lambda_+^2\rho_1\rho_2}{L}}\\
        &=\Im m_ 1\Im m_2 \braket{AA'}+\Im m_1 \braket{\Im (G_2-m_2) A'A}+\overline{m_1}\ov{\braket{G_1-m_1}}\braket{\Im G_1 A\Im G_2 A'} \\
        &\quad+\Im[m_1\braket{G_1-m_1}]\braket{G_1 A\Im G_2 A'} -\Im m_1 \braket{\underline{WG_1A\Im G_2} A'} -\overline{m_1}\braket{\underline{W\Im G_1A\Im G_2} A'}\\
        &\quad+\Im m_1\braket{G_1AG_2}\braket{\Im G_2A'}+\Im m_1\braket{G_1A\Im G_2}\braket{G_2^*A'}+\overline{m_1}\braket{\Im G_1 AG_2}\braket{\Im G_2A'}  \\
        &\quad+\overline{m_1}\braket{\Im G_1A\Im G_2}\braket{G_2^*A'} +\frac{\sigma\Im m_1}{N}\braket{G_1A\Im G_2A' G_1^t}+\frac{\sigma \overline{m_1}}{N}\braket{\Im[G_1^t G_1]A\Im G_2 A'}.
      \end{split}
    \end{equation}
    The two remaining \(1/N\) terms,
    where  \(\Im G_2\) is separated from \(G_1\) by \(A\)'s, are estimated as follows:
    \begin{equation}\label{eq:newin2}
      |\braket{G_1A\Im G_2A' G_1^t}|\le \braket{G_1A\Im G_2A'G_1^*}^{1/2}\braket{(G_1^t)^*A\Im G_2A'G_1^t}^{1/2}  \prec \frac{N\rho_1\rho_2\Lambda_+^2}{L},
    \end{equation}
    where in the last inequality we used the Ward identity and Lemma~\ref{degree two lemma} below. Inserting~\eqref{eq:newin2}, the local law \(|\braket{G_2-m_2}|\prec (N\eta_2)^{-1}\) and~\eqref{eq:singGA}--\eqref{eq:boundimGGA} into~\eqref{eq:bigeq} we conclude that
    \begin{equation}\label{eq:chsonn}
      \begin{split}
        \braket{\Im G_1 A\Im G_2 A'}&=\Im m_1 \Im m_2 \braket{AA'}-\Im m_1 \braket{\underline{WG_1A\Im G_2} A'} \\
        &\quad -\overline{m_1}\braket{\underline{W\Im G_1 A\Im G_2} A'}+\mathcal{O}_\prec\left(\frac{\Lambda_+^2\rho_1\rho_2}{L}\right).
      \end{split}
    \end{equation}
    Finally, combining~\eqref{eq:chsonn} with the bound for \(\braket{\underline{WG_1A\Im G_2} A'}\) and \(\braket{\underline{W\Im G_1 A\Im G_2} A'}\) in~\eqref{eq:under}, we conclude that
    \begin{equation}
      \braket{\Im G_1 A\Im G_2 A'}=\Im m_1 \Im m_2 \braket{AA'}+\left(\frac{\Lambda_+^2\rho_1\rho_2}{\sqrt{L}}\right).
    \end{equation}
    
  \end{proof}
  
  \begin{proof}[Proof of the local law for  \(\braket{\Im G_1\Im G_2^t}\) in~\eqref{eq:lltrantran}]
    We closely follow the proof of \(\braket{\Im G_1A\Im G_2 A'}\), hence we only explain the differences. As each traceless \(A\), \(A'\) between two resolvents gave rise to a factor \(\Lambda_+\) in the proof of \(\braket{\Im G_1A\Im G_2 A'}\), here the fact that a resolvent is followed by its transpose
    gives rise to a factor \(\Pi_+\). Keeping this modification in mind, in the basic equation for \(\braket{\Im G_1\Im G_2^t}\) we can again estimate all the \(1/N\) terms as in~\eqref{eq:bda1}--\eqref{eq:bda2} and~\eqref{eq:newin2} by \((1+\Pi^2)\rho_1\rho_2L^{-1}\). Then, using the local law \(|\braket{G_i-m_i}|\prec (N\eta_i)^{-1}\), similarly to~\eqref{eq:bigeq}, we conclude that 
    \begin{equation}\label{eq:transpimim}
      \begin{split}
        \braket{\Im G_1 \Im G_2^t}&=\Im m_1 \Im m_2-\Im m _1 \braket{\underline{WG_1\Im G_2^t}}-\overline{m_1}\braket{\underline{W\Im G_1\Im G_2^t}} \\
        &\quad+\sigma \Im m_1 \braket{G_1G_2^t}\braket{\Im G_2}+\sigma \Im m_1\braket{G_1\Im G_2^t}\braket{G_2^*}+\sigma \overline{m_1}\braket{\Im G_1 G_2^t}\braket{\Im G_2} \\
        &\quad+\sigma \overline{m_1}\braket{\Im G_1\Im G_2^t}\braket{G_2^*}+\mathcal{O}_\prec\left(\frac{\Pi_+^2\rho_1\rho_2}{L}\right),
      \end{split}
    \end{equation}
    where we used $\Pi_+:=1+\Pi$. Note that several ``large'' terms remained in~\eqref{eq:transpimim} in contrast to~\eqref{eq:chsonn} since the analogues of \(\braket{\Im G_2A}\) and  \(\braket{G_2^*A}\)  in~\eqref{eq:bigeq} are now  not small. Then using the bounds in~\eqref{eq:under} for the underlined terms in~\eqref{eq:transpimim}, and the local laws
    \begin{equation}\label{eq:needlltransp}
      \begin{split}
        \braket{G_1G_2^t}&=\frac{m_i\Im m_j}{1-\sigma m_1m_2}+\mathcal{O}_\prec\left(\frac{\rho^*\Pi_+^2}{\sqrt{K}}\right), \\
        \braket{G_i\Im G_j^t}&=\frac{m_i\Im m_j}{(1-\sigma m_i m_j)(1-\sigma m_i\overline{m_j})}+\mathcal{O}_\prec\left(\frac{\rho_j\Pi_+^2}{\sqrt{K}}\right),
      \end{split}
    \end{equation}
    we conclude
    \begin{equation}\label{eq:finpap}
      \braket{\Im G_1 \Im G_2^t}=\frac{\Im m_1\Im m_2(1-\sigma^2|m_1m_2|^2)}{|1-\sigma \overline{m_1}m_2|^2(1-\sigma m_1m_2)}+\sigma \overline{m_1}\braket{\Im G_1\Im G_2^t}\braket{G_2^*}+\mathcal{O}_\prec\left(\frac{\Pi_+^2\rho_1\rho_2}{\sqrt{L}}\right).
    \end{equation}
    We remark that the second local law in~\eqref{eq:needlltransp} follows analogously to~\eqref{eq:g1g2t}.
    Finally, writing \(\braket{G_2^*}\) in the last term in the rhs.\ of~\eqref{eq:finpap} as \(\braket{G_2^*}=\overline{m_2}+\braket{(G_2-m_2)^*}\), we conclude~\eqref{eq:lltrantran}.
  \end{proof}
  This concludes the proof of Proposition~\ref{pro:lambll}.
\end{proof}

\section{Feynman diagrams: Proof of Theorem~\ref{chain G underline theorem}}\label{sec cum exp}

For the sake of simpler notations we abbreviate 
\begin{equation}\label{eta K def}
  \eta:=\eta_\ast=\min_k\eta_k,\quad \rho:=\rho^\ast =\max_k\rho_k, \quad K:= N\eta_\ast\rho^\ast\ge L = N \min_k(\eta_k\rho_k)
\end{equation} 
and within this section write \(\rho^i\) and \(\Lambda_+^{a}\) with \(i:=\abs{\mathfrak i},a=\abs{\mathfrak a}\) and \(\Lambda_+:=\max_{k\in\mathfrak a}\Lambda_+^{B_k}\) rather than carrying the products \(\prod_{k\in\mathfrak i}\rho_i\) and \(\prod_{k\in\mathfrak a}\Lambda_+^{B_k}\). Within the formal proof of Theorem~\ref{chain G underline theorem} we argue, however, that the proof naturally yields the latter. In order to present the main body of the proof of Theorem~\ref{chain G underline theorem} more concisely we first make four simplifying assumptions: 
\begin{enumerate}[label=(A-\roman*)]
  \item\label{assump w2} we assume that \(w_2=1+\sigma\),
  \item\label{assump eta} we consider the regime \(\eta\lesssim1\),
  \item\label{assump l a t} for the averaged case we assume that \(l\in\mathfrak a\cup\mathfrak t\) whenever \(\abs{\mathfrak a\cup\mathfrak t}\ne 0\),
  \item\label{assump j iso} in the isotropic bound we only consider \(j\ge 1\).
\end{enumerate}
In Appendix~\ref{sec gen} we address the necessary changes to remove each of these four simplifying assumptions.

\subsection{Graphical representation of the cumulant expansion}
Using multiple cumulant expansions we expand the high moments
\[\E \abs{\braket{\un{WG_1B_1G_2B_2\cdots G_l B_l}}}^{2p}\quad \text{and} \quad\abs{\braket{\vx,\un{G_1 B_1\cdots G_{j} B_j W G_{j+1}B_{j+1}\cdots B_{l-1}G_l }\vy}}^{2p}\] 
as a polynomial of resolvent entries for any \(p\in\N\). More precisely, we iteratively use the expansion
\begin{equation}\label{eq cum exp}
  \begin{split}
    \E w_{ab} f(W) &= \sum_{k=1}^R \sum_{\bm\alpha\in\set{ab,ba}^k}\frac{\kappa(ab,\bm\alpha)}{k!} \E \partial_{\bm\alpha} f(W) + \Omega_R 
  \end{split}
\end{equation}
with some explicit error term \(\Omega_R\) (see e.g.~\cite[Proposition 3.2]{MR3941370}) which for our application can easily be seen to be \(\landauO{N^{-2p}}\) if \(R=12p\). Here for a \(k\)-tuple of double indices \(\bm\alpha=(\alpha_1,\ldots,\alpha_k)\) we use the short-hand notation \(\kappa(ab,(\alpha_1,\ldots,\alpha_k)):=\kappa(w_{ab},w_{\alpha_1},\ldots,w_{\alpha_k})\)
for the joint cumulant of \(w_{ab},w_{\alpha_1},\ldots,w_{\alpha_k}\) and set \(\partial_{\bm\alpha}:=\partial_{w_{\alpha_1}}\cdots\partial_{w_{\alpha_k}}\), \(\partial_{ab}:=\partial_{w_{ab}}\). 
We wish to express the cumulant factors in~\eqref{eq cum exp} as a matrix with \(a,b\) matrix elements. To encode the fact the that cumulants have slightly different combinatorics for \(a=b\) and \(a\ne b\), we rewrite~\eqref{eq cum exp} as
\begin{equation}
  \begin{split}
    \E w_{ab} f(W) &= \sum_{k=1}^R \biggl(\bm1(a=b)\frac{\kappa(\set{aa}^{k+1})}{k!} \E \partial_{aa}^k f(W) \\
    &\qquad\qquad + \bm1(a\ne b) \sum_{q+q'=k}\binom{k}{q}\frac{\kappa(\set{ab}^{q+1},\set{ba}^{q'})}{k!}\E \partial_{ab}^p \partial_{ba}^{q'} f(W)\biggr)+ \Omega_R,
  \end{split}
\end{equation}
where we used that cumulants are invariant under reordering their entries, and thus \(\kappa(ab,\bm\alpha)\) can be expressed as the cumulant of \(q+1\) copies \(\set{ab}^{q+1}\) of \(ab\) and \(q'\) copies \(\set{ba}^{q'}\) of \(ba\). In order to simplify notations we introduce the matrices \(\kappa^{q+1,q'}\) for integers \(q,q'\ge 0\) with \(q+q'\ge 1\) with matrix elements
\begin{align}\label{kappa mat def0}
  \kappa^{1,1}_{ab} &:= 1,\qquad \kappa^{2,0}_{ab} := \sigma, \\\nonumber
  \frac{\kappa^{q+1,q'}_{ab}}{N^{(k+1)/2}} &:=  \bm1(a=b)  \frac{\kappa(\set{aa}^{k+1})}{(k+1)k!} + 
  \bm1(a\ne b)  \binom{k}{q} \frac{\kappa(\set{ab}^{q+1},\set{ba}^{q'})}{k!},  \quad k=q+q'\ge 2,
\end{align}
so that~\eqref{eq cum exp} can be rewritten as 
\begin{equation}\label{eq cum exp2}
  \begin{split}
    \E w_{ab} f(W) &= \sum_{k=1}^R  \sum_{q+q'=k} \frac{\kappa^{q+1,q'}_{ab}}{N^{(k+1)/2}} \E \partial_{ab}^q \partial_{ba}^{q'} f(W) + \Omega_R\\
    &=  \E\frac{\partial_{ba} f(W) + \sigma \partial_{ab} f(W)}{N}+\sum_{k=2}^R\sum_{q+q'=k} \frac{\kappa^{q+1,q'}_{ab}}{N^{(k+1)/2}} \E \partial_{ab}^q \partial_{ba}^{q'} f(W) + \Omega_R,
  \end{split}
\end{equation}
where we used that due to~\ref{assump w2} we have that \(\kappa(\set{\sqrt{N}w_{aa}}^2)=w_2 =1+\sigma=\kappa^{1,1}_{aa}+\kappa^{2,0}_{aa}\). 

We begin with some examples before describing the general structure of the expansion. We consider the case \(p=1\) and \(l=2\) and perform a cumulant expansion 
\begin{equation}\label{WGAImGA 1st exp}
  \begin{split}
    &\E\abs{\braket{\un{WGA \Im G A}}}^2 \\
    &=\E\braket{\un{WGA\Im G A}}\braket{\un{A \Im G A G^\ast W}}\\
    &=  N^{-1}\sum_{ab} \E\Bigl(\braket{\Delta^{ab}GA\Im G A} \partial_{ba} \braket{\un{A\Im G A G^\ast W}}  + \sigma\braket{\Delta^{ab}GA\Im G A} \partial_{ab} \braket{\un{A\Im G A G^\ast W}} \Bigr)\\
    &\qquad\qquad + \sum_{k=2}^R \sum_{q+q'=k}\frac{\kappa^{q+1,q'}_{ab}}{N^{(k+1)/2}} \E\partial_{ab}^q \partial_{ba}^{q'}\Bigl[\braket{\Delta^{ab}GA\Im G A} \braket{\un{A\Im G A G^\ast W}}\Bigr],
  \end{split}
\end{equation}
where \((\Delta^{ab})_{cd}=\delta_{ac}\delta_{bd}\). In order to compute the derivative of \(\Im G\) we write 
\[ \partial_{ba} \Im G = \eta\partial_{ba} G G^\ast = - \eta (G\Delta^{ba} GG^\ast+GG^\ast \Delta^{ba}G^\ast) = - (G\Delta^{ba} \Im G + \Im G\Delta^{ba} G^\ast).\]
By distributing the derivatives according to Leibniz' rule we can write~\eqref{WGAImGA 1st exp} as 
\begin{equation}\label{WGAImGA 2nd exp}
  \begin{split}
    & \E\sum_{ab}\frac{\kappa^{1,1}_{ab}}{N} \braket{\Delta^{ab}GA\Im G A} \braket[\Big]{A\Im G A G^\ast\Delta^{ba}-\un{A(G\Delta^{ba}\Im G + \Im G\Delta^{ba} G^\ast) A G^\ast W}}\\
    &+\E\sum_{ab}\frac{\kappa^{2,0}_{ab}}{N} \braket{\Delta^{ab}GA\Im G A} \braket[\Big]{A\Im G A G^\ast\Delta^{ab}-\un{A(G\Delta^{ab}\Im G + \Im G\Delta^{ab} G^\ast) A G^\ast W}}\\
    & - \E\sum_{ab} \frac{\kappa^{2,1}_{ab}}{N^{3/2}} \braket{\Delta^{ab}G\Delta^{ba}GA\Im G A} \braket[\Big]{A\Im G A G^\ast \Delta^{ab}-\un{A\Im G A G^\ast \Delta^{ab}G^\ast W}} +\cdots
  \end{split}
\end{equation}
where we chose two representative terms for \(k=1\) and \(k=2\) each. By performing another cumulant expansion for the remaining underlined terms in~\eqref{WGAImGA 2nd exp} we obtain 
\begin{equation}\label{WGAImGA 3rd exp}
  \begin{split}
    &N^2\E\abs{\braket{\un{WGA\Im G A}}}^2 \\
    &= \E\sum_{ab}\frac{\kappa^{1,1}_{ab}}{N} (GA\Im G A)_{ba}\biggl[(A\Im G AG^\ast)_{ab} + \sigma(A\Im G AG^\ast)_{ba}\biggr]\\
    &\quad -\E\sum_{ab} \frac{\kappa^{2,1}_{ab}}{N^{3/2}} G_{bb} (GA\Im G A)_{aa}(A\Im G AG^\ast)_{ab}\\
    &\quad +\E\sum_{abcd} \frac{\kappa^{1,1}_{ab}}{N}\frac{\kappa^{1,1}_{cd}}{N}G_{bd}(GA\Im GA)_{ca} \Bigl[(A\Im G)_{db}(G^\ast A G^\ast)_{ac} +(A\Im G)_{da}(G^\ast A G^\ast)_{bc} \Bigr]\\
    &\quad -\E\sum_{abcd} \frac{\kappa^{2,1}_{ab}}{N^{3/2}}\frac{\kappa^{1,1}_{ab}}{N} G_{bb}(GA\Im G )_{ad}(G^\ast A)_{ca} (A\Im G A G^\ast)_{da} G^\ast_{bc}\\
    &\quad +\E\sum_{abcd}\frac{\kappa^{2,1}_{ab}}{N^{3/2}}\frac{\kappa^{2,1}_{cd}}{N^{3/2}} G_{bb}G_{ac}G_{dd}(GA\Im G A)_{ca} (A\Im G A G^\ast)_{da} G^\ast_{bc} + \cdots,
  \end{split}
\end{equation}
where we again selected representative terms. We notice that the rhs.\ can be written as a polynomial in the entries of two types of matrices; the \(\kappa\)-matrices representing cumulants like \(\kappa^{2,1}\), and the \(G\)-matrices representing resolvents like \(\Im G\) or \(G^\ast\), or their multiples with \(A\) like \(A\Im G\), \(G^\ast A\). In order to achieve this representation we introduce additional internal summation indices to expand longer products e.g.\ we write \((A\Im G A G^\ast)_{da}=\sum_e (A\Im G)_{de}(AG^\ast)_{ea}\). The \emph{value} of any given graph is the numerical result of summing up all indices. The precise definition will be given later in~\eqref{val def}; here, as an example, the first term in~\eqref{WGAImGA 3rd exp} with indicated summation indices reads 
\[\sum_{ab}\sum_{ij} \frac{1}{N} \E(GA)_{bi}(\Im G A)_{ia}(A\Im G)_{aj} (AG^\ast)_{jb}= \E\Val\left(\sGraph{ a[label=left:\(a\)] --[dashed] b[label=below:\(b\)]; b -- i1[fn,label=right:\(i\)] -- a; a -- i2[fn,label=above:\(j\)] -- b; }\right),\]
where the (directed) edges represent matrices and the vertices represent summation indices. The edge orientation indicates the order of indices of the represented matrix which for the \(G\)-edges is uniquely determined from the expansion, while for \(\kappa\)-edges it may be chosen arbitrarily, as long as the represented matrix is defined consistently with the orientation, see~\eqref{kappa mat def} later. Here we drew the internal vertices as empty, and the \emph{\(\kappa\)-vertices} as filled nodes, the \(\kappa\)-matrices as dashed, and the \(G\)-matrices as solid edges. Both internal- and \(\kappa\)-vertices correspond to independent summations over the index set \([N]\).  

Thus, graphically we can represent~\eqref{WGAImGA 3rd exp}  as 
\begin{equation}\label{WGAImGA 3rd exp graph}
  \begin{split}
    &N^2\E\abs{\braket{\un{WGA\Im G A}}}^2\\
    & =\E\Biggr[\Val\left(\sGraph{ a --[dashed] b; b -- i1[fn] -- a; a -- i2[fn] -- b; }\right) +\Val\left(\sGraph{ a --[dashed] b; b -- i1[fn] -- a; b -- i2[fn] -- a; }\right) - 
    \Val\left(\sGraph{ a --[dashed] b; b --[glb] b; a --[bl] i1[fn] --[bl] a; a -- i2[fn] --b; }\right) +
    \Val\left(\sGraph{ a --[dashed] b; c --[dashed] d; b--[bl]d; c--i1[fn]--a; d--[bl]b;a--i2[fn]--c; }\right) \\
    &\quad - \Val\left(\sGraph{ a --[dashed] b; d; c; c --[dashed] d; i2[fn]; i1[fn]; b--d; c--i1--a; d--a;b--i2--c; }\right) -
    \Val\left(\sGraph{ a --[dashed] b; c --[dashed] d; b--[glb]b; a--i1[fn]--d; c--a; d--i2[fn]--a; b--c;}\right) +
    \Val\left(\sGraph{ a --[dashed] b; c --[dashed] d; b--[glb]b; a--c; d--[glr]d; c--i1[fn]--a; d--i2[fn] -- a; b--c;}\right)+\cdots\Biggr].
  \end{split}
\end{equation}
Note that the dashed edges connect only filled nodes and they form a perfect matching. The number of \(G\)-edges adjacent to each filled vertex is equal to the order of the corresponding cumulant expansion.

Similarly, for the isotropic case we obtain, for example the polynomials
\begin{equation}\label{iso expansion example}
  \begin{split}
    &\E \abs{\braket{\vx,\un{GA W G}\vy}}^2  = \E \braket{\vx,\un{GA W G}\vy} \braket{\vy,\un{G^\ast W AG^\ast}\vx}\\ 
    &= \E\sum_{ab} \frac{\kappa^{1,1}_{ab}}{N}(GA)_{\vx a}G_{b\vy} G^\ast_{\vy b}(AG^\ast)_{a\vx} \\
    &\quad+  \E\sum_{abcd}\frac{\kappa^{1,1}_{ab}}{N}\frac{\kappa^{1,1}_{cd}}{N} (GA)_{\vx a}G_{bd}G_{c\vy} \Bigl[G^\ast_{\vy b}G^\ast_{ac}(AG^\ast)_{d\vx} + G^\ast_{\vy c}(AG^\ast)_{db} G^\ast_{a\vx} \Bigr]\\
    &\quad - \E\sum_{abcd} \frac{\kappa^{2,1}_{ab}}{N^{3/2}}\frac{\kappa^{1,1}_{cd}}{N} (GA)_{\vx a} G_{bb} G_{ad} G_{c\vy} G^\ast_{\vy a} G^\ast_{bc} (AG^\ast)_{d\vx} + \cdots
  \end{split}
\end{equation}
which we represent graphically as 
\begin{equation}\label{iso expansion example graph}
  \begin{split}
    \E \abs{\braket{\vx,\un{GA W G}\vy}}^2 &
    = \E\Val\left(\sGraph{ x1[rectangle,fill=none,label=above:\(\vx\)] -- a[label=below:\(a\)] --[dashed] b[label=below:\(b\)] -- y1[rectangle,fill=none,label=above:\(\vy\)]; y2[rectangle,fill=none,label=above:\(\vy\)] -- b;  a -- x2[rectangle,fill=none,label=above:\(\vx\)]; }\right)
    + \E\Val\left(\sGraph{ x1[rectangle,fill=none]; y1[rectangle,fill=none]; c; d; x2[rectangle,fill=none]; y2[rectangle,fill=none]; b; a; a --[dashed] b;  x1[rectangle,fill=none]  -- a ; y2 -- b; c --[dashed] d; c -- y1[rectangle,fill=none];  d -- x2[rectangle,fill=none];  b -- d;  a -- c;  }\right) \\&\quad  
    + \E\Val\left(\sGraph{ x1[rectangle,fill=none]  -- a --[dashed] b --[bl] d --[dashed] c -- y1[rectangle,fill=none]; y2[rectangle,fill=none] -- c;  d --[bl] b; a -- x2[rectangle,fill=none]; }\right) 
    - \E\Val\left(\sGraph{ y2[rectangle,fill=none]; x1[rectangle,fill=none]; a; b; y1[rectangle,fill=none]; c; d; c; x2[rectangle,fill=none]; a --[dashed] b; c --[dashed] d; x1[rectangle,fill=none]  -- a; b --[glr] b; a -- d; c -- y1[rectangle,fill=none]; y2 -- a;  d -- x2[rectangle,fill=none]; b--c;}\right)+\cdots,
  \end{split}
\end{equation}
with external vertices drawn as squares. Note that the vectors \(\vx\) and \(\vy\) are naturally represented by \emph{external vertices} that are drawn as solid squares.

After these examples we now explain the general structure of the graphs and give a precise definition of graphs and their \emph{graph value} used in~\eqref{WGAImGA 3rd exp graph} and~\eqref{iso expansion example graph}. 
\begin{definition}\label{def graphs}
  We define the class \(\cG\) of oriented graphs used within this paper by the following requirements. Each \(\Gamma=(V,E)\in\cG\) has three types of vertices, \(\kappa\)-vertices \(V_\kappa\), \emph{internal} vertices \(V_\mathrm{i}\) and external vertices \(V_\mathrm{e}\), so that \(V=V_\kappa\dot\cup V_\mathrm{i}\dot\cup V_\mathrm{e}\), and two types of edges, \emph{\(\kappa\)-edges} \(E_\kappa\) and \emph{\(G\)-edges} \(E_g\), so that \(E=E_\kappa\dot\cup E_g\). 
  For each vertex \(v\in V\) we define its \(G\)-in- and out-degree \(d_g^\mathrm{in}(v),d_g^\mathrm{out}(v)\) as the number of incoming and outgoing \(G\)-edges. 
  The total degree \(d_g(v)\) is defined as the sum \(d_g(v):=d_g^\mathrm{in}(v)+d_g^\mathrm{out}(v)\) of in- and out-degrees and the three vertex classes satisfy 
  \begin{equation}\label{eq dg}
    d_g(v)= \begin{cases}
      1,& v\in V_\mathrm{e}\\
      2, & v\in V_\mathrm{i},
    \end{cases},\qquad d_g(v)\ge 2, \quad v\in V_\kappa.
  \end{equation} 
  We can partition \(V_\kappa = \bigcup_{k\ge 2} V_\kappa^k\) with \(V_\kappa^k:=\set{v\in V_\kappa\given d_g(v)=k}\).
  
  Within the graphs \(\Gamma\) each external vertex \(v\in V_\mathrm{e}\) carries some \(\vx(v)\in\C^N\) as a vector-valued label recording which vector the vertex represents. Each \(\kappa\)-edge \(e\in E_\kappa\) carries two integer-valued labels \(r(e)\ge 1,s(e)\ge 0\) recording the cumulant type. Each \(G\)-edge \(e\in E_g\) carries six labels. The binary labels \(i(e),t(e),\ast(e)\in\set{0,1}\) indicate whether \(e\) represents the imaginary part, the transpose and/or the adjoint of a resolvent. The scalar label \(z(e)\) records the spectral parameter of the resolvent and the matrix-valued labels \(L(e),R(e)\) record deterministic matrices which are multiplied with the resolvent from the left/right. 
\end{definition}

We now relate the graphs to the polynomials they represent. Each internal vertex or \(\kappa\)-vertex \(v\) corresponds to an independent summation \(a_v\in[N]\). 
In order to unify notations we define a \emph{labelling map} 
\begin{equation}
  \begin{split}
    \vx\colon V &\to \C^N,\qquad 
    v\mapsto \vx_v := \begin{cases}
      \vx(e),& v\in V_\mathrm{e},\\
      \bm e_{a_v},& v\in V_\mathrm{i}\cup V_\kappa,
    \end{cases}
  \end{split}
\end{equation} 
where \(\bm e_a\) is the \(a\)-th unit vector in the standard basis, and for \(v\in V_\mathrm{e}\), the vector \(\vx(v)\) is the label of \(v\) from Definition~\ref{def graphs}. The \(G\)-edges \(e\in E_g\) represent resolvents defined via the labels of \(e\) from Definition~\ref{def graphs}. We define the matrix \(\cG^{e}\) as the resolvent \(G(z(e))\) modified according to \(i(e),t(e),\ast(e)\) and multiplied by \(L(e),R(e)\) from the left/right. As an example, we set 
\begin{equation*}
  \cG^e = B(\Im G(z))^t \quad\text{for \(e \in E_g\) with}\quad \Bigl(i(e),t(e),\ast(e),z(e),L(e),R(e)\Bigr)=(1,1,0,z,B,I).
\end{equation*}
We remark that for all \(G\)-edges \(e\) considered in this paper at most one of the matrices \(L(e),R(e)\) is different from the identity matrix \(I\). The \(\kappa\)-edges \(e\in E_\kappa\) represent \(N\times N\) cumulant matrices \(\kappa^e\) which are determined by the two integers \(r(e),s(e)\) from Definition~\ref{def graphs} such that for \(a\ne b\),
\begin{equation}\label{kappa mat def}
  \kappa_{ab}^{(uv)} := \kappa^{r((uv)),s((uv))}_{ab},
\end{equation}
where on the rhs.\ \(\kappa\) was defined in~\eqref{kappa mat def0}. We note that \(\abs{\kappa_{ab}^{(uv)}}\lesssim1\) by~\eqref{eq:momentass}. Finally, we define the graph value 
\begin{equation}\label{val def} 
  \begin{split}
    \Val(\Gamma) := \sum_{\substack{a_v\in[N]\\v\in V_\mathrm{i}\cup V_\kappa}} \biggl[\prod_{(uv)\in E_\kappa} \biggl( N^{-d_g(u)/2}\kappa^{(uv)}_{a_u a_v} \biggr)\biggr] \biggl(\smashoperator[r]{\prod_{(uv)\in E_g}} \cG^{(uv)}_{\vx_u \vx_v}\biggr).
  \end{split}
\end{equation}

Among the degree-\(2\) vertices the ones between edges representing matrices whose eigenvectors are asymptotically orthogonal are of particular importance. There 
are two different mechanism for such orthogonality;  (a) two resolvents, one with and one without transpose stand  next to each other, e.g.\ \(GG^t\) or \(G^\ast (A(\Im G)^t)\), (b) a traceless matrix \(A\) stands between two resolvents, e.g.\ \((GA)G^\ast\) or \(G(A(\Im G)^t)\). Note that in some cases, e.g.\ \( (GA)(\Im G)^t\), both mechanism can be present simultaneously, and hence a vertex can be \(0\mathrm{tr}\)- and \(t\)-vertex at the same time.
\begin{definition}[Orthogonality vertices]\leavevmode%
  \begin{enumerate}[label=(\alph*)]
    \item A vertex \(v\in V_\kappa^2\cup V_\mathrm{i}\) is called a \emph{\(t\)-orthogonality vertex}, or short \emph{\(t\)-vertex} if the two unique \(G\)-edges \(e_1,e_2\in E_g\) adjacent to \(v\) satisfy \((t(e_1),t(e_2))\in\set{(0,1),(1,0)}\), i.e.\ if exactly one of the two \(G\)-edges adjacent to \(v\) is transposed. 
    \item A vertex \(v\in V_\mathrm{i}\cup V_\kappa^2\) is called an \emph{zero-trace-orthogonality vertex}, or short \emph{\(0\mathrm{tr}\)-vertex} if exactly one of the two edges adjacent to \(v\) represents a resolvent (which is allowed to be the imaginary part, transposed, or adjoint) multiplied by a traceless matrix on the side of \(v\), while the other adjacent edge represents a resolvent matrix multiplied by the identity matrix on the side of \(v\). More precisely, using the labels \(L(e),R(e)\) of
    the edges, \(v\) is defined to be an \(0\mathrm{tr}\)-vertex if one of the following three conditions is satisfied:
    \begin{enumerate}[label=(\alph{enumi}.\roman*)]
      \item there are incoming/outgoing edges \((uv),(vw)\in E_g\) such that either \(\braket{R((uv))}=0,\allowbreak L((vw))=I\) or \(\braket{L((vw))}=0,R((uv))=I\), 
      \item there are two outgoing edges \((vu),(vw)\in E_g\) such that either \(\braket{L((vu))}=0,L((vw))=I\) or \(\braket{L((vw))}=0,L((vu))=I\), 
      \item there are two incoming edges \((uv),(wv)\in E_g\) such that either \(\braket{R((uv))}=0,R((wv))=I\) or \(\braket{R((wv))}=0,R((uv))=I\).
    \end{enumerate}
  \end{enumerate}
\end{definition}

\begin{proposition}[Cumulant expansion]\label{cumulant expansion prop}
  Let \(\mathfrak a,\mathfrak t,\mathfrak i\) be fixed sets as in Theorem~\ref{chain G underline theorem} of sizes \(a:=\abs{\mathfrak a}, t:=\abs{\mathfrak t}, i:=\abs{\mathfrak i}\). 
  Then for any \(p\in\N\) there exists a finite (\(N\)-independent) family of graphs \(\cG_p=\cG_p^\mathrm{av}\cup\cG_p^\mathrm{iso}\subset \cG\) such that 
  \begin{align}\label{eq cum expansion formula}
    \E\abs{\Tr \un{WG_1B_1G_2B_2\cdots G_l B_l}}^{2p}&=\sum_{\Gamma\in\cG_p^\mathrm{av}}\E\Val(\Gamma) + \landauO{N^{-2p}},\\\label{eq cum expansion formula iso}
    \E\abs{\braket{\vx,\un{G_1 B_1\cdots G_{j} B_j W G_{j+1}B_{j+1}\cdots B_{l-1}G_l }\vy}}^{2p}&=\sum_{\Gamma\in\cG_p^\mathrm{iso}}\E\Val(\Gamma) + \landauO{N^{-2p}},
  \end{align}
  and for each graph \(\Gamma\) we may select two disjoint subsets \(V_\mathrm{o}^t\dot\cup V_\mathrm{o}^{0\mathrm{tr}}=:V_\mathrm{o}\) of \(t\)- and \(0\mathrm{tr}\)-vertices, respectively, such that the following properties are satisfied: 
  \begin{genprop}
    \item\label{perfect matching} The graph \((V_\kappa,E_\kappa)\) is a perfect matching, in particular, \(\abs{V_\kappa}=2\abs{E_\kappa}\).
    \item\label{number of kappa edges} The number of \(\kappa\)-edges satisfies \(1\le\abs{E_\kappa}\le 2p\).
    \item\label{number of G edges} The number of \(G\)-edges satisfies 
    \begin{subequations}
      \begin{align}
        \abs{\set*{e\in E_g\given i(e)=1}}&=2ip\label{imG count}\\
        \abs{E_g} &= \sum_{e\in E_\kappa} d_g(e)+2(l-1)p\ge 2p.\label{eg total count}
      \end{align}
    \end{subequations}
    \item\label{degree of kappa vertices} For \((uv)\in E_\kappa\) the \(G\)-degrees of \(u,v\in V_\kappa\) satisfy \(d_g^\mathrm{in}(u)=d_g^\mathrm{out}(v)\), \(d_g^\mathrm{in}(v)=d_g^\mathrm{out}(u)\) and \(d_g(u)=d_g(v)\ge 2\). Therefore we may define the \(G\)-degree of \((uv)\) as \(d_g((uv)):=d_g(u)=d_g(v)\) and partition \(E_\kappa=\dot\bigcup_{k\ge 2} E_\kappa^k\) into \(E_\kappa^k := \set{e\in E_\kappa\given d_g(e)=k}\). 
    \item\label{no loops} Every \(E_g\)-cycle on \(V_\kappa^2\cup V_\mathrm{i}\) must contain at least two \(V_\kappa^2\)-vertices, and in particular there cannot exist isolated loop edges, and there are at most \(\abs{E_\kappa^2}\) cycles. 
    \item\label{Va Vcyc sum claim} Denoting the number of isolated cycles in \((V_\kappa\cup V_\mathrm{i},E_g)\) with \(k\) vertices in \(V_\mathrm{o}\) by \(n_\mathrm{cyc}^{o=k}\), we have 
    \[2n_\mathrm{cyc}^{o=0}+n_\mathrm{cyc}^{o=1} \le 2\abs{E_\kappa^2}-\abs{V_\mathrm{o}\cap V_\kappa^2}.\]
    \item\label{int a rule} The numbers of \emph{selected internal \(0\mathrm{tr}\)- and \(t\)-vertices} are
    \[\abs{V_\mathrm{i}\cap V_\mathrm{o}^{0\mathrm{tr}}}= \begin{cases}
      2p(a-1), & j\in \mathfrak{a}\\ 2pa, & j\not\in\mathfrak{a},
    \end{cases},\quad \abs{V_\mathrm{i}\cap V_\mathrm{o}^{t}}= \begin{cases}
      2p(t-1), & j\in \mathfrak{t}\\ 2pt, & j\not\in\mathfrak{t},
    \end{cases}\] 
    where in the averaged case \(j:=l\) and \(j\) is determined by the lhs.\ of~\eqref{eq cum expansion formula iso} in the isotropic case.
    \item\label{a rule} \newcounter{lastgen}\setcounter{lastgen}{\value{genpropi}+1} If \(j\in\mathfrak{a}\) (with again \(j:=l\) in the averaged case), then the set of selected \(0\mathrm{tr}\)-vertices \(V_\mathrm{o}^{0\mathrm{tr}}\) satisfies
    \[2 \abs{E_\kappa^2} + \abs{E_\kappa^{\ge 3}} - 2p\le \abs{V_\mathrm{o}^{0\mathrm{tr}}\cap V_\kappa^2}\le 2p,\]
    while otherwise \(V_\mathrm{o}^{0\mathrm{tr}}\cap V_\kappa^2=\emptyset\). Similarly, if \(j\in\mathfrak{t}\), then the set of selected \(t\)-vertices \(V_\mathrm{o}^t\) satisfies
    \[2 \abs{E_\kappa^2} + \abs{E_\kappa^{\ge 3}} - 2p\le \abs{V_\mathrm{o}^t\cap V_\kappa^2}\le 2p,\]
    while otherwise \(V_\mathrm{o}^t\cap V_\kappa^2=\emptyset\). 
  \end{genprop}
  The graphs \(\Gamma\in \cG_p^\mathrm{av}\) satisfy~\ref{perfect matching}--\ref{a rule} and in addition:
  \begin{avprop}[start=\value{lastgen}]
    \item\label{no external} There are no external vertices, i.e.\ \(V_\mathrm{e}=\emptyset\)
    \item\label{number of internal vertices} The number of internal vertices satisfies \(\abs{V_\mathrm{i}}=2(l-1)p\).
  \end{avprop}  
  The graphs \(\Gamma\in \cG_p^\mathrm{iso}\) satisfy~\ref{perfect matching}--\ref{a rule} and in addition: 
  \begin{isoprop}[start=\value{lastgen}]
    \item\label{number of external} The number of external vertices is \(\abs{V_\mathrm{e}}=4p\) each \(v\in V_\mathrm{e}\) has degree \(d_g(v)=1\) and the unique connected vertex \(u\in V\) with \((uv)\in E_g\) or \((vu)\in E_g\) satisfies \(u\in V_\kappa\).
    \item\label{no internal} The number of internal vertices satisfies \(\abs{V_\mathrm{i}}=2p(l-2)\).
  \end{isoprop}
\end{proposition}

\begin{definition}\label{def av iso graphs}
  For some parameters \(a,t,l,i,p\in\N\) we call graphs \(\Gamma\in\cG\) together with their selected \(V_\mathrm{o}^t,V_\mathrm{o}^{0\mathrm{tr}}\) sets satisfying~\ref{perfect matching}--\ref{a rule} and~\ref{no external}--\ref{number of internal vertices} \emph{av-graphs}, while we call graphs \(\Gamma\in\cG\) (together with the sets \(V_\mathrm{o}^t,V_\mathrm{o}^{0\mathrm{tr}}\) and the extra parameter \(j\in[l-1]\)) satisfying~\ref{perfect matching}--\ref{a rule} and~\ref{number of external}--\ref{no internal} \emph{iso-graphs}. 
\end{definition}

\begin{proof}[Proof of Proposition~\ref{cumulant expansion prop}]
  In order to obtain~\eqref{eq cum expansion formula} we iteratively perform cumulant expansions exactly as in the examples~\eqref{WGAImGA 1st exp} and~\eqref{iso expansion example} until no underlined terms remain. Each cumulant expansion removes at least one underlined term, hence this process terminates. 
  
  We now explain which kinds of \(G\)-edges are created through this cumulant expansion procedure for the averaged case~\eqref{eq cum expansion formula}, the isotropic case~\eqref{eq cum expansion formula iso} being very similar. Initially, the graph representing the lhs.\ of~\eqref{eq cum expansion formula} after writing out \(\abs{\Tr X}^{2p}=(\Tr X)^p (\Tr X^\ast)^p\) consists of \(2p\) cycles each with a \(W\) factor and \(l\) \(G\)-edges representing \emph{\(\cG\)-factors} \(G_k B_k\) or \(B_k^\ast G_k^\ast\) for \(k\in[l]\). Each of these \(\cG\)-factors can be fully described via the labels \(i(e),t(e),\ast(e),z(e),L(e),R(e)\) from Definition~\ref{def graphs}, the first four being determined by the form of \(G_k\) while the latter two encode the multiplication from the left/right by deterministic matrices, e.g.\ \(L(e)=I,R(e)=B_k\) for \(G_k B_k\). While performing cumulant expansions of some \(W=\sum_{ab}w_{ab}\Delta^{ab}\) using~\eqref{eq cum exp} these \(G\)-edges are modified and new \(G\) edges are created via the action of derivatives, and \(\kappa\)-edges representing \(\kappa(ab,\bm\alpha)\) are also created. This process creates creates (finitely) many different graphs for every cumulant expansion, both through the explicit summation over cumulants in~\eqref{eq cum exp} and the Leibniz rule for the derivative \(\partial_{\bm\alpha}\) acting on the product of all remaining \(W\)'s and \(G\)'s. We note that for resolvent derivatives we have 
  \[
  \begin{split}
    \partial_{ab} G&= -G\Delta^{ab}G,\quad \partial_{ab} G^\ast = -G^\ast \Delta^{ab}G^\ast, \\
    \partial_{ab} G^t&= -G^t\Delta^{ba}G^t,\quad \partial_{ab} (G^\ast)^t= -(G^\ast)^t\Delta^{ba}(G^\ast)^t
  \end{split}\] 
  and  
  \[\partial_{ab} \Im G= -G\Delta^{ab}\Im G - \Im G\Delta^{ab}G^\ast, \quad \partial_{ab} (\Im G)^t= -(\Im G)^t\Delta^{ba}G^t - (G^\ast)^t\Delta^{ba}(\Im G)^t.\] 
  Hence, a derivative action on \(e\) representing the \(\cG\)-factor \(\cG^e=G_kB_k\) (or similarly \(B_k^\ast G_k^\ast\)) creates two \(G\)-edges \(e_1,e_2\), such that only the resolvent representing \(e_2\) is multiplied from the right by \(R(e_2)=B_k\) while \(L(e_2)=L(e_1)=R(e_1)=I\). The labels \(t(e),z(e)\) indicating the transposition status and spectral parameter are directly inherited to both \(e_1,e_2\), while the label \(i(e)\) is inherited to exactly one of \(e_1,e_2\), \(i(e_1)=1\) the other one satisfying \(i(e_2)=0\), \(\ast(e_2)\in\set{0,1}\). If \(\ast(e)=1,i(e)=0\), then both \(e_1,e_2\) satisfy \(\ast(e_1)=\ast(e_2)=1\). It follows inductively that each \(\cG\)-factor encountered in the expansion can be represented by an edge \(e\) with six labels \(i(e),t(e),\ast(e),z(e),L(e),R(e)\), with \(L(e)=I\), or \(L(e)=B_k^\ast\) for some \(k\), while \(R(e)=I\) or \(R(e)=B_k\) for some \(k\), with for each \(e\) at least one of \(L(e),R(e)\) being the identity. The spectral parameter label \(z(e)\) satisfies \(z(e)\in \set{z_1,\dots,z_l}\) for each \(e\). For example, the \(ab\)-derivative of the \(\cG\)-factor \(\cG^e=B(\Im G(z))^t\) described by the edge \(e\) with labels \((1,1,0,z,B,I)\) yields a sum of two terms and hence the two new graphs given by
  \[ \begin{tikzpicture}[baseline={([yshift=-5pt]current bounding box.center)},font=\footnotesize,>=Stealth,node distance = 15pt and 60pt]
    \node[fill=white,rectangle,draw=black,inner sep=0pt,minimum size=3pt] (A) {};
    \node[fill=white,rectangle,draw=black,inner sep=0pt,minimum size=3pt] (B) [right= of A] {};
    \node (C) [below= of A] {\(\partial_{ab}\)};
    \path[-stealth]
    (A) edge node [above] {\((1,1,0,z,B,I)\)} (B);
    \path[-stealth]
    (C) edge[dashed, bend right] ($ (A) !.5! (B) $);
  \end{tikzpicture}  
  \quad\Rightarrow\quad
  \begin{tikzpicture}[baseline={([yshift=-6pt]current bounding box.center)},font=\footnotesize,>=Stealth,node distance = 20pt and 60pt]
    \node[fill=white,rectangle,draw=black,inner sep=0pt,minimum size=3pt] (A) {};
    \node[circle,fill=black,label={\(b\)},inner sep=0pt,minimum size=3pt] (B) [right= of A] {};
    \node[fill=white,rectangle,draw=black,inner sep=0pt,minimum size=3pt] (C) [below= of A] {};
    \node[circle,fill=black,label={\(a\)},inner sep=0pt,minimum size=3pt] (D) [below= of B] {};
    \path[-stealth]
    (A) edge node [above] {\((1,1,0,z,B,I)\)} (B);
    \path[-stealth]
    (D) edge node [above] {\((0,1,0,z,I,I)\)} (C);
  \end{tikzpicture} 
  \quad \& \quad
  \begin{tikzpicture}[baseline={([yshift=-6pt]current bounding box.center)},font=\footnotesize,>=Stealth,node distance = 20pt and 60pt]
    \node[fill=white,rectangle,draw=black,inner sep=0pt,minimum size=3pt] (A) {};
    \node[circle,fill=black,label={\(b\)},inner sep=0pt,minimum size=3pt] (B) [right= of A] {};
    \node[fill=white,rectangle,draw=black,inner sep=0pt,minimum size=3pt] (C) [below= of A] {};
    \node[circle,fill=black,label={\(a\)},inner sep=0pt,minimum size=3pt] (D) [below= of B] {};
    \path[-stealth]
    (A) edge node [above] {\((0,1,1,z,B,I)\)} (B);
    \path[-stealth]
    (D) edge node [above] {\((1,1,0,z,I,I)\)} (C);
  \end{tikzpicture} 
  \]
  or, in formulas, 
  \[ \partial_{ab}\Bigl[ B(\Im G(z))^t\Bigr] = -B(\Im G(z))^t\Delta^{ba}G^t-B(G^\ast)^t\Delta^{ba} (\Im G(z))^t.\]
  
  We now describe the selection of the orthogonality vertices \(V_\mathrm{o}^t,V_\mathrm{o}^{0\mathrm{tr}}\) which is done in two steps. To unify notations we set \(j:=l\) in the averaged case.
  \begin{enumerate}[label=(orth-\arabic*)]
    \item\label{claim orth1} For each \(k\in\mathfrak{t}\setminus\set{j},\mathfrak a\setminus\set{j}\) we collect \(2p\) distinct vertices from \(V_\mathrm{i}\) into the sets \(V_\mathrm{o}^t\) and \(V_\mathrm{o}^{0\mathrm{tr}}\), respectively. 
    \item\label{claim orth2} If \(j\in\mathfrak t\) or \(j\in\mathfrak a\), then we select one vertex from \(V_\kappa^2\) into \(V_\mathrm{o}^t\) or \(V_\mathrm{o}^{0\mathrm{tr}}\), respectively, for each \(W\) acting as a degree-\(2\) cumulant on some resolvent. 
  \end{enumerate}
  Regarding~\ref{claim orth1} for \(k\in\mathfrak{a}\cup\mathfrak{t}\setminus\set{j}\) the initial graphs representing the lhs.\ of~\eqref{eq cum expansion formula}--\eqref{eq cum expansion formula iso} contain \(p\) internal vertices \(v_1,\ldots v_{p}\) between \(G\)-edges representing \((G_k B_k),(G_{k+1}B_{k+1})\) and \(p\) internal vertices \(v_{p+1},\ldots v_{2p}\) between \(G\)-edges representing \((B_{k+1}^\ast G_{k+1}^\ast),(B_k^\ast G_k^\ast)\). The \(G\)-edges adjacent to these internal vertices may change due to derivative actions along the cumulant expansions, however in case \(k\in\mathfrak a\), due to the derivative rules explained in the paragraph above it is ensured that all times the two unique \(G\)-edges \(e_1,e_2\) adjacent to \(v_k\) satisfy \(R(e_1)=B_k,L(e_2)=I\) for \(k\le p\) and \(R(e_1)=I,L(e_2)=B_k^\ast\), so that \(v_k\) is guaranteed to remain an \(0\mathrm{tr}\)-vertex. Similarly, for \(k\in\mathfrak t\) it is ensured that the two unique \(G\)-edges \(e_1,e_2\) adjacent to \(v_k\) satisfy \(t(e_1)=1,t(e_2)=0\), so that \(v_k\) is guaranteed to remain an \(t\)-vertex. 
  
  Regarding~\ref{claim orth2} we note that while performing the cumulant expansion for \(W=\sum_{ab}w_{ab}\Delta^{ab}\) in \(G_j B_j WG_{j+1}\) we obtain the degree-\(2\) cumulant term as 
  \[ \sum_{ab}G_j B_j \Delta^{ab}G_{j+1}\bigl(\partial_{ba}+\sigma\partial_{ab}\bigr)\] 
  the derivatives \(\partial_{ab}\) or \(\partial_{ba}\) acting on some resolvent \(G\) result in \(G\Delta^{ab}G\) or \(G\Delta^{ba}G\). In case \(j\in\mathfrak a\) the \(\kappa\)-vertex corresponding to the summation index \(a\) satisfies the definition of \(0\mathrm{tr}\)-vertex since \(\braket{B_j}=0\) and the other resolvent is not multiplied by some additional matrix in the \(a\)-direction. Similarly, in case \(j\in\mathfrak{t}\) either both or none of the two \(G\)'s in \(G\Delta^{ab}G\) or \(G\Delta^{ba}G\) are transposed, while, by definition, exactly one of \(G_j,G_{j+1}\) is transposed. Thus exactly one of the \(\kappa\)-vertices corresponding to the \(a\) or \(b\)-summations satisfies the definition of being a \(t\)-vertex. 
  
  We note that the condition \(\mathfrak a\cap\mathfrak t=\emptyset\) ensures the sets \(V_\mathrm{o}^t,V_\mathrm{o}^{0\mathrm{tr}}\) constructed in this way to be disjoint. We now check that the properties~\ref{perfect matching}--\ref{a rule}, as well as~\ref{no external}--\ref{number of internal vertices} and~\ref{number of external}--\ref{no internal} also hold for these graphs.
  
  The properties~\ref{perfect matching}--\ref{number of kappa edges} are obvious by construction since each cumulant expansion comes with two \(\kappa\)-vertices, and in total there are \(2p\) underlined terms and thereby at most \(2p\) cumulant expansions. The properties~\ref{number of internal vertices},~\ref{no internal} follow from the fact that for each factor of \(\Tr \un{WG_1B_1\cdots G_l B_l}\) and \(\braket{\vx,\un{G_1 B_1\cdots G_{j} B_j W G_{j+1}B_{j+1}\cdots B_{l-1}G_l }\vy}\) there are \(l-1\) and respectively \(l-2\) internal vertices of in- and out-degree \(1\) and that these properties remain invariant under cumulant expansions. Similarly, the properties~\ref{no external} and~\ref{number of external} hold true trivially for the initial terms and remain invariant under cumulant expansions. 
  
  For~\ref{degree of kappa vertices} note that the cumulant \(\kappa(ab,(\alpha_1,\ldots,\alpha_k))\) comes together with matrices 
  \[ \Delta^{ab}, (\Delta^{\alpha_1})^{(t)},\ldots, (\Delta^{\alpha_k})^{(t)}\] 
  after derivative action, where the transpose is taken in case the derivative acts on a transposed resolvent. In all cases the in-degree of the vertex associated with \(a\) is equal to the out-degree of the vertex associated with \(b\). 
  
  For~\ref{no loops} note that by the definition of the underline-renormalisation it follows that for degree two edges the corresponding \(\partial_{ba}\) derivative cannot act on its own trace and therefore cycles have to involve at least two \(V_\kappa^2\) vertices. 
  
  For~\ref{Va Vcyc sum claim} we note that \(\abs{V_\kappa^2\setminus V_\mathrm{o}}=2\abs{E_\kappa^2}-\abs{V_\mathrm{o}\cap V_\kappa^2}\), while due to~\ref{no loops} each cycle with zero \(V_\mathrm{o}\)-vertices contains at least two \(V_\kappa^2\setminus V_\mathrm{o}\)-vertices and each cycle with one \(V_\mathrm{o}\)-vertex contains at least one \(V_\kappa^2\setminus V_\mathrm{o}\)-vertex. 
  
  The claim~\ref{int a rule} follows immediately from the construction~\ref{claim orth1}. Similarly, claim~\ref{a rule} follows from the construction~\ref{claim orth2} together with the observation that because \(\abs{E_\kappa}\) is the total number of cumulant expansions, a total of \(2p-\abs{E_\kappa}\) derivatives have acted on some \(W\), and thus the number \(n\) of \(W\)'s acting on as degree-\(2\) cumulants on some \(G\) satisfies 
  \begin{equation}\label{number of G actions}
    n \ge \abs{E_\kappa^2} - (2p-\abs{E_\kappa})= 2\abs{E_\kappa^2} + \abs{E_\kappa^{\ge3}}-2p,
  \end{equation}
  and, trivially, \(n\le 2p\). This concludes the proof of~\ref{a rule} in the mutually exclusive cases \(j\in\mathfrak a\) and \(j\in\mathfrak t\) (recall that \(\mathfrak a\cap\mathfrak t=\emptyset\) by assumption). 
  
  For the claim~\eqref{imG count} on the number of \(G\)-edges in~\ref{number of G edges} note that the number of \(\Im G\)'s remains invariant under the derivative actions. For~\eqref{eg total count} note that each derivative acting on some \(G\) increases the number of \(G\)'s by one, while each of the \(2p-\abs{E_\kappa}\)derivatives acting on some \(W\) leaves the number of \(G\)'s invariant. Thus we conclude that the total number of \(G\)'s is  
  \[ 2lp + \sum_{e\in E_\kappa} d_g(e) - \abs{E_\kappa} - (2p-\abs{E_\kappa}) = \sum_{e\in E_\kappa} d_g(e) +2(l-1)p\]
  and~\eqref{eg total count} follows.
\end{proof}
\begin{remark}\label{remark alt underline expl}
    Proposition~\ref{cumulant expansion prop} holds true verbatim also under the alternative definition of the renormalisation outlined in Remark~\ref{remark alt underline} in case no \(G\) is transposed. Also the proof of the proposition remains unchanged except for the proof of Property~\ref{no loops}. For the alternative renormalisation also for degree two edges when expanding \(\un{WG\cdots}=\sum_{ab}\Delta^{ab}G\cdots \partial_{ba}\) the derivative \(\sigma\partial_{ab}\) may act on its own trace. However, since no \(G\) is transposed this action will necessarily result in \(\Delta^{ab} G\cdots G\Delta^{ab}\) and therefore no loops are created. 
\end{remark}

Using Proposition~\ref{cumulant expansion prop}, in order to conclude Theorem~\ref{chain G underline theorem}, it remains to estimate \(\Val(\Gamma)\) for each \(\Gamma\in\cG_p\) as follows. We note that the following Proposition is valid for any av-/iso-graphs \(\Gamma\in\cG\) from Definition~\ref{def av iso graphs}, i.e.\ graphs satisfying the properties~\ref{perfect matching}--\ref{a rule} and~\ref{no external}--\ref{number of internal vertices}/\ref{number of external}--\ref{no internal} above rather than only for the specific families of graphs \(\cG_p^\mathrm{av},\cG_p^\mathrm{iso}\) arising in the cumulant expansion. 
\begin{proposition}[Value estimate]\label{prop value est}
  For each av-graph \(\Gamma\in\cG\) for some parameters \(a,t,l,p,i\in\N\) we have the bound
  \begin{equation}\label{eq val est}
    \abs{\Val(\Gamma)} \prec \begin{cases}
      \rho^{2(b+1)p} N^{2bp} K^{-2bp}, & b=l\\
      \Lambda_+^{2ap}\Pi_+^{2tp}\rho^{2ip\vee 2(b+1)p}N^{p(a+t+2b)} K^{-p(1+2b)}  , & b<l,
    \end{cases},\quad b:= l-a-t 
  \end{equation}  
  with \(K\) as in~\eqref{eta K def}, while for each iso-graph \(\Gamma\) for some parameters \(a,t,l,p,i\in\N\) we have the bound
  \begin{equation}\label{eq val est2}
    \abs{\Val(\Gamma)} \prec \Lambda_+^{2ap}\Pi_+^{2tp}\rho^{2ip\vee 2(b+1)p}N^{p(a+t+2b)} K^{-p(1+2b)},\quad b:= l-a-t-1.
  \end{equation}  
\end{proposition}
\begin{proof}[Proof of Theorem~\ref{chain G underline theorem}]
  Theorem~\ref{chain G underline theorem} follows immediately by combining Propositions~\ref{cumulant expansion prop} and~\ref{prop value est} under the simplifying assumptions made at the beginning of Section~\ref{sec cum exp}, the removal of which is discussed in Appendix~\ref{sec gen}. Following the proof Proposition~\ref{prop value est} it is evident that both \(\Lambda_+^{2ap}\) and \(\Pi_+^{2tp}\) can be replaced by the product of individual \(\Lambda_+^{B_k}\), \(\Pi_+^{B_k}\) for \(k\in\mathfrak a\cup \mathfrak t\), as claimed in Theorem~\ref{chain G underline theorem}.  
  
  Finally, regarding the replacement of \(\rho^i\) by \(\prod_{k\in\mathfrak i}\rho(z_k)\) in the bounds of Theorem~\ref{chain G underline theorem}, it is easy to see that during the cumulant expansion the number of \(\Im G(z_k)\) is preserved and each gives rise to a factor \(\rho(z_k)\) in Proposition~\ref{prop value est}, hence the factor \(\rho^{2ip}\) may be replaced by the factor \(\prod_{k\in\mathfrak i} \rho(z_k)^{2p}\). Similarly, for the replacement of \(\Lambda_+^a\) by \(\prod_{k\in\mathfrak a}\Lambda_+^{B_k}\) we note that each \(B_k\) appears exactly \(2p\) times also after the cumulant expansions, and therefore each \(\Lambda_+^{B_k}\) can only appear in at most the \(2p\)-th power on the rhs.\ of~\eqref{eq val est}--\eqref{eq val est2}.
\end{proof}

\subsection{Estimating graph values: Proof of Proposition~\ref{prop value est}} The proof of Proposition~\ref{prop value est} goes in three major steps formulated in Lemmata~\ref{lemma reduction},~\ref{lemma Itwo} and~\ref{lemma Ithree} which we first state and then use to conclude the proof of Proposition~\ref{prop value est}.

First, we express the value \(\Val(\Gamma)=\Val(\Gamma_\mathrm{red})\) as the value of the \emph{reduced graph} \(\Gamma_\mathrm{red}\) obtained from \(\Gamma\) by collapsing all degree-\(2\) vertices \(V_\mathrm{i}\cup V_\kappa^2\). Thus, in graph-theoretic terms, \(\Gamma_\mathrm{red}\) is the minimal (with the least number of edges) graph having \(\Gamma\setminus E_\kappa^2\) as a \emph{subdivision}. We claim that each summation index \(a_v\) for \(v\in V_\mathrm{i}\cup V_\kappa^2\) appears in exactly two \(\cG\)-factors and no \(\kappa\)-matrices, and thus the summation can be written as a matrix product after (potentially) transposing one of the two \(\cG\)'s in the cases of two incoming or two outgoing edges, e.g.\ \(\sum_{a_v} (GB)_{\vx a_v} G_{\vy a_v}=(GBG^t)_{\vx\vy}\). Indeed, the index \(a_v\) appears only in exactly two \(G\)-edges since \(d_g(v)=2\), cf.\ Definition~\ref{def graphs}. Moreover, due to~\ref{perfect matching} no \(\kappa\)-edge is adjacent to \(V_\mathrm{i}\) while for \(v\in V_\kappa^2\) the corresponding \(\kappa\)-edge \((uv)\) or \((vu)\) due to~\eqref{kappa mat def0} and~\eqref{kappa mat def} is given by \(\kappa^{1,1}\) or \(\kappa^{2,0}\) which are constant-\(1\), and constant \(\sigma\)-matrices, and thus effectively the index \(a_v\) does not appear in any \(\kappa^{(vu)/(uv)}\) matrix. 

In the reduction process the value of \(\Gamma\) effectively reduces to a summation over vertices of degree at least \(3\), traces of \(G\)-cycles and entries of \(G\)-chains and \(E_\kappa^{\ge 3}\)-matrices, represented by \(\Gamma_\mathrm{red}\). Here we use the terminology that a \emph{\(G\)-cycle} is a cycle of \(G\)-edges on \(V_\kappa^2 \cup V_\mathrm{i}\) vertices, irrespective of the edge orientation, and that a \emph{\(G\)-chain} is a chain of \(G\)-edges with internal \(V_\kappa^2 \cup V_\mathrm{i}\)-vertices and external \(V_\kappa^{\ge3}\cup V_\mathrm{e}\)-vertices, again irrespective of the edge orientation. Note that the reduction completely collapses each \(E_g\)-cycles on \(V_\mathrm{i}\cup V_\kappa^2\)-vertices into a single vertex with a loop edge. The sets of these single vertices and loop edges are denoted by \(V_\mathrm{cyc}\) and \(E_g^{\mathrm{red},\mathrm{cyc}}\). Therefore the edge set of reduced graph \(\Gamma_\mathrm{red}\) is naturally partitioned into \(V(\Gamma_\mathrm{red}):= V^{\ge 3}_\kappa \dot \cup V_\mathrm{e} \dot \cup V_\cyc\) and its edge set is \(E(\Gamma_\mathrm{red}):= E_g^\mathrm{red} \dot\cup E^{\ge 3}_\kappa\). 

The graph reduction by partial resummations corresponds to generalising the definition of value to
\begin{equation}\label{Val reduced def} 
  \begin{split}
    \Val(\Gamma_\mathrm{red}):= N^{-\abs{E_\kappa^2}+\abs{V_\mathrm{cyc}}}  \sum_{\substack{a_v\in[N]\\v\in V_\kappa^{\ge 3}}} &\biggl[\prod_{(uv)\in E_\kappa^{\ge 3}} \biggl( N^{-d_g(u)/2} \kappa^{(uv)}_{a_u a_v} \biggr)\biggr] \\
    &\times \biggl(\smashoperator[r]{\prod_{v\in V_\mathrm{cyc}}}\braket{\cG^{(vv)}}\biggr)\biggl(\prod_{\substack{(uv)\in E_g^{\red}\setminus E_g^{\red,\cyc}}} \cG^{(uv)}_{\vx_u\vx_v}\biggr).
  \end{split}
\end{equation}
where we defined
\begin{equation}\label{G chain def}
  \cG^{(v_1 v_k)}:=(\cG^{(v_1v_2)})^{(t)}\cdots(\cG^{(v_{k-1}v_k)})^{(t)}
\end{equation} 
as a matrix product of (possibly transposed, depending on the in- and out-degrees) of \(\cG^{(v_1v_2)},\allowbreak \ldots,\allowbreak\cG^{(v_{k-1}v_k)}\), whenever \(d_g(v_2)=\cdots=d_g(v_{k-1})=2\). For each edge \(e\in E_g^\mathrm{red}\) we record the number of \(\Im G\)'s, the total number of \(G\)-edges and the number of summed up \(V_\mathrm{o}^{t/0\mathrm{tr}}\)-vertices in the corresponding chains and cycles by \(i(e),l(e),t(e),a(e)\), respectively and set \(o(e):=a(e)+t(e)\). The letter \(o\) refers to the counting of vertices with the asymptotic orthogonality effect. Note that for cycles all \(V_\mathrm{o}^{t/0\mathrm{tr}}\)-vertices in the cycle contribute towards \(t(e),a(e)\) while for chains the first and last vertex necessarily are in \(V_\kappa^{\ge 3}\cup V_\mathrm{e}\) and hence, by definition cannot be \(V_\mathrm{o}^{t/0\mathrm{tr}}\)-vertices. Thus the parameters \(a,t,i,l\) satisfy the relations 
\begin{equation}\label{ali relations}
  1\le l(e), \quad 0 \le i(e)\le l(e),\quad 0\le a(e)+ t(e)=o(e)\le \begin{cases}
    l(e), & e\in E_g^{\red,\cyc},\\ 
    l(e)-1,& e\in E_g^{\red}\setminus E_g^{\red,\cyc}.
  \end{cases} 
\end{equation}
We denote the set of \(v\in V_\mathrm{cyc}\) with \(o((vv))=k\) by \(V_\mathrm{cyc}^{o=k}\) which are of cardinality \(\abs{V_\mathrm{cyc}^{o=k}}=n_\mathrm{cyc}^{o=k}\), c.f.~\ref{Va Vcyc sum claim}. 
\begin{lemma}\label{lemma reduction}
  For each av-/iso-graph \(\Gamma\in \cG\) with parameters \(a, t, l, i,p\) and the selected vertex sets \(V_\mathrm{o}^{0\mathrm{tr}},V_\mathrm{o}^{t}\), let \(\Gamma_\mathrm{red}=(V_\kappa^{\ge 3}\cup V_\mathrm{e}\cup V_\mathrm{cyc},E_g^\mathrm{red}\cup E_\kappa^{\ge 3})\) denote its reduction. The reduced graph then satisfies
  \begin{equation}\label{number edges Egred}
    \abs{E_g^\mathrm{red}} = \abs{E_g} - \abs{V_\mathrm{i}} -\abs{V_\kappa^2} + \abs{V_\mathrm{cyc}}=\abs{E_g} - \abs{V_\mathrm{i}} -2\abs{E_\kappa^2} + \abs{V_\mathrm{cyc}}
  \end{equation}
  and 
  \[\Val(\Gamma_\red)=\Val(\Gamma).\] 
  Moreover, we have
  \begin{equation}\label{red e a i rel} 
    \sum_{e\in E_g^\mathrm{red}} t(e) = \abs{V_\mathrm{o}^t}, \quad\sum_{e\in E_g^\mathrm{red}} a(e) = \abs{V_\mathrm{o}^{0\mathrm{tr}}}, \quad \sum_{e\in E_g^\mathrm{red}} l(e)=\abs{E_g}, \quad \sum_{e\in E_g^\mathrm{red}} i(e)=2ip.
  \end{equation}
\end{lemma}
Second, we estimate the value of each graph by bounding the size of each of the reduced \(G\)-edges entrywise and the summations trivially. 
\begin{lemma}\label{lemma Itwo}
  For each av-/iso-graph \(\Gamma\in\cG\) with the selected vertex sets \(V_\mathrm{o}^{0\mathrm{tr}},V_\mathrm{o}^{t}\) we have \(\abs{\Val(\Gamma_\red)}\prec \ItwoEst(\Gamma)\) with 
  \begin{equation}\label{Itwo eq def}
    \begin{split}
      \ItwoEst(\Gamma) &:= \Lambda_+^{\abs{V_\mathrm{o}^{0\mathrm{tr}}}}\Pi_+^{\abs{V_\mathrm{o}^t}} \rho^{2ip\vee(\abs{V_\mathrm{i}}+2\abs{E_\kappa^2}-\abs{V_\mathrm{o}})} N^{\abs{V_\mathrm{i}}+\abs{E_\kappa^2}+\abs{E_\kappa^3}/2-\abs{V_\mathrm{o}}/2-\delta^{\ge 4}}  \\
      &\quad \times K^{\abs{V_\mathrm{o}}-\abs{V_\mathrm{i}}-2\abs{E_\kappa^2}+\abs{V_\mathrm{cyc}^{o=0}}+\abs{V_\mathrm{cyc}^{o=1}}/2},
    \end{split}
  \end{equation}
  where 
  \[\delta^{\ge 4}:=\sum_{e\in E_\kappa}\Bigl(\frac{d_g(e)}{2}-2\Bigr)_+\]
\end{lemma}
Finally, in the third step we improve upon the entrywise estimate as by estimating summations corresponding to some \(V_\kappa^{\ge3}\)-vertices more effectively, using a Schwarz inequality followed by the Ward identity \(GG^\ast= \Im G/\eta\). 
\begin{lemma}\label{lemma Ithree}
  For each av-graph \(\Gamma\in\cG\) with the selected vertex sets \(V_\mathrm{o}^{0\mathrm{tr}},V_\mathrm{o}^{0\mathrm{tr}}\)  we have \(\abs{\Val(\Gamma_\red)}\prec \IthreeEst(\Gamma)\) with 
  \begin{subequations}
    \begin{equation}\label{deg 3 lemma av case}
      \begin{split}
        \IthreeEst(\Gamma) &:= \Lambda_+^{\abs{V_\mathrm{o}^{0\mathrm{tr}}}} \Pi_+^{\abs{V_\mathrm{o}^t}}\rho^{2ip\vee(\abs{V_\mathrm{i}}+2\abs{E_\kappa^2}+\abs{E_\kappa^3}-\abs{V_\mathrm{o}})} N^{\abs{V_\mathrm{i}}+\abs{E_\kappa^2}+\abs{E_\kappa^3}/2-\abs{V_\mathrm{o}}/2-\delta^{\ge 4}} \\
        &\quad \times K^{\abs{V_\mathrm{o}}-\abs{V_\mathrm{i}}-2\abs{E_\kappa^2}-\abs{E_\kappa^3}/2+\abs{V_\mathrm{cyc}^{o=0}}+\abs{V_\mathrm{cyc}^{o=1}}/2}
      \end{split}
    \end{equation}
    and for each iso-graph \(\Gamma\in\cG\) we have \(\abs{\Val(\Gamma_\red)}\prec \IthreeEst(\Gamma)\) with
    \begin{equation}\label{deg 3 lemma iso case}
      \begin{split}
        \IthreeEst(\Gamma) &:= \Lambda_+^{\abs{V_\mathrm{o}^{0\mathrm{tr}}}}\Pi_+^{\abs{V_\mathrm{o}^t}} \rho^{2ip\vee(\abs{V_\mathrm{i}}+2\abs{E_\kappa^2}+\abs{E_\kappa^3}-\abs{V_\mathrm{o}})} N^{\abs{V_\mathrm{i}}+\abs{E_\kappa^2}+\abs{E_\kappa^3}/2-\abs{V_\mathrm{o}}/2-\delta^{\ge 4}} \\
        &\quad \times K^{\abs{V_\mathrm{o}}-\abs{V_\mathrm{i}}-2\abs{E_\kappa^2}-\abs{E_\kappa^3}/2+\abs{V_\mathrm{cyc}^{o=0}}+\abs{V_\mathrm{cyc}^{o=1}}/2-\bigl(p-\abs{E_\kappa^2}+\abs{V_\cyc}-\delta^{\ge 4}\bigr)_+}
      \end{split}
    \end{equation}
  \end{subequations}
\end{lemma}
Before proving Lemmata~\ref{lemma reduction}--\ref{lemma Ithree} we conclude the proof of Proposition~\ref{prop value est}. 

\begin{proof}[Proof of Proposition~\ref{prop value est}]
  The proof of Proposition~\ref{prop value est} distinguishes several cases. For the averaged bound we consider the two cases \(a=t=0\) and \(a+t=o>0, \abs{V_\mathrm{i}\cap V_\mathrm{o}}=2(o-1)p\) separately, with the remaining case \(o>0,\abs{V_\mathrm{i}}=2op\) being discussed in Section~\ref{sec a last case}, while for the isotropic bound we consider the cases \(o\ge 0, \abs{V_\mathrm{i}\cap V_\mathrm{o}}=2op\) and \(o>0, \abs{V_\mathrm{i}\cap V_\mathrm{o}}=2(o-1)p\) separately. 
  
  We first consider the \(o=0\) case of the averaged bound where we obtain from Lemma~\ref{lemma reduction},~\eqref{deg 3 lemma av case} with \(V_\mathrm{o}=\emptyset\) from~\ref{a rule}, \(\abs{V_\mathrm{cyc}}\le \abs{E_\kappa^2}\) from~\ref{Va Vcyc sum claim} and \(\abs{V_\mathrm{i}}=2p(l-1)\) from~\ref{number of internal vertices} that 
  \[ 
  \begin{split}
    \abs{\Val(\Gamma)} &\prec \rho^{ \abs{V_\mathrm{i}}+2\abs{E_\kappa^2}+\abs{E_\kappa^3}} N^{\abs{V_\mathrm{i}}+\abs{E_\kappa^2}+\abs{E_\kappa^3}/2} K^{-\abs{V_\mathrm{i}} -\abs{E_\kappa^2} -\abs{E_\kappa^3}/2 }\\
    & =\rho^{2(l-1)p+2\abs{E_\kappa^2}+\abs{E_\kappa^3}} N^{2lp} K^{-2p(l-1)} N^{-2p+\abs{E_\kappa^2}+\abs{E_\kappa^3}/2} K^{-\abs{E_\kappa^2} -\abs{E_\kappa^3}/2 } \\
    &\lesssim \rho^{2p(l+1)} N^{2lp} K^{-2lp},
  \end{split} \]
  where in the last step we used \(K\lesssim N\rho^2\) due to \(\eta=\min_k\eta_k\lesssim\max_k\rho_k=\rho\) and \(\abs{E_\kappa^2}+\abs{E_\kappa^3}/2\le \abs{E_\kappa}\le 2p\) from~\ref{number of kappa edges} and~\ref{degree of kappa vertices}. 
  
  Next, we consider the \(\abs{V_\mathrm{o}\cap V_\mathrm{i}}=2op\) case of isotropic bound, where we obtain from Lemma~\ref{lemma reduction},~\eqref{deg 3 lemma iso case}, and \(\abs{V_\mathrm{i}}=2p(o+b-1)\) from~\ref{no internal} that 
  \[ \begin{split}
    \abs{\Val(\Gamma)} & \prec \Lambda_+^{2ap}\Pi_+^{2tp}\rho^{2ip\vee(2p(b-1) + 2 \abs{E_\kappa^2} + \abs{E_\kappa^3}) } N^{p(o+2b)} N^{\abs{E_\kappa^2}+\abs{E_\kappa^3}/2-2p-\delta^{\ge 4}} \\
    &\qquad \times K^{p(1-2b)-\abs{E_\kappa^2}-\abs{E_\kappa^3}/2 + \delta^{\ge 4}} \\
    &\lesssim \Lambda_+^{2ap}\Pi_+^{2tp} \rho^{2ip\vee 2p(b+1)} N^{p(o+2b)} K^{-p(1+2b)} 
  \end{split} \]
  again using \(K\lesssim N\rho^2\) and \(\abs{E_\kappa^2}+\abs{E_\kappa^3}/2\le \abs{E_\kappa}\le 2p\).
  
  Next, we consider the \(\abs{V_\mathrm{o}\cap V_\mathrm{i}}=2(o-1)p, o>0\) case of both the averaged bound, and the isotropic bound where we similarly obtain (from estimating \((\ldots)_+\ge 0\) for the iso-graphs)
  \begin{equation}\label{o-1 case} \begin{split}
    \abs{\Val(\Gamma)} &\prec \Lambda_+^{\abs{V_\mathrm{o}^{0\mathrm{tr}}}}\Pi_+^{\abs{V_\mathrm{o}^t}} \rho^{i'} N^{\abs{V_\mathrm{i}}+\abs{E_\kappa^2}+\abs{E_\kappa^3}/2-\abs{V_\mathrm{o}}/2} K^{\abs{V_\mathrm{o}}-\abs{V_\mathrm{i}} -2\abs{E_\kappa^2} -\abs{E_\kappa^3}/2 +\abs{V_\mathrm{cyc}^{o=0}}+\abs{V_\mathrm{cyc}^{o=1}}/2}\\
    & = \Lambda_+^{\abs{V_\mathrm{o}^{0\mathrm{tr}}}} \Pi_+^{\abs{V_\mathrm{o}^t}} \rho^{i'} N^{p(o+2b-1)+\abs{E_\kappa^2}+\abs{E_\kappa^3}/2-\abs{V_\mathrm{o}\cap V_\kappa}/2} \\
    &\qquad \times K^{\abs{V_\mathrm{o}\cap V_\kappa}-2pb -2\abs{E_\kappa^2} -\abs{E_\kappa^3}/2 +\abs{V_\mathrm{cyc}^{o=0}}+\abs{V_\mathrm{cyc}^{o=1}}/2}\\
    & \le \Lambda_+^{2pa}\Pi_+^{2pt}\rho^{i'} N^{p(o+2b)} K^{-(2b+1)p } \Bigl(\frac{K}{N}\Bigr)^{p+\abs{V_\mathrm{o}\cap V_\kappa}/2 -\abs{E_\kappa^2} -\abs{E_\kappa^3}/2} \\
    &\lesssim \Lambda_+^{2pa}\Pi_+^{2pt} \rho^{2ip\vee 2(b+1)p} N^{p(o+2b)} K^{-(2b+1)p },
  \end{split} \end{equation}
  with 
  \[ \begin{split}
    i' &= 2ip\vee(\abs{V_\mathrm{i}}+2\abs{E_\kappa^2}+\abs{E_\kappa^3}-\abs{V_\mathrm{o}})= 2ip \vee\Bigl( 2bp + 2\abs{E_\kappa^2}+\abs{E_\kappa^3} - \abs{V_\mathrm{o}\cap V_\kappa} \Bigr).
  \end{split} \]
  Here we used~\ref{number of G edges} and~\ref{number of internal vertices}/\ref{no internal} and \(V_\mathrm{o}\subset V_\mathrm{i}\cup V_\kappa^2\) (since by definition \(V_\mathrm{o}\) are degree-\(2\) vertices, while \(V_\mathrm{e}=\emptyset\) due to~\ref{no external} in the averaged case and \(d_g(v)=1,v\in V_\mathrm{e}\) due to~\ref{number of external} in the isotropic case and \(V_\kappa^{\ge 3}\)-vertices have degree at least \(3\) by~\ref{degree of kappa vertices}) in the equality. Furthermore, we used~\ref{Va Vcyc sum claim} in the first inequality, and~\ref{a rule} in the second inequality, and \(K/N\lesssim\rho^2\) and~\ref{a rule} in the final step.  
\end{proof}

\subsubsection{Graph reduction: Proof of Lemma~\ref{lemma reduction}}
Since for \(d_g((uv))=2\) we have \(\kappa^{(uv)}_{ab}=1\) or \(\kappa^{(uv)}_{ab}=\sigma\) for all \(a,b\) due to~\eqref{kappa mat def0} (using Assumption~\ref{assump w2}) it is possible to write (with potential transpositions) the summation over \(a_v\) for \(v\in V_\mathrm{i}\cup V_\kappa^2\) as matrix products which are then associated with edges of the reduced graph \(\Gamma_\mathrm{red}\). In this way \(G\)-chains \((v_1v_2),\dots,(v_{k-1}v_k)\in E_g\) with \(v_2,\ldots,v_{k-1}\in V_\kappa^2\cup V_\mathrm{i}\) and \(v_1,v_k\not \in V_\kappa^2\cup V_\mathrm{i}\) are reduced to the edge \((v_1v_k)\in E_g^\mathrm{red}\), and \(G\)-cycles \((v_1v_2),\dots,(v_k v_1)\in E_g\) with \(v_1,\ldots,v_k\in V_\kappa^2\cup V_\mathrm{i}\) are reduced to isolated loops which we represent by the vertex \(v_1\in V_\mathrm{cyc}\) and the loop-edge \((v_1v_1)\in E_g^\mathrm{red,cyc}\subset E_g^\mathrm{red}\). For each cycle of length \(k\) we arbitrarily pick one of the \(k\) possible reductions since they are all equivalent. 

The first relation in~\eqref{number edges Egred} follows trivially since for each of the carried out summations corresponding to \(V_\kappa^2\cup V_\mathrm{i}\) the number of \(G\)-edges is reduced by one with the exception that for cycles the last index is kept in \(V_\mathrm{cyc}\). The second relation in~\eqref{number edges Egred} is a direct consequence of~\ref{perfect matching}. Next, the claim~\eqref{red e a i rel} follows from~\ref{number of G edges} and by noting that the definition of \(a(e),t(e)\) is consistent with the counting of \(t/0\mathrm{tr}\)-vertices in \(\Gamma\). This concludes the proof of Lemma~\ref{lemma reduction}.

\subsubsection{Entrywise bound: Proof of Lemma~\ref{lemma Itwo}}
For edges in the reduced graph we use the bound from the following lemma. Note that \(o(e)\le l(e)\) for cycles \(e\) and \(o(e)\le l(e)-1\) for chains \(e\) and therefore 
the exponents of \(K\) below are guaranteed to be non-positive.
\begin{lemma}\label{degree two lemma}
  For \(e\in E_g^{\red,\cyc}\) we have the averaged bound
  \begin{subequations}
    \begin{equation}\label{Improved 2 reduced edge est}
      \abs{\braket{\cG^{e}}} \prec \Lambda_+^{a(e)}\Pi_+^{t(e)} \rho^{i(e)\vee (l(e)-o(e)+\bm 1[0<o(e)<l(e)])} N^{l(e)-\frac{o(e)}{2}-1} K^{o(e)-l(e)+\bm 1(o(e)=0) + \frac{\bm1(o(e)=1)}{2}}
    \end{equation}
    and for \(e\in E_g^{\red}\setminus E_g^{\red,\cyc}\) the isotropic bound
    \begin{equation}\label{Improved 2 reduced edge est iso}
      \abs{\braket{\bm v,\cG^{e} \bm w}} \prec \norm{\bm v} \norm{\bm w} \Lambda_+^{a(e)}\Pi_+^{t(e)} \rho^{i(e)\vee (l(e)-o(e)-\bm1[o(e)=l(e)-1])}  N^{l(e)-\frac{o(e)}{2}-1}  K^{o(e)-l(e)+1}
    \end{equation}
  \end{subequations}
  for any two deterministic vectors \({\bm v}, {\bm w}\). Moreover, the same bounds hold true if within the chain \(\cG\) absolute values of resolvents \(\abs{G(z)}\) appear in addition, to \((\Im G)^{(t)},(G^\ast)^{(t)},(G)^{(t)}\). 
\end{lemma}
\begin{remark}\label{remark loss}
  The estimates~\eqref{Improved 2 reduced edge est}--\eqref{Improved 2 reduced edge est iso} are designed to take advantage of the asymptotic orthogonality vertices. Indeed, using that a posteriori we will show that \(\Lambda_++\Pi_+\prec 1\) in the bulk, \(\rho\sim 1\), both inequalities essentially depend on the number of orthogonality-vertices as \((K/\sqrt{N})^o \sim (\sqrt{N}\eta)^o\) (ignoring some \(K\) factors in~\eqref{Improved 2 reduced edge est} for \(o=0, 1\)). Therefore as long as \(\eta\ll N^{-1/2}\) the orthogonality helps and our bounds do exploit this effect. However, for \(\eta\gg N^{-1/2}\) it is better to use~\eqref{Improved 2 reduced edge est}--\eqref{Improved 2 reduced edge est iso} by simply ignoring the asymptotic orthogonality, i.e.\ choosing \(\mathfrak{a}=\mathfrak t=\emptyset\). We will not need to use this improvement in the main body of the proof of Theorem~\ref{chain G underline theorem}, but it will be used when we remove simplification (iii) in Section~\ref{sec large eta}.
\end{remark}
Using Lemma~\ref{degree two lemma}, the proof of which we defer to the the end of the subsection, we now conclude the proof of Lemma~\ref{lemma Itwo}. From Lemma~\ref{lemma reduction} we obtain \(\Val(\Gamma)=\Val(\Gamma_\mathrm{red})\) with \(\Val(\Gamma_\mathrm{red})\) as in~\eqref{Val reduced def}. By estimating \(\abs{\kappa^{(uv)}_{ab}}\lesssim 1\) and \(\cG\) via Lemma~\ref{degree two lemma} we obtain from~\eqref{red e a i rel}, 
\[ \prod_{(uv)\in E_\kappa^{\ge 3}}\biggl( \sum_{a_u,a_v} N^{-d_g((uv))/2} \biggr)=\prod_{k\ge 3}\Bigl(N^{2-k/2}\Bigr)^{\abs{E_\kappa^k}}= N^{\abs{E_\kappa^3}/2-\delta^{\ge 4}},\]
and \(\abs{V_\mathrm{cyc}}=\abs{E_g^{\mathrm{red},\mathrm{cyc}}}\) that 
\begin{equation}\label{eq val entrywise bd}
  \begin{split}
    \abs{\Val(\Gamma)} &\prec \Lambda_+^{\abs{V_\mathrm{o}^{0\mathrm{tr}}}}\Pi_+^{\abs{V_\mathrm{o}^t}} \rho^{ i'' } N^{-\abs{E_\kappa^2}+\abs{V_\mathrm{cyc}}+\abs{E_\kappa^3}/2+\abs{E_g}-\abs{V_\mathrm{o}}/{2}-\abs{E_g^\mathrm{red}}-\delta^{\ge 4}} \\ 
    &\qquad \times K^{\abs{V_\mathrm{o}}-\abs{E_g} + \abs{E_g^\mathrm{red}}-\abs{V_\mathrm{cyc}}+\abs{V_\mathrm{cyc}^{o=0}}+\abs{V_\mathrm{cyc}^{o=1}}/2} \\
    & \lesssim \Lambda_+^{\abs{V_\mathrm{o}^{0\mathrm{tr}}}}\Pi_+^{\abs{V_\mathrm{o}^t}} \rho^{ 2ip\vee(\abs{V_\mathrm{i}}+2\abs{E_\kappa^2}-\abs{V_\mathrm{o}})} N^{\abs{E_\kappa^2}+\abs{E_\kappa^3}/{2}+\abs{V_i}-\abs{V_\mathrm{o}}/{2}-\delta^{\ge 4}} \\
    &\qquad \times K^{\abs{V_\mathrm{o}}-\abs{V_i}-2\abs{E_\kappa^2}+\abs{V_\mathrm{cyc}^{o=0}}+\abs{V_\mathrm{cyc}^{o=1}}/{2}},
  \end{split} 
\end{equation}
where we used~\eqref{number edges Egred} in the second step. Here we counted the factors of \(\rho\) as 
\begin{equation}\label{counting of rhos I two}
  \begin{split}
    i'' :={}& \sum_{e\in E_g^\red} i(e) \vee \begin{cases}
      l(e)-o(e)+\bm1(0<o(e)<l(e)),& e\in E_g^{\red,\cyc},\\ l(e)-o(e)-\bm1(o(e)=l(e)-1), & e\in E_g^{\red}\setminus E_g^{\red,\cyc},
    \end{cases}\\
    \ge{}& \sum_{e\in E_g^\red} i(e) \vee [l(e)-o(e)-\bm1(e\not\in E_g^{\red,\cyc})] \\
    \ge{}& 2ip \vee\Bigl( \abs{E_g}-\abs{V_\mathrm{o}} - \abs{E_g^\red} + \abs{V_\cyc} \Bigr) = 2ip \vee \Bigl(\abs{V_\mathrm{i}}+2\abs{E_\kappa^2}-\abs{V_\mathrm{o}}\Bigr)
  \end{split}
\end{equation}
due to~\eqref{number edges Egred}, completing the proof of Lemma~\ref{lemma Itwo}.

\begin{proof}[Proof of Lemma~\ref{degree two lemma}]
  We actually prove a slightly more general bound which allows for chains \(\cG^{e}=G_1B_1\cdots G_l B_l\) with 
  \[G_k\in\set{G(z_k),G(z_k)^\ast,\Im G(z_k),\abs{G(z_k)},(G(z_k))^t,(G(z_k)^\ast)^t,(\Im G(z_k))^t,\abs{G(z_k)}^t},\] 
  i.e.\ including factors of the form \(\abs{G}=\sqrt{G^\ast G}=\sqrt{GG^\ast}\). Within the proof we will repeatedly use~\eqref{ali relations} which implies \(o(e)\le l(e)\) for \(e\in E_g^{\mathrm{red},\mathrm{cyc}}\) and \(o(e)\le l(e)-1\) for \(e\in E_g^{\red}\setminus E_g^{\red,\cyc}\). We prove~\eqref{Improved 2 reduced edge est} by distinguishing several cases depending on the parameters \(o(e)\) and \(l(e)\) and a new parameter \(c(e)\) counting the number of \emph{alternating chains} associated with \(e\) defined as follows. For any \(e\in E_g^\red\) we consider the original chain or cycle in \(\Gamma\) that was reduced to \(e\). The alternating chains associated with \(e\) are the maximal subchains of these original chain/cycle  with  internal vertices from \(V_\mathrm{o}\) and at least one \(V_\mathrm{o}\)-vertex.
  For example, if \(e\in E_g^\red\) was the reduction of the cycle \(\braket{(GA)(\Im G A)(B G^\ast)(G)(AG^\ast)(\Im G)^t}\) then the alternating chains associated with \(e\) are \((GA)(\Im G A)\) and \((G)(AG^\ast)(\Im G)^t\). By maximality, \(o(e)\), the number of \(V_\mathrm{o}\)-vertices in the original chain/cycle that has been reduced to \(e\) is equal to the total number of \(V_\mathrm{o}\) vertices in the alternating chains associated with \(e\). In particular, \(c(e) \le o(e)\).
  
  \subsubsection*{Averaged bound for \(o(e)=0\)} 
  In the case without alternating chains, i.e.\ for \(o(e)=0\) we simply split off any \(G\)-factor by Cauchy-Schwarz and obtain 
  \[ 
  \begin{split}
    \abs{\braket{G_1 B_1 G_2 B \cdots G_l B_l}}&\le \sqrt{\braket{G_1 \abs{B_1}^2 G_1^\ast}\braket{G_2B_2 \cdots G_l \abs{B_l}^2 G_l^\ast \cdots B_2^\ast G_2^\ast}} \\
    &\prec \frac{\rho}{\eta^{l-1}} \le \rho^l N^{l-1} K^{1-l}.
  \end{split} \]
  Here, and frequently in the remaining proof we used the \emph{Ward identity} \(G(z)G(z)^\ast=\Im G(z)/\Im z\) and the norm bounds \(\norm{G}\lesssim 1/\eta\), \(\norm{B_k}\lesssim1\).  
  
  \subsubsection*{Averaged bound for \(a(e)=l(e)\)} 
  For \(\cG^e = G_1 B_1 G_2 B_2 \cdots G_l B_l \) we use spectral decomposition to write 
  \[
  \begin{split}
    \braket{\cG^e} = N^{-1} \sum_{a_1\ldots a_l} \braket{\bm u^{(1)}_{a_1}, B_1 \bm u^{(2)}_{a_2}}\cdots \braket{\bm u^{(l)}_{a_l}, B_l \bm u_{a_1}^{(1)}} p_{a_1}^{(1)}\cdots p_{a_l}^{(l)},
  \end{split}
  \]
  where \(p^{(k)}_a=(\lambda_a-z_k)^{-1},(\lambda_a-\ov{z_k})^{-1}, \Im (\lambda_a-z_k)^{-1},\abs{\lambda_a-z_k}^{-1}\) depending on whether \(G_k=G,G^\ast,\Im G,\abs{G}\), and \(\bm u^{(k)}_a\in\set{\bm u_a,\ov{\bm u_a}}\), depending on whether \(G_k\) is transposed or not. By additional averaging using the analogue of~\eqref{tricka}, Cauchy-Schwarz and the high-probability bounds
  \begin{equation}\label{eig sum bounds}
    \sum_{a} \frac{1}{\abs{\lambda_a-z}} \lesssim N \log N,\qquad \sum_{a} \abs*{\Im\frac{1}{\lambda_a-z} }\lesssim \rho(z) N ,
  \end{equation}
  from rigidity~\eqref{eq:rig} it follows that 
  \begin{equation}\label{cG l=a>=2} 
    \begin{split}
      \abs{\braket{\cG^e}} &\lesssim \frac{1}{N} \sum_{\substack{a_k\in[N]\\ k\in[l]}} \abs[\big]{p_{a_1}^{(1)}}\cdots \abs[\big]{p_{a_l}^{(l)}} \frac{1}{L^l}\sum_{\substack{\abs{b_k-a_k}\le L\\ k\in[l]}} \abs{\braket{\bm u_{b_1}^{(1)}, B_1 \bm u_{b_2}^{(2)}}}\cdots \abs{\braket{\bm u_{b_l}^{(l)}, B_l \bm u_{b_1}^{(1)}}} \\
      & \lesssim  \frac{1}{N} \sum_{\substack{a_k\in[N]\\ k\in[l]}} \abs[\big]{p_{a_1}^{(1)}}\cdots \abs[\big]{p_{a_l}^{(l)}} \sqrt{\frac{1}{L^2}\sum_{\abs{a_1-b_1}\le L}\sum_{\abs{a_2-b_2}<L} \abs{\braket{\bm u_{b_1}^{(1)},B_1\bm u_{b_2}^{(2)}}}^2 }\cdots\\
      &\qquad\qquad\qquad\qquad\qquad\qquad\times \sqrt{\frac{1}{L^2}\sum_{\abs{a_l-b_l}\le L}\sum_{\abs{a_1-b_1}<L} \abs{\braket{\bm u_{b_l}^{(l)},B_l\bm u_{b_1}^{(1)}}}^2 }\\
      &\prec \Lambda_+^{a}\Pi_+^{t} \rho^{i} N^{l/2-1},
    \end{split}
  \end{equation}
  where \(\log N\) factors have been incorporated into the \(\prec\) notation in the ultimate inequality. 
  
  \subsubsection*{Averaged bound for \(o(e)=1\)} 
  By cyclicity we may assume \(\cG^e=G_1 B_1 G_2 B_l\cdots G_l B_l\) is such that the index between \(G_1B_1\) and \(G_2B_2\) is the asymptotic orthogonality index and estimate
  \[ 
  \begin{split}
    \abs{\braket{G_1 B_1 G_2 B_2 \cdots G_l B_l}}&\le \sqrt{\braket{G_1 B_1 G_2 G_2^\ast B_1^\ast G_1^\ast} \braket{B_2 G_3 \cdots G_l \abs{B_l}^2 G_l^\ast \cdots G_3^\ast B_2^\ast}} \\
    &\prec \eta^{-1} \Lambda_+ \rho N^{l-5/2}\rho^{l-2} K^{5/2-l}\le \Lambda_+ \rho^l N^{l-3/2} K^{3/2-l},
  \end{split}  \]
  from the \(o(e)=0\) and \(o(e)=l(e)\) cases, and using the Ward identity.
  
  \subsubsection*{Averaged bound for \(2\le o(e)<l(e)\) and \(c(e)=1\)} 
  For this case we may assume by cyclicity that 
  \[\cG^e=G_1 B_1 \cdots G_o B_o G_{o+1}B_{o+1}\cdots G_l B_l\] 
  such that the summations between \(G_1\) and \(G_o\) correspond to orthogonality indices. Here we make use of the inequality 
  \begin{equation}\label{XYZ claim}
    \abs{\braket{XYZ}} \le \Bigl[\braket{X^\ast X (YY^\ast)^{1/2}}\braket{ZZ^\ast (Y^\ast Y)^{1/2}}\Bigr]^{1/2}
  \end{equation}
  for arbitrary matrices \(X,Y,Z\) which follows from singular value decomposition of \(Y=USV^\ast\) and Cauchy-Schwarz in the form 
  \[ \begin{split}
    \abs{\braket{XYZ}}^2 &= \abs{\braket{X U\sqrt{S}\sqrt{S}V^\ast Z}}^2 \\
    &\le\braket{XUSU^\ast X^\ast}\braket{Z^\ast V S V^\ast Z } =\braket{X^\ast X (YY^\ast)^{1/2} }\braket{Z Z^\ast  (Y^\ast Y)^{1/2}  }.
  \end{split} \]
  By applying~\eqref{XYZ claim} with \(X=G_1 B_1,Y=G_{2},Z=B_2 G_3\cdots B_o G_{o+1} B_{o+1}\cdots G_l B_l\) we obtain 
  \[ 
  \begin{split}
    &\abs{\braket{G_1 B_1 \cdots G_o B_o G_{o+1} \cdots G_l B_l}}
    \\& \le \sqrt{\braket{B_1^\ast G_1^\ast G_1 B_1 \abs{G_2}} \braket{\abs{G_2}^{1/2} B_2 G_3 \cdots B_o G_{o+1}\cdots G_l \abs{B_l}^2 G_l^\ast \cdots  G_{o+1 }^\ast B_o^\ast \cdots G_3^\ast B_2^\ast \abs{G_2}^{1/2} }}\\
    & \lesssim \frac{1}{\eta^{l-o}} \Bigl[\braket{B_1^\ast \Im G_1 B_1 \abs{G_2}} \braket{\abs{G_2} B_2 G_3 \cdots B_o \Im G_{o+1} B_o^\ast \cdots G_3^\ast B_2^\ast }\Bigr]^{1/2} \\ 
    & \prec \Lambda_+^a\Pi_+^t \rho^{l-o+1+i_{2\ldots o}} N^{l-o/2-1} K^{o-l} ,
  \end{split} 
  \]
  where \(i_{2\cdots o}\) is the number of \(\Im G\)'s among \(G_2,\ldots G_o\), and we used the previously considered \(o(e)=l(e)\) case in the last step.  
  
  \subsubsection*{Averaged bound for \(2\le o(e)<l(e)\) and \(c(e)\ge 2\)}
  For at least two alternating chains, \(c(e)\ge 2\), we may write by cyclicity \(\braket{\cG^e}=\braket{\cG^{e_1}\cdots \cG^{e_{c(e)}}}\) for 
  \[\cG^{e_j}=G_{j,1}B_{j,1}G_{j,2}\cdots B_{j,o_j} G_{j,o_j+1}B_{j,o_j+1} \cdots G_{j,l_j} B_{j,l_j},\]
  for some \(1\le o_j\le l_j -1\) such for each \(\cG^{e_j}\) the first \(o_j\) internal summation indices are orthogonality indices. By Cauchy-Schwarz it follows that 
  \begin{equation}\label{general av case}
    \begin{split}
      \abs{\braket{\cG^e}} &\le \frac{1}{N}\sqrt{\prod_{j\in[c(e)]} \Tr \cG^{e_j}(\cG^{e_j})^\ast }\\
      &\lesssim \frac{1}{N}\prod_{j\in[c(e)]} \frac{1}{\eta^{l_j-o_j}} 
      \sqrt{\Tr \Im G_{j,1} B_{j,1} G_{j,2} \cdots G_{j,o_j} B_{j,o_j} 
      \Im G_{j,o_j+1} B_{j,o_j}^\ast G_{j,o_j}^\ast \cdots G_{j,2}^\ast B_{j,1}^\ast } \\
      &\prec \frac{1}{N}\prod_{j\in[c(e)]} \frac{N^{o_j/2}\rho^{i_j+1}}{\eta^{l_j-o_j}} \Lambda_+^{a_j}\Pi_+^{t_j} \\
      &\le \Lambda_+^{\sum_j a_j}\Pi_+^{\sum_j t_j} \rho^{\sum_j(l_j-o_j+i_j+1)} N^{\sum_j (l_j - o_j/2) -1 } K^{\sum_j (o_j-l_j)} ,
    \end{split} 
  \end{equation}
  where \(i_j\) denotes the number of \(\Im G\)'s among \(G_{j,2},\ldots, G_{j,o_j}\), and we used the previously discussed \(o(e)=l(e)\) case in the third inequality. This concludes the proof of~\eqref{Improved 2 reduced edge est}.
  
  \subsubsection*{Isotropic bound for \(o(e)=l(e)-1\)}
  The claimed bound is trivial if \(l(e)=1\) (and hence \(o(e)=0\)). Otherwise for \(l(e)\ge 2\) we estimate 
  \begin{equation}
    \begin{split}
      \abs{\braket{\bm v,\cG^e \bm w}} &= \abs{\braket{\bm v,G_1 B_1 G_2 B_2 \cdots B_{l-1} G_l,\bm w}} \\
      &\lesssim \sum_{\substack{a_k\in[N]\\k\in[l]}} \abs{p_{a_1}^{(1)}}\cdots \abs{p_{a_l}^{(l)}} \abs{ \braket{\bm v,\bm u_{a_1}^{(1)}} } \abs{\braket{\bm u_{a_1}^{(1)},B_1\bm u_{a_2}^{(2)}}}\cdots\abs{\braket{\bm u_{a_{l-1}}^{(l-1)},B_{l-1}\bm u_{a_l}^{(l)}}}\abs{\braket{\bm u_{a_l}^{(l)},\bm w}}\\
      &\prec \frac{1}{N} \sum_{\substack{a_k\in[N]\\ k\in[l]}} \abs[\big]{p_{a_1}^{(1)}} \cdots \abs[\big]{p_{a_l}^{(l)}} \frac{1}{L^{l}}\sum_{\substack{\abs{b_k-a_k}\le L\\ k\in[l]}} \abs{\braket{\bm u_{b_1}^{(1)},B_1\bm u_{b_2}^{(2)}}}\cdots\abs{\braket{\bm u_{b_{l-1}}^{(l-1)},B_{l-1}\bm u_{b_l}^{(l)}}}\\
      &\prec \Lambda_+^{a}\Pi_+^{t} \rho^{i}   N^{l/2-1/2} ,
    \end{split}
  \end{equation}
  using delocalisation \(\abs{\braket{\bm u_a,\bm v}}+\abs{\braket{\ov{\bm u_a},\bm v}}\prec N^{-1/2}\) for any deterministic \(\bm v\) with \(\norm{\bm v}\lesssim 1\), by the isotropic law in~\eqref{eq:oldlocal}, in the second inequality. 
  
  \subsubsection*{Isotropic bound for \(o(e)\le l(e)-2\)}
  We decompose \(\cG^e=\cG^{e_1}\cdots \cG^{e_{k}}\) such that each of \(\cG^{e_2},\ldots,\cG^{e_{k-1}}\) begins with a new alternating chain followed (potentially) by further \(G\)'s, \(\cG^{e_1}\) either begins with an alternating chain, or is a chain without orthogonality indices, and \(\cG^{e_k}\) is either an alternating chain or a chain without orthogonality indices. For example, by brackets denoting the decomposition, we would separate
  \[ \braket{\bm v,(GB_1G^\ast B_2 (\Im G)^t B_3) (GB_4G^t B_5)(G^\ast B_6)\bm w}\] 
  if the indices associated with \(B_1,B_2,B_4\) are orthogonality indices, and estimate
  \[\abs{\braket{\bm v,\cG^e \bm w}}\le \Bigl[\braket{\bm v,\cG^{e_1}(\cG^{e_1})^\ast \bm v} (\Tr \cG^{e_2} (\cG^{e_2})^\ast) \cdots ( \Tr \cG^{e_{k-1}} (\cG^{e_{k-1}})^\ast) \braket{\bm w,(\cG^{e_{k}})^\ast \cG^{e_{k}}\bm w} \Bigr]^{1/2}. \]
  For the two isotropic factors of length \(l_j\) with \(o_j\) orthogonality indices and \(i_j\) many \(\Im G\)'s we claim that 
  \begin{equation}\label{iso claim ei}
    \abs{\braket{\bm v,\cG^{e_j}(\cG^{e_j})^\ast\bm v}}  \prec \frac{N^{2l_j-o_j-1}\rho^{2i_j\vee 2(l_j-o_j)}\Lambda_+^{2a_j}\Pi_+^{2t_j}}{K^{2(l_j-o_j)-1}}
  \end{equation}
  which follows from 
  \[
  \begin{split}
    &\abs{\braket{\bm v,G_1 B_1 \cdots G_o B_o G_{o+1}B_{o+1} \cdots G_l G_l^\ast \cdots B_{o+1}^\ast G_{o+1}^\ast B_o^\ast G_o \cdots B_1^\ast G_1^\ast \bm v}} \\
    &\quad\lesssim \frac{\abs{\braket{\bm v,G_1 B_1 \cdots G_o B_o \Im G_{o+1} B_o^\ast G_o \cdots B_1^\ast G_1^\ast \bm v}}}{\eta^{2(l-o)-1}}\\
    &\quad \prec \frac{N^{2l-o-1}\rho^{2i_{1\cdots o}+2(l-o)}\Lambda_+^{2a}\Pi_+^{2t}}{K^{2(l-o)-1}},
  \end{split}
  \]
  where \(i_{1\cdots o}\) is the number of \(\Im G\)'s among \(G_1,\ldots,G_o\). For the tracial factors we have, as in~\eqref{general av case}, that 
  \begin{equation}\label{av claim ei} \Tr \cG^{e_j}(\cG^{e_j})^\ast \prec \frac{N^{2l_j-o_j}\rho^{2i_j\vee 2(l_j-o_j)}\Lambda_+^{2a_j}\Pi_+^{2t_j}}{K^{2(l_j-o_j)}}.\end{equation} 
  By combining~\eqref{iso claim ei}--\eqref{av claim ei} we obtain 
  \[\begin{split}
    \abs{\braket{\bm v,\cG^e\bm w}}&\prec \frac{K}{N} \prod_{j\in[k]} \frac{N^{l_j-o_j/2}\rho^{i_j\vee (l_j-o_j)}\Lambda_+^{a_j}\Pi_+^{t_j}}{K^{l_j-o_j}} \\
    &= \Lambda_+^{a}\Pi_+^{t} \rho^{i\vee(l-o)} N^{l-o/2-1} K^{o-l+1},
  \end{split} \]
  completing the proof of~\eqref{Improved 2 reduced edge est iso} also in this case.
\end{proof}

\subsubsection{Improved degree three estimate: Proof of Lemma~\ref{lemma Ithree}}
The proof of Lemma~\ref{lemma Ithree} consists of identifying improvements over the estimate given in Lemma~\ref{lemma Itwo} that relied solely on entrywise bounds for each individual \(\cG\)-factor. In order to quantify the improvement we distinguish the two different entrywise bounds in Lemma~\ref{lemma Itwo} as 
\begin{equation}\label{eq I2est min} \abs{\Val(\Gamma_\mathrm{red})} \prec \ItwoiEst(\Gamma)\wedge \ItwozEst(\Gamma),\end{equation}
where \(\ItwoiEst,\ItwozEst\) are defined as in~\eqref{Itwo eq def} but with \(\ItwoiEst\) having \(\rho\)-exponent \(2ip\), and \(\ItwozEst\) having \(\rho\)-exponent \(\abs{V_\mathrm{i}}+2\abs{E_\kappa^2}-\abs{V_\mathrm{o}}\). Note that \(\rho\lesssim 1\) and therefore the maximum in the exponent of \(\rho\) in~\eqref{Itwo eq def} corresponds to the minimum of \(\ItwoiEst,\ItwozEst\). 

Within the reduced graphs we call a subset \(E_\mathrm{Ward}\subset E_g^{\red}\setminus(E_g^{\red,\cyc}\cup\set{(vv)\given v\in V_\kappa^{\ge 3}})\) \emph{Wardable} if each subgraph \(\Gamma'\subset (V_\mathrm{e}\cup V_\kappa^{\ge 3},E_\mathrm{Ward})\) satisfies \(\min\set{ d_g^{\Gamma'}(v)\given v\in V_\kappa^{\ge 3}}\le 2\). The contribution of these 
Wardable edges will be estimated better than their trivial entrywise bound to obtain \(\IthreeEst\). We start with a simple alternative characterization of Wardable subsets (see~\cite[Lemma 4.5]{MR4134946} and~\cite{MR193025,MR266812}).
\begin{lemma}\label{lemma degen}
  A subset \(E_\mathrm{Ward}\) is Wardable if and only if there exists an ordering \(V_\kappa^{\ge 3}=\set{v_1,v_2,\ldots}\) such that the sequence of graphs \(\Gamma_0:=(V_\mathrm{e}\cup V_\kappa^{\ge 3},E_\mathrm{Ward})\), \(\Gamma_{k}:=\Gamma_{k-1}\setminus\set{v_k}\) satisfies \(d_g^{\Gamma_{k-1}}(v_k)\le 2\) for each \(k\ge 1\), where it is understood that \(\Gamma_k\) is obtained from \(\Gamma_{k-1}\) by removing \(v_k\) and all adjacent edges. 
\end{lemma}
\begin{proof}
  Suppose that \(E_\mathrm{Ward}\) is Wardable. Then by definition there exists \(v_1\) with \(d_g^{\Gamma_0}(v_1)\le 2\) and we obtain \(\Gamma_1\) which in turn contains some vertex \(v_2\) with \(d_g^{\Gamma_1}(v_2)\le 2\). Continuing inductively yields the desired ordering. 
  
  For the reverse implication let \(v_1,v_2,\ldots\) be the given ordering and let \(\Gamma'\) be arbitrary. Set \(k_{\min}:=\min\set{k\given v_k\in \Gamma'}\) so that \(\Gamma'\subset \Gamma_{k_{\min}-1}\) and consequently \(d_g^{\Gamma'}(v_{k_{\min}})\le d_g^{\Gamma_{k_{\min}-1}}(v_{k_{\min}})\le 2\). 
\end{proof}

Lemma~\ref{lemma Ithree} follows immediately from combining the following two statements (where for the iso-graphs we simply estimate \(\rho^{\abs{E_\mathrm{Ward}}}\le \rho^{\abs{E_\kappa^3}}\) in the definition of \(\IthreeEst(\Gamma)\) below):
\begin{enumerate}[label=(S\arabic*)]
  \item\label{step number of edges} For each av-graph \(\Gamma\) the reduced graph \(\Gamma_\red\) admits a Wardable set \(E_\mathrm{Ward}\) of size 
  \begin{subequations}
    \begin{equation}\label{First Ward}
      \abs{E_\mathrm{Ward}}\ge \abs{E_\kappa^3}.  
    \end{equation}
    and for each iso-graph \(\Gamma\) satisfying the reduced graph \(\Gamma_\red\) admits a Wardable set of size 
    \begin{equation}\label{Second Ward}
      \begin{split}
        \abs{E_\mathrm{Ward}}&\ge \abs{E_\kappa^3}  + \Bigl(2p - 2(\abs{E_\kappa^2}-\abs{V_\cyc})-\sum_{e\in E_\kappa} (d_g(e)-4)_+\Bigr)_+.    
      \end{split} 
    \end{equation}
  \end{subequations}
  \item\label{step number of gains} For any av- or iso-graph \(\Gamma\in\cG\) and a given Wardable set \(E_\mathrm{Ward}\) we have the improved estimates 
  \[\abs{\Val(\Gamma_\red)} \prec \IthreeiEst(\Gamma) \wedge \IthreezEst(\Gamma)\]
  with 
  \[ \IthreeiEst(\Gamma):= K^{-\abs{E_\mathrm{Ward}}/2} \ItwoiEst(\Gamma), \quad \IthreezEst(\Gamma):= \rho^{\abs{E_\mathrm{Ward}}}K^{-\abs{E_\mathrm{Ward}}/2} \ItwozEst(\Gamma).\]
\end{enumerate}

\begin{proof}[Proof of~\ref{step number of edges}] 
  We start  with two inequalities that will be proven later. Denoting the number of \(E_g^\mathrm{red}\)-edges between two subsets of 
  vertices \(V',V''\subset V\) by \(e_g(V',V'')\),
  we claim that for av-/iso graphs \(\Gamma\) we have
  \begin{subequations}\label{number V3 Vge3 adj}
    \begin{equation}\label{number V3 adj}
      e_g(V_\kappa^3,V_\mathrm{e})+e_g(V_\kappa^3,V_\kappa^{\ge 3}) \ge 3\abs{E_\kappa^3},
    \end{equation}  
    while for iso-graphs \(\Gamma\) we also have
    \begin{equation}\label{number Vge3 adj}
      e_g(V_\kappa^{\ge 3},V_\mathrm{e})+e_g(V_\kappa^{\ge 3},V_\kappa^{\ge 3}) \ge \sum_{e\in E_\kappa^{\ge 3}}d_g(e)+2p - 2(\abs{E_\kappa^2}-\abs{V_\cyc}). 
    \end{equation}
  \end{subequations}
  Armed with these inequalities, we first construct \emph{candidate sets} of edges within \(E_g^{\red}\setminus E_g^{\red,\cyc}\) which are not necessarily Wardable, and then iteratively remove  certain 
  edges to make the sets Wardable. 
  For the proof of \(\abs{E_\mathrm{Ward}}\ge \abs{E_\kappa^3}\)
  for both av- and iso-graphs we start with
  the candidate set consisting of all \(G\)-edges adjacent to \(V_\kappa^3\)-vertices.
  The size of this set is \( e_g(V_\kappa^3,V_\mathrm{e})+e_g(V_\kappa^3,V_\kappa^{\ge 3}) \).
  We remove  at most one edge adjacent to any \(v\in V_\kappa^3\), so that the at most two remaining edges are not loops. After doing so in arbitrary order 
  for all \(V_\kappa^3\)-vertices
  we obtain an edge set which is Wardable by construction. 
  Since the total number of removed edges is at most \( \abs{ V_\kappa^3} = 2  \abs{ E_\kappa^3}\),
  we immediately obtain~\eqref{First Ward}, and~\eqref{Second Ward} in case \((\ldots)_+=0\) from~\eqref{number V3 adj}. 
  
  For the proof of~\eqref{Second Ward} in case \((\ldots)_+>0\) we consider a larger
  candidate set of size \(e_g(V_\kappa^{\ge 3},V_\mathrm{e})+e_g(V_\kappa^{\ge 3},V_\kappa^{\ge 3}) \) that consists of all edges adjacent to \(V_\kappa^{\ge 3}\)-vertices. Going through  
  all   \(V_\kappa^{\ge3}\)-vertices in arbitrary order
  we remove at most \(k-2\) edges for each vertex  \(v\in V_\kappa^{k}\),
  so that the at most two remaining edges  are not loops; 
  this yields again a Wardable set. 
  Since \(\abs{V_\kappa^{k}}= 2\abs{E_\kappa^k}\), the total number of removed edges is at most
  \[
  \sum_{k\ge 3} \sum_{e\in E_\kappa^k} 2(k-2)  = \sum_{e\in E_\kappa^{\ge 3}} 
  (2d_g(e)-4)   =  \sum_{e\in E_\kappa^{\ge 3}} d_g(e) + 
  \sum_{e\in E_\kappa} (d_g(e)-4)_+ - \abs{E_\kappa^3},
  \]
  which, together with~\eqref{number Vge3 adj} yields~\eqref{Second Ward}. This completes the proof of~\ref{step number of edges}
  modulo~\eqref{number V3 Vge3 adj} that we prove now.
\end{proof}
\begin{proof}[Proof of~\eqref{number V3 Vge3 adj}]
  The bound~\eqref{number V3 adj} follows from 
  \[
  \begin{split}
    6\abs{E_\kappa^3} &= 2\sum_{(uv)\in E_\kappa^3} d_g((uv)) = \sum_{v\in V_\kappa^3} d_g(v) \\
    &= 2 e_g(V_\kappa^3,V_\kappa^3) + e_g(V_\kappa^3,V_\kappa^{\ge 4}\cup V_\mathrm{e}) \le 2 e_g(V_\kappa^3,V_\kappa^{\ge 3}\cup V_\mathrm{e}).
  \end{split}\]
  For the bound~\eqref{number Vge3 adj} we note that the set \(E_g^{\red}\setminus E_g^{\red,\cyc}\) can be partitioned into edges within \(V_\kappa^{\ge3}\), edges within \(V_\mathrm{e}\) and edges between these two sets, and thus 
  from~\ref{number of G edges}--\ref{no internal} and~\eqref{number edges Egred} we obtain 
  \[\begin{split}
    e_g(V_\kappa^{\ge 3},V_\mathrm{e})+e_g(V_\kappa^{\ge 3},V_\kappa^{\ge 3})&=\abs{E_g^{\red}\setminus E_g^{\red,\cyc}}-e_g(V_\mathrm{e},V_\mathrm{e})=\abs{E_g^\red}-\abs{V_\mathrm{cyc}}-e_g(V_\mathrm{e},V_\mathrm{e}) \\
    & =  \sum_{e\in E_\kappa^{\ge 3}}d_g(e)+2p - e_g(V_\mathrm{e},V_\mathrm{e}).
  \end{split}\]
  Furthermore, by~\ref{number of external} each \(V_\mathrm{e}\)-\(V_\mathrm{e}\) edge corresponds to at least one \(V_\kappa^2\)-vertex, while by~\ref{no loops} each cycle \(E_g^{\red,\cyc}\) corresponds to at least two \(V_\kappa^2\)-vertices in \(\Gamma\) (which are in particular not part of any chain), whence
  \[e_g(V_\mathrm{e},V_\mathrm{e}) \le \abs{V_\kappa^2} -  2\abs{V_\cyc}=2(\abs{E_\kappa^2}-\abs{V_\cyc})\]
  and the claim follows. 
\end{proof}
\begin{proof}[Proof of~\ref{step number of gains}]
  We recall from the proof of Lemma~\ref{lemma Itwo} that~\eqref{Itwo eq def} is the minimum of two different estimates given in~\eqref{eq I2est min}. Estimating each \(\cG^e\) for \(e\in E_g^\red\) by Lemma~\ref{degree two lemma} with a \(\rho\)-exponent of \(i(e)\) in~\eqref{Val reduced def} yields the first bound \(\abs{\Val(\Gamma_\red)}\prec \ItwoiEst(\Gamma)\). Similarly, estimating each \(\cG^e\) by Lemma~\ref{degree two lemma} with a \(\rho\)-exponent of \(l(e)-o(e)-\bm1(e\in E_g^{\red}\setminus E_g^{\red,\cyc})\) yields the second bound \(\abs{\Val(\Gamma_\red)}\prec \ItwozEst(\Gamma)\), cf.\ the first inequality in~\eqref{counting of rhos I two}. In order to prove~\ref{step number of gains} for a given Wardable set \(E_\mathrm{Ward}\) we estimate \(\cG^{e}\) for \(e\in E_g^{\red,\cyc}\cup (E_g^{\red}\setminus (E_g^{\red,\cyc}\cup E_\mathrm{Ward}))\) exactly as in Lemma~\ref{lemma Itwo} and remove the corresponding edges from the graph, leaving only \(E_\mathrm{Ward}\)-edges. In order to conclude the proof it remains to establish an additional gain of \(K^{-1/2}\) (compared to the first bound) and \(\rho K^{-1/2}\) (compared to the second bound) per \(E_\mathrm{Ward}\)-edge \(e\) compared to the entrywise estimates. 
  
  Let \(v_1,v_2,\ldots\) denote the ordering of \(V_\kappa^{\ge 3}\) guaranteed to exist by Lemma~\ref{lemma degen}. By definition of \(E_\mathrm{Ward}\) at most two Wardable edges are adjacent to \(v_1\) and whence the part of the value depending on \(a_{v_1}\) can be estimated by either 
  \begin{equation}\label{ward 1 gain} 
    \begin{split}
      \sum_{a_{v_1}} \abs{\cG_{\vx_w a_{v_1}}^{(wv)}} &\le N^{1/2} \sqrt{ [\cG^{(wv_1)} (\cG^{(wv_1)})^\ast]_{\vx_w \vx_x} }
    \end{split}
  \end{equation}
  or
  \begin{equation}\label{ward 2 gain} 
    \sum_{a_{v_1}} \abs{\cG_{\vx_w a_{v_1}}^{(wv_1)}} \abs{\cG_{a_{v_1}\vx_y}^{(v_1y)}} \le \sqrt{ [\cG^{(wv_1)} (\cG^{(wv_1)})^\ast]_{\vx_w \vx_w} } \sqrt{ [ (\cG^{(v_1y)})^\ast \cG^{(yv_1)}  ]_{\vx_y \vx_y} }
  \end{equation}
  using Cauchy-Schwarz for some \(w,y\in V_\kappa^{\ge 3}\cup V_\mathrm{e}\). In case of \(\ItwoiEst\) the entrywise estimate on the lhs.\ of~\eqref{ward 1 gain}--\eqref{ward 2 gain} used in the proof of Lemma~\ref{lemma Itwo} is at least
  \[\Lambda_+^a\Pi_+^t\rho^{i}N^{l-o/2}K^{o-l+1} \quad \text{and}\quad \Lambda_+^{a+a'}\Pi_+^{t+t'}\rho^{i+i'}N^{l+l'-o/2-o'/2-1}K^{o+o'-l-l'+2}\]
  with \(i=i((wv_1))\), \(l=l((wv_1))\), \(a=a((wv_1))\), \(t=t((wv_1))\), \(o=t+a\) and \(i'=i((v_1y))\), \(l'=l((v_1y))\), \(a'=a((v_1y))\), \(t'=t((v_1y))\), \(o'=t'+a'\) while applying Lemma~\ref{degree two lemma} to the rhs.\ yields 
  \[ \Lambda_+^a\Pi_+^t \rho^{i} N^{l-o/2} K^{o-l+1/2} \quad \text{and}\quad \Lambda_+^{a+a'}\Pi_+^{t+t'}\rho^{i+i'}N^{l+l'-o/2-o'/2-1}K^{o+o'-l-l'+1},\]
  demonstrating the gains of at least \(K^{-1/2}\) and \((K^{-1/2})^2\), respectively.
  Similarly, the \(\ItwozEst\)-estimate on the lhs.\ of~\eqref{ward 1 gain}--\eqref{ward 2 gain} is at least   
  \[\Lambda_+^a\Pi_+^t\rho^{l-o-1}N^{l-o/2}K^{o-l+1} \quad \text{and}\quad \Lambda_+^{o+o'}\rho^{l+l'-o-o'-2}N^{l+l'-o/2-o'/2-1}K^{o+o'-l-l'+2}\]
  while, in comparison, when applying Lemma~\ref{degree two lemma} to the rhs.\ of~\eqref{ward 1 gain}--\eqref{ward 2 gain}, we obtain bounds of 
  \[ \Lambda_+^a\Pi_+^t \rho^{l-o} N^{l-o/2} K^{o-l+1/2} \quad \text{and}\quad \Lambda_+^{o+o'}\rho^{l+l'-o-o'}N^{l+l'-o/2-o'/2-1}K^{o+o'-l-l'+1}, \]
  demonstrating exactly the claimed gain of \(\rho K^{-1/2}\) per edge. Here, for example, we counted that \(\cG^{(wv_1)} (\cG^{(wv_1)})^\ast\) contains \(2l\) factors of \(G\) and \(2o\) orthogonality indices satisfying \(2o\le 2l-2<2l-1\). 
  
  The proof now follows by induction since by Lemma~\ref{lemma degen} after the removal of \(v_1\), the next vertex \(v_2\) has degree at most \(2\) etc.\ and~\eqref{ward 1 gain}--\eqref{ward 2 gain} can be used to establish the gain of \((\rho)K^{-1/2}\) iteratively for each \(e\in E_\mathrm{Ward}\).
\end{proof}

\appendix
\section{Removing the simplifying assumptions in the proof of Theorem~\ref{chain G underline theorem}}\label{sec gen}

\subsection{Removing the \texorpdfstring{\(w_2=1+\sigma\)}{w2=1+sigma} Assumption~\ref{assump w2}}\label{ref relaxing diag} If \(w_2\ne 1+\sigma\), then an additional diagonal \(\delta_{ab}\) term appears in~\eqref{eq cum exp2}, i.e.\ 
\begin{equation}\label{eq cum exp3}
  \begin{split}
    \E w_{ab} f(W) &=  \E\frac{\partial_{ba} f(W) + \sigma \partial_{ab} f(W)}{N} + \delta_{ab}\frac{w_2-1-\sigma}{N} \E \partial_{aa} f(W) \\
    &\qquad + \sum_{k=2}^R\sum_{p+q=k} \frac{\kappa^{p+1,q}_{ab}}{N^{(k+1)/2}} \E \partial_{ab}^p \partial_{ba}^q f(W) + \Omega_R.
  \end{split}
\end{equation}
As a consequence additional graphs appear in the estimate where degree-two \(\kappa\)-edges are collapsed due to \(\delta_{ab}\) which we will show to be lower order due to fewer summations. Indeed, let \(\Gamma\) be any av/iso-graph and for \((uv)\in E_\kappa^2\) consider the graph \(\Gamma'\) obtained from collapsing the vertices \(u,v\) into one. We claim that 
\begin{equation}\label{collapsed value}
  \IthreeEst(\Gamma')\le \IthreeEst(\Gamma),
\end{equation} 
and thus the bounds in Proposition~\ref{prop value est} remain valid for partially collapsed graphs. Repeating the estimates~\eqref{collapsed value} recursively for all collapsed \(\kappa\)-vertices we see that the proof of Theorem~\ref{chain G underline theorem} is complete also without the simplifying Assumption~\ref{assump w2}. 

It remains to prove~\eqref{collapsed value}. By Lemma~\ref{lemma reduction} both vertices \(u,v\) are necessarily internal vertices of some \(G\)-chain or \(G\)-cycle. Now there are several possible scenarios. First, one of \(u,v\) may be in the set \(V_\mathrm{o}\) of selected orthogonality vertices (but not both cf.\ the construction in~\ref{claim orth2}), and second, 
\begin{enumerate}[label=(opt\arabic*)]
  \item\label{one chain} \(u,v\) are in the same chain, 
  \item\label{one cycle} \(u,v\) are in the same cycle,
  \item\label{two chains} \(u,v\) are in two different chains,
  \item\label{two cycles} \(u,v\) are in two different cycles,
  \item\label{chain and cycle} one of \(u,v\) is in a chain, the other one in a cycle. 
\end{enumerate}

For instance suppose we are in scenario~\ref{two chains}, in which we compare \(\IthreeEst(\Gamma)\) of the two graphs 
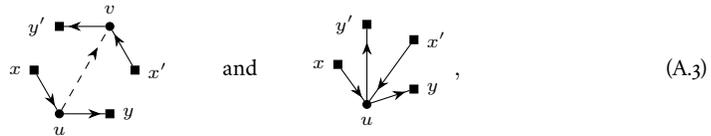
\begin{equation}\label{eq graph collapse example}
  \sGraph{x[rectangle,label=left:\(x\)] -- u[label=below:\(u\)] -- y[rectangle,label=right:\(y\)]; x2[rectangle,label=right:\(x'\)] -- v[label=above:\(v\)] -- y2[rectangle,label=left:\(y'\)]; u --[dashed] v;}\qquad\text{and}\qquad 
  \sGraph{x[rectangle,label=left:\(x\)] -- u[label=below:\(u\)] -- y[rectangle,label=right:\(y\)]; x2[rectangle,label=right:\(x'\)] -- u -- y2[rectangle,label=left:\(y'\)]; },
\end{equation}
where the square vertices denote vertices from \(V_\mathrm{e}\cup V_\kappa^{\ge 3}\), and the \(G\)-edges may denote chains of arbitrary lengths \(l_1,\ldots,l_4\) with \(o_1,\ldots o_4\) internal \(V_\mathrm{o}\)-vertices. On the lhs.\ of~\eqref{eq graph collapse example} the product of the estimates on the chains \((xy)\) and \((x'y')\) using Lemma~\ref{degree two lemma} is at least
\begin{equation}\label{eq example lhs bound}
  N^{\sum_i (l_i-o_i/2)}K^{\sum_i(o_i-l_i)+2}\rho^{\sum_i(l_i-o_i)} \times \begin{cases}
    N^{-2} K^2 & u,v\not\in V_\mathrm{o}\\
    N^{-5/2} K^3 \rho^{-2}, & \abs{\set{u,v}\cap V_\mathrm{o}}=1
  \end{cases} 
\end{equation} 
while on the rhs.\ of~\eqref{eq graph collapse example} the product of the estimates on the four chains \((xu),(uy),(x'u),(uy')\) and the size \(N\) of the summation corresponding to \(u\) is at most \(N N^{\sum_i (l_i-o_i/2)-4}K^{\sum_i(o_i-l_i)+4}\rho^{\sum(l_i-o_i)-4}\). However, for the graph on the rhs.\ we can gain at least two additional factors of \(\rho K^{-1/2}\) since there are at least two additional edges for which the Ward gain from~\ref{step number of gains} is applicable due to the \(u\)-summation. Thus, we obtain an estimate 
\begin{equation}\label{eq example rhs bound}
  N^{\sum_i (l_i-o_i/2)}K^{\sum_i(o_i-l_i)}\rho^{\sum(l_i-o_i)} N^{-3}K^3 \rho^{-2}.
\end{equation}
Using \(K\lesssim N\rho^2\) it follows that the bound~\eqref{eq example rhs bound} is not larger than~\eqref{eq example lhs bound} in both cases, and thus the \(\IthreeEst\)-estimate on the subgraph on the lhs.\ cannot be larger than the subgraph on the rhs.\ and the claim~\eqref{collapsed value} follows. The comparison works similarly for the other scenarios~\ref{one chain},~\ref{one cycle},~\ref{two cycles} and~\ref{chain and cycle}, so we omit further details regarding the proof of~\eqref{collapsed value} in those cases. 

\subsection{Estimate for general \texorpdfstring{\(\eta>0\)}{eta>0}; removal of assumption~\ref{assump eta}}\label{sec large eta} 
The proof is considerably simpler in the regime \(\eta\not\lesssim1\), where using norm bounds \(\norm{G}\le 1/\eta\) we can replace the bounds~\eqref{Improved 2 reduced edge est}--\eqref{Improved 2 reduced edge est iso} simply by \(\eta^{-l(e)}\). Thus, from~\eqref{Val reduced def} we obtain for any av-graph \(\Gamma\),
\[ \abs{\Val(\Gamma)}\prec \ItwoEst'(\Gamma) := N^{-\abs{E_\kappa^2}+\abs{V_\cyc}+\abs{E_\kappa^3}/2} \eta^{-2lp} \]
using~\eqref{red e a i rel} and~\ref{number of G edges}. Similarly, for any iso-graph \(\Gamma\) we obtain 
\[ \abs{\Val(\Gamma)}\prec \ItwoEst'(\Gamma) := N^{-\abs{E_\kappa^2}+\abs{V_\cyc}+\abs{E_\kappa^3}/2-\sum_{e\in E_\kappa}(d_g(e)/2-2)_+} \eta^{-2lp}.\]

Moreover, exactly as in~\ref{step number of edges} and similarly to~\ref{step number of gains} we find Wardable sets of edges, for each of which we gain a factor of \(N^{-1/2}\) compared to the \(\ItwoEst'(\Gamma)\) estimates above, in order to obtain 
\[ \abs{\Val(\Gamma)} \prec N^{-\abs{E_\kappa^2}+\abs{V_\cyc}} \eta^{-2lp}\le \eta^{-2lp} \]
for any av-graph \(\Gamma\) (using~\ref{no loops}) and 
\[ \abs{\Val(\Gamma)} \prec N^{-p} \eta^{-2lp} \]
for any iso-graph \(\Gamma\), concluding the proof of~\eqref{eq large eta bounds} for the remaining large \(\eta\) regime.

\subsection{Averaged bound in case \texorpdfstring{\(l\not\in\mathfrak{a}\cup\mathfrak{t}\ne\emptyset\)}{l e a u t}; removing Assumption~\ref{assump l a t}}\label{sec a last case}
Here we consider the \(l\not\in\mathfrak{a}\cup\mathfrak{t}\) case of the averaged bound in Theorem~\ref{chain G underline theorem}. The only difference to the case \(l\in\mathfrak{a}\cup\mathfrak t\) is the selection process of \(V_\mathrm{o}\) vertices described in~\ref{claim orth1}--\ref{claim orth2}. We fix some \(j\in\mathfrak a\cup \mathfrak t\) arbitrarily and select vertices into \(V_\mathrm{o}\) as follows:
\begin{enumerate}[label=(orth'-\arabic*)]
  \item\label{claim orthp1} For each \(k\in\mathfrak{t}\setminus\set{j},\mathfrak a\setminus\set{j}\) we collect \(2p\) distinct vertices from \(V_\mathrm{i}\) into the sets \(V_\mathrm{o}^t\) and \(V_\mathrm{o}^{0\mathrm{tr}}\), respectively. 
  \item\label{claim orthp2} If \(j\in\mathfrak t\) or \(j\in\mathfrak a\), then we select one additional vertex from \(V_\mathrm{i}\) into \(V_\mathrm{o}^t\) or \(V_\mathrm{o}^{0\mathrm{tr}}\), respectively, for each \(W\) acting as a degree-\(2\) cumulant on some resolvent. 
\end{enumerate}
The fact that the selection of \(V_\mathrm{o}\)-vertices in~\ref{claim orthp1} is possible follows exactly as for~\ref{claim orth1}, i.e.\ due to the fact that internal orthogonality vertices are guaranteed to remain orthogonality vertices throughout the cumulant expansion. For~\ref{claim orthp2} we note that we could also add all \(2p\) vertices corresponding to \(j\) into the set \(V_\mathrm{o}\) due to them being internal, however more \(V_\mathrm{o}\) vertices is not necessarily beneficial, cf.\ Remark~\ref{remark loss}. It remains to establish that the set \(V_\mathrm{o}\) satisfies the same bounds as the set \(V_\mathrm{o}\) constructed in case \(l\in\mathfrak a\cup \mathfrak t\), i.e.\ 
\begin{subequations}
  \begin{equation}\label{Vaeff le}
    \abs{V_\mathrm{o}} + 2\abs{V_\cyc^{o=0}}+\abs{V_\cyc^{o=1}} \le 2(o-1)p + 2\abs{E_\kappa^2},
  \end{equation}
  and 
  \begin{equation}\label{Vaeff ge}
    \abs{V_\mathrm{o}} \ge 2\abs{E_\kappa^2} +\abs{E_\kappa^3} + 2(o-2)p,
  \end{equation}
\end{subequations}
which follow from~\ref{Va Vcyc sum claim}--\ref{a rule}, as from~\eqref{Vaeff le}--\eqref{Vaeff ge} the claimed bound follows exactly as in~\eqref{o-1 case}. 

The claim~\eqref{Vaeff ge} follows immediately from~\eqref{number of G actions}. Regarding~\eqref{Vaeff le} we monitor the change of \(\abs{E_\kappa^{2,3}}\), \(V_\mathrm{o}\), etc.\ along the iterative construction of the graphs in the proof of Proposition~\ref{cumulant expansion prop}. In addition, we count the number \(n_{\cyc,\un{W}G}\) of cycles including \(W\) and some non-underlined \(G\). In the beginning of the expansion we have \(2p\) cycles, each including one \(W\) and \(l\) underlined \(G\)'s. Now, if some \(W\) acts on some \(W\) in another cycle, then the two cycles are replaced by one cycle with \(2l\) non-underlined \(G\)'s. However, if some \(W\) acts on a \(G\) in another cycle, then the two cycles are replaced by one cycle with one \(W\) and \((2l+1)\) \(G\)'s, \(l\) of which are not underlined, e.g.
\[ 
\begin{split}
  &\E\Tr\un{WGAGB} \Tr\un{WGAGB} \\
  &\quad = N^{-1}\E\sum_{ab} \Tr \Delta^{ab}GAGB \Bigl(\Tr \Delta^{ba}GAGB - \Tr\un{WG\Delta^{ba}GAGB} \Bigr) +\cdots\\
  & \quad = N^{-1} \E \Bigl( \Tr GAGBGAGB - \Tr GAGB\un{GAGBWG}\Bigr) + \cdots,
\end{split} 
\]
demonstrating the two possible actions. The newly created partially underlined cycle is of importance since it, contrary to the original fully underlined cycles, allows for \(W\) to act on some \(G\)'s (the non-underlined ones) in its own cycle, e.g.
\[ 
\begin{split}
  \E\Tr GAGB\un{GAGBWG} &= -N^{-1}\E\sum_{ab} \Tr G\Delta^{ba}G AGBGAGB\Delta^{ab} G +\cdots\\
  &=-N^{-1}\Tr G^2  \Tr GAGBGAGB + \cdots.
\end{split}\]
This mechanism was also present in the main body of the proof of Theorem~\ref{chain G underline theorem}, see e.g.~\eqref{WGAImGA 2nd exp}, but there we did not need to monitor the number of partially underlined cycles along the cumulant expansion and they disappeared in the end. In the current proof \(n_{\cyc,\un WG}\) is an auxiliary quantity to prove~\eqref{Vaeff le}. We claim that in all steps along the expansion the inequality 
\begin{equation}\label{Vaeff le gen}
  \abs{V_\mathrm{o}} + 2\abs{V_\cyc^{o=0}}+\abs{V_\cyc^{o=1}} + n_{\cyc,\un{W}G}\le 2(o-1)p + 2\abs{E_\kappa^2}
\end{equation}
is valid, which is obvious initially since there we have \(E_\kappa=\emptyset\), \(\abs{V_\mathrm{o}}=2(o-1)p\) and \(n_{\cyc,\un W G}=0\). Whenever some \(W\) acts as a degree-\(2\) cumulant on another \(W\), then in our algorithm no vertex is added to \(V_\mathrm{o}\), while \(\abs{E_\kappa^2}\) is increased by \(1\), and one pure-\(G\) cycle is created, so~\eqref{Vaeff le gen} continues to remain valid. Otherwise, if some \(W\) acts on some \(G\) in its own cycle (which is only possible if the corresponding \(G\) is not underlined), then \(n_{\cyc,\un{W}G}\) is decreased by \(1\), while \(\abs{V_\mathrm{o}},\abs{E_\kappa^2}\) are increased by \(1\), and either \(\abs{V_\cyc^{o=0}}\) is increased by at most \(1\), or \(\abs{V_\cyc^{o=1}}\) is increased by at most \(2\), confirming~\eqref{Vaeff le gen}. Next, if some \(W\) acts on \(G\) in another cycle, then \(\abs{V_\cyc^{o=0,1}}\) cannot increase, while \(n_{\cyc,\un{W}G}\) may increase by \(1\), and both \(\abs{V_\mathrm{o}},\abs{E_\kappa^2}\) do increase by \(1\), respecting~\eqref{Vaeff le gen}. Finally, if \(W\) acts on either \(W\) or \(G\) in some non-cycle, then the number of cycles cannot be increased, making the validity of~\eqref{Vaeff le gen} trivial. Any higher-degree cumulant expansions cannot increase the lhs.\ of~\eqref{Vaeff le gen} while leaving the rhs.\ invariant. This proves~\eqref{Vaeff le gen} inductively along the expansion. Hence, after all cumulant expansions are performed (so that, in particular, \(n_{\cyc,\un{W}G}=0\))~\eqref{Vaeff le gen} implies~\eqref{Vaeff le}. 

\subsection{Isotropic bound with \texorpdfstring{\(j=0\)}{j=0}; relaxing assumption~\ref{assump j iso}}
We now consider the isotropic bound~\eqref{eq:underIso} for \(j=0\), i.e.\ 
\[\braket{\vx,\un{WG_1B_1\cdots B_{l-1}G_l}\vy}.\] 
In order to describe the structure of graphs encoding the polynomial from the cumulant expansion of 
\begin{equation}\label{iso j0 exp}
  \E\abs{\braket{\vx,\un{WG_1B_1\cdots B_{l-1}G_l}\vy}}^{2p}
\end{equation}
similarly to the graphs from Definition~\ref{def graphs}, it is convenient to add a new type of edge \(E_=\) for edges encoding the identity matrix \(I\) which we use to ``connect'' \(\vx\) and \(W\), i.e.\ in the resulting graph, \((V_\kappa\dot\cup V_\mathrm{e},E_=)\) is bipartite. Since the \(E_=\)-edges represent the identity matrix which is symmetric, their orientation is irrelevant contrary to the \(E_g\) and \(E_\kappa\)-edges. The set of graphs we obtain satisfies~\ref{perfect matching}--\ref{no loops} (after redefining the \(d_g\)-degree to also count \(E_=\)-edges), 
\begin{equation}\label{Vi edges}
  \abs{V_\mathrm{i}}=2(l-1)p=2(o+b)p, \quad \abs{V_\mathrm{o}}=\abs{V_\mathrm{i}\cap V_\mathrm{o}}=2op, \quad V_\mathrm{o}\cap V_\kappa^2 =\emptyset,
\end{equation}
Eq.~\ref{number of external} and 
\begin{equation}\label{eq Eeq edges}
  \abs{E_=}=2p.
\end{equation}
Moreover we claim that each graph satisfies the inequality 
\begin{equation}\label{eq Vsc}
  \abs*{\set{v\in V_\kappa\given \abs{v\cap E_=}\ge 2}} \le (2p-\abs{E_\kappa})\wedge p,
\end{equation}
where \(v\cap E_=\) is understood as the set of edges from \(E_=\) adjacent to \(v\). Indeed, in the expansion of~\eqref{iso j0 exp} two \(E_=\) edges can only meet in some \(v\in V_\kappa\) if one of the adjacent \(W\)'s acted as a derivative on the other one. Thus~\eqref{eq Vsc} follows from the fact that in total \(2p-\abs{E_\kappa}\) derivatives have acted on some \(W\), cf.\ the proof of~\ref{a rule}, noting that the upper bound of \(p\) is trivial by~\eqref{eq Eeq edges}. 

Example graphs occurring along the expansion encode the polynomials (where the second and third term are non-zero only in the real case)
\begin{equation}\label{iso 0 expansion example}
  \begin{split}
    &\E \abs{\braket{\vx,\un{W G}\vy}}^2  = \E \braket{\vx,\un{W G}\vy} \braket{\vy,\un{G^\ast W}\vx}\\ 
    &= \E\sum_{ab} \kappa(ab,ba)I_{\vx a}G_{b\vy} G^\ast_{\vy b}I_{a\vx} + \E\sum_{ab} \kappa(ab,ab)I_{\vx a}G_{b\vy} G^\ast_{\vy a}I_{b\vx} \\
    &\quad+  \E\sum_{abcd}\kappa(ab,ab)\kappa(cd,cd) I_{\vx a}G_{bc}G_{d\vy} G^\ast_{\vy a}G^\ast_{bc}I_{d\vx} \\
    &\quad - \E\sum_{abcd}\kappa(ab,ab,ba)\kappa(cd,dc) I_{\vx a} G_{bb} G_{ad} G_{c\vy} G^\ast_{\vy a} G^\ast_{bc} I_{d\vx} + \cdots
  \end{split}
\end{equation}
which we represent graphically as 
\begin{equation}\label{iso 0 expansion example graph}
  \begin{split}
    \E \abs{\braket{\vx,\un{W G}\vy}}^2 &
    = \E\Val\left(\sGraph{ x1[rectangle,label=above:\(\vx\)] --[double] a[label=below:\(a\)] --[dashed] b[label=below:\(b\)] -- y1[rectangle,label=above:\(\vy\)]; y2[rectangle,label=above:\(\vy\)] -- b;  a --[double] x2[rectangle,label=above:\(\vx\)]; }\right)
    + \E\Val\left(\sGraph{ x1[rectangle,label=above:\(\vx\)] --[double] a[label=below:\(a\)] --[dashed] b[label=below:\(b\)] -- y1[rectangle,label=above:\(\vy\)]; x2[rectangle,label=above:\(\vx\)]; y2[rectangle,label=above:\(\vy\)] -- a;  b --[double] x2; }\right) \\&\quad  
    + \E\Val\left(\sGraph{ x1[rectangle]  --[double] a --[dashed] b --[bl] c --[dashed] d -- y1[rectangle]; x2[rectangle]; y2[rectangle] -- a;  b --[br] c; d --[double] x2; }\right) 
    - \E\Val\left(\sGraph{ y2[rectangle]; x1[rectangle]; a; b; y1[rectangle]; c; d; c; x2[rectangle]; a --[dashed] b; c --[dashed] d; x1[rectangle]  --[double] a; b --[glr] b; a -- d; c -- y1[rectangle]; y2 -- a;  d --[double] x2[rectangle]; b--c;}\right)+\cdots,
  \end{split}
\end{equation}
with double line representing \(E_=\) edges. 

The graph reduction procedure of \(V_\mathrm{i}\cup V_\kappa^2\) is performed exactly as in Lemma~\ref{lemma reduction}. However, regarding the \(E_=\)-edges, two new phenomena occur. First, if some \(v\in V_\kappa^2\) connects two \((uv),(vw)\in E_=\) edges, then the two edges are collapsed and contribute just a scalar factor to the graph value, the inner product \(\braket{\vx_u,\vx_w}\) of the vectors \(\vx_u,\vx_w\) associated with the external vertices \(u,w\), see e.g.\ in the first term on the rhs.\ of~\eqref{iso 0 expansion example}--\eqref{iso 0 expansion example graph}: \(\sum_{a} I_{\vx a} I_{a\vy} = \braket{\vx,\vy}\). In the reduced graph, we record this as a new type of \emph{isolated} vertex \(v\in V_\mathrm{sc}\), representing the scalar product. Second, if some \(v\in V_\kappa^2\) connects one \((uv)\in E_=\) and one \((vw)\in E_g\), then the reduction process corresponds to simply replacing these two edges by \((uw)\in E_g^\mathrm{red}\), representing the matrix \(\cG^{(uw)}:=\cG^{(vw)}\), see the second, third and fourth term in~\eqref{iso 0 expansion example}--\eqref{iso 0 expansion example graph}, e.g.\ \(\sum_b G_{b\vy} I_{\vx b} = G_{\vx \vy}\). As a result of the reduction process we obtain a reduced graph 
\[\Gamma_\mathrm{red}=(V_\mathrm{e}\cup V_\kappa^{\ge3}\cup V_\cyc\cup V_\mathrm{sc},E'_=\cup E_g^\mathrm{red}),\]
where \(E'_=\) denotes the subset of \(E_=\) connecting \(V_\mathrm{e}\) and \(V_\kappa^{\ge 3}\). Similarly to~\eqref{number edges Egred}, the number of edges in the reduced graph is given by
\begin{equation}\label{eq gamma red j0 case}
  \abs{E_g^\mathrm{red}}=\abs{E_g}+\abs{E_=} - \abs{V_\mathrm{i}} - 2\abs{E_\kappa^2} + \abs{V_\cyc} - \abs{V_\mathrm{sc}}.
\end{equation}

For each \(e\in E_g^\mathrm{red}\) we use the entrywise bound of Lemma~\ref{degree two lemma}. On top of that, similarly to Lemma~\ref{lemma Ithree}, we obtain a set \(E_\mathrm{Ward}\subset (E_g^{\red}\setminus E_g^{\red,\cyc})\cup E'_=\) of Wardable edges, which contrary to the previous case may include both \(E_g^\red\)- and \(E'_=\)-edges. The proof of~\eqref{Second Ward} verbatim also applies to the current case by choosing the \emph{candidate sets} of all \(E'_=\cup (E_g^{\red}\setminus E_g^{\red,\cyc})\)-edges adjacent to either \(V_\kappa^3\) or \(V_\kappa^{\ge 3}\) vertices. Thus we obtain a Wardable set of the same minimal size as in~\eqref{Second Ward}. Due to 
\[ 
\begin{split}
  \sum_{a} \abs{I_{\vx a}} &= \sum_{a} \abs{\vx_a} \le \sqrt{N} \norm{\vx}\lesssim \sqrt{N}\\
  \sum_{a} \abs{I_{\vx a}} \abs{I_{\vy a}}&\le \norm{\vx}\norm{\vy} \lesssim 1
\end{split}\]
we can gain at least a factor of \(N^{-1/2}\) per \(E'_=\cap E_\mathrm{Ward}\) edge.  Thus, similarly to~\eqref{eq val entrywise bd} and with the additional improvement from the Wardable edges, we obtain 
\[ \begin{split}
  \abs{\Val(\Gamma)} &\prec \Lambda_+^{\abs{V_\mathrm{o}^{0\mathrm{tr}}}}\Pi_+^{\abs{V_\mathrm{o}^{t}}} \rho^{\abs{V_\mathrm{i}}+2\abs{E_\kappa^2}+\abs{E_\kappa^3}-\abs{V_\mathrm{o}}} K^{ \abs{V_\mathrm{o}} - \abs{E_g} + \abs{E_g^\red} -\abs{E_\mathrm{Ward}\cap E_g^\red}/2} \\
  &\qquad \times N^{-\abs{E_\kappa^2}+\abs{V_\cyc} +\abs{E_\kappa^3}/2 + \abs{E_g}-\abs{V_\mathrm{o}}/2-\abs{E_g^\red}-\abs{E_\mathrm{Ward}\cap E'_=}/2-\delta^{\ge 4}}  \\
  &= \Lambda_+^{2ap}\Pi_+^{2tp}  \rho^{2pb + 2\abs{E_\kappa^2}+\abs{E_\kappa^3} }N^{ p(a+2b)} N^{\abs{E_\kappa^2}+\abs{E_\kappa^3}/2+\abs{V_\mathrm{sc}} - 2p -\abs{E_\mathrm{Ward}\cap E'_=}/2-\delta^{\ge 4}}\\
  &\qquad \times K^{ -2bp }K^{2p -\abs{V_\mathrm{sc}} -2\abs{E_\kappa^2} +\abs{V_\cyc} -\abs{E_\mathrm{Ward}\cap E_g^\red}/2},
\end{split}  \]
where in the equality we used~\eqref{eq Eeq edges} and~\eqref{eq gamma red j0 case}. Due to~\eqref{eq Vsc} we have \(\abs{V_\mathrm{sc}}\le 2p-\abs{E_\kappa}\) and therefore the \(N\)-exponent is non-positive and we obtain from \(1/N\lesssim \rho^2/K\) that
\[\begin{split}
  \abs{\Val(\Gamma)}&\lesssim \Lambda_+^{2ap}\Pi_+^{2tp} N^{ (o+2b)p} \rho^{2(b+2)p - 2\abs{V_\mathrm{sc}} } K^{ -2bp } K^{-\abs{E_\kappa^2}+\abs{E_\kappa^3}/2  -\abs{E_\mathrm{Ward}}/2-\delta^{\ge 4}  +\abs{V_\cyc} }\\
  &\lesssim \Lambda_+^{2ap}\Pi_+^{2tp} \rho^{2(b+1)p} N^{ (o+2b)p} K^{ -(1+2b)p },
\end{split}  \] 
with the second inequality following from~\eqref{eq Vsc} and~\eqref{Second Ward}. 

\printbibliography%

\end{document}

%% file: packages.tex
\usepackage[english]{babel}
\usepackage[p,osf]{cochineal}
\usepackage[scale=.95,type1]{cabin}
\usepackage[zerostyle=c,scaled=.94]{newtxtt}
\usepackage[T1]{fontenc} 
\usepackage[utf8]{inputenc} 
\usepackage{amsthm}
\usepackage{amsmath}
\usepackage{amssymb}
\usepackage{amsfonts}
\usepackage{amsaddr}
\usepackage{xspace}
\usepackage{enumitem} 
\usepackage{csquotes}
\usepackage{xcolor}
\usepackage[colorlinks]{hyperref}
\hypersetup{colorlinks, breaklinks, citecolor=teal,filecolor=black}
\usepackage{graphicx}
\usepackage{mathtools}
\mathtoolsset{centercolon}
\usepackage{bm}
\usepackage{pgfkeys}
\usepackage{pgf,interval}
\intervalconfig{soft open fences}

%% file: notations.tex
\DeclareMathOperator{\diag}{diag}
\DeclareMathOperator{\Tr}{Tr}

\DeclareMathOperator{\Var}{Var}
\DeclareMathOperator{\rank}{rank}
\DeclareMathOperator{\Val}{Val}
\DeclareMathOperator{\E}{\mathbf{E}}
\DeclareMathOperator{\Prob}{\mathbf{P}}

\DeclareMathOperator{\ItwoEst}{I_2-Est}
\DeclareMathOperator{\ItwoiEst}{I_2^i-Est}
\DeclareMathOperator{\ItwozEst}{I_2^0-Est}
\DeclareMathOperator{\IthreezEst}{I_3^0-Est}
\DeclareMathOperator{\IthreeEst}{I_3-Est}
\DeclareMathOperator{\IthreeiEst}{I_3^i-Est}

\newcommand{\ov}{\overline}
\newcommand{\red}{\mathrm{red}}
\newcommand{\cyc}{\mathrm{cyc}}

\newcommand{\ii}{\mathrm{i}}

\newcommand{\F}{\mathbf{F}}

\ifdefined\C
\renewcommand{\C}{\mathbf{C}}
\else
\newcommand{\C}{\mathbf{C}}
\fi
\newcommand{\HC}{\mathbf{H}}

\newcommand{\un}{\underline}
\newcommand{\vx}{\bm{x}}

\newcommand{\vm}{\bm{m}}

\newcommand{\bu}{\bm{u}}
\newcommand{\bq}{\bm{q}}
\newcommand{\vy}{\bm{y}}

\newcommand{\wt}{\widetilde}

\newcommand{\R}{\mathbf{R}}
\newcommand{\N}{\mathbf{N}}
\newcommand{\Z}{\mathbf{Z}}

\newcommand{\cG}{\mathcal{G}}
\newcommand{\cO}{\mathcal{O}}
\newcommand{\co}{{\scriptstyle\mathcal{O}}}

\newcommand{\dif}{\operatorname{d}\!{}}
\newcommand\restr[3]{{%
  \left.\kern-\nulldelimiterspace %
  #1 %
  \vphantom{\big|} %
  \right|_{#2}^{#3} %
  }}
\DeclarePairedDelimiter{\braket}{\langle}{\rangle}%
\DeclarePairedDelimiter{\abs}{\lvert}{\rvert}%
\DeclarePairedDelimiter{\norm}{\lVert}{\rVert}%
\providecommand\given{}
\newcommand\SetSymbol[1][]{\nonscript\:#1\vert\allowbreak\nonscript\:\mathopen{}}
\DeclarePairedDelimiterX{\tuple}[1](){\renewcommand\given{\SetSymbol[\delimsize]}#1}
\DeclarePairedDelimiterX{\set}[1]\{\}{\renewcommand\given{\SetSymbol[\delimsize]}#1}
\DeclarePairedDelimiterXPP{\landauO}[1]{\cO}(){}{#1}
\DeclarePairedDelimiterXPP{\landauo}[1]{\co}(){}{#1}
\DeclarePairedDelimiterXPP{\landauOprec}[1]{\cO_\prec}(){}{#1}
\DeclarePairedDelimiterXPP{\landauOstd}[1]{\cO_\prec^2}(){}{#1}
\DeclarePairedDelimiterXPP{\landauOE}[1]{\cO_\prec^1}(){}{#1}
\DeclarePairedDelimiterXPP{\landauOd}[1]{\cO_\mathrm{m}}(){}{#1}

%% file: bibliography_setup.tex
\usepackage[giveninits=true,url=false,doi=false,isbn=false,eprint=true,datamodel=mrnumber,sorting=nty,maxcitenames=4,maxbibnames=99,backref=false,block=space,backend=biber,bibstyle=phys]{biblatex} 
\AtEveryBibitem{\clearfield{month}}
\AtEveryCitekey{\clearfield{month}} 
\renewbibmacro{in:}{}
\DeclareFieldFormat[article]{title}{\emph{#1}} 
\DeclareFieldFormat{mrnumber}{\ifhyperref{\href{http://www.ams.org/mathscinet-getitem?mr=#1}{\nolinkurl{MR#1}}}{\nolinkurl{#1}}}
\DeclareFieldFormat{pmid}{\ifhyperref{\href{https://www.ncbi.nlm.nih.gov/pubmed/#1}{\nolinkurl{PMID#1}}}{\nolinkurl{#1}}}
\DeclareFieldFormat{eprint}{\ifhyperref{\href{https://arxiv.org/abs/#1}{\nolinkurl{arXiv:#1}}}{\nolinkurl{#1}}}
\renewbibmacro*{doi+eprint+url}{%
  \iftoggle{bbx:doi}{\printfield{doi}}{}
  \newunit\newblock%
  \printfield{mrnumber}%
  \newunit\newblock%
  \printfield{pmid}%
  \newunit\newblock%
  \printfield{eprint}%
  \iftoggle{bbx:url}{\usebibmacro{url+urldate}}{}}

%% file: drawGraphs.tex
\usepackage{ifdraft}
\usepackage{tikz,pgfplots}
\usepackage{etoolbox,listings,xstring}

\makeatletter

\tikzset{
  fl/.style = {path fading=fade l},
  fr/.style = {path fading=fade r},
  wh/.style = {draw=none,fill=none},
  Gc/.style = {draw=none,circle split, inner sep=0pt,minimum size=8pt,rotate=90,path picture={\draw[pattern=#1] (0,0.07) circle (1.5pt); }},
  Gd/.style = {draw=none,circle split, inner sep=0pt,minimum size=8pt,rotate=270,path picture={\draw[pattern=#1] (0,0.07) circle (1.5pt); }},
  Gf/.style={draw=none,circle split, inner sep=1pt,minimum size=8pt,rotate=90},
  G/.style={circle,draw,minimum size=8pt,inner sep=1pt,font=\tiny},
  xx/.style={circle,fill,draw,inner sep=0pt,minimum size=3pt},
  ab/.style={circle,fill,draw=none,inner sep=0pt,minimum size=0pt,as=},
  Gend/.style={inner sep=0pt,minimum size=3pt},
  r/.style={draw=red},
  o/.style={draw=black,fill=white},
  lb/.style args={#1}{label=below:#1},
  lt/.style args={#1}{label=above:#1},
  lr/.style args={#1}{label=right:#1},
  ll/.style args={#1}{label=left:#1},
  fn/.style={fill=none},
  g/.style={postaction={decorate,decoration={
        markings,
        mark=at position .7 with {\arrow[#1]{Stealth[sep=-3pt]}}
      }}},
  s/.style={densely dashed,postaction={decorate,decoration={
        markings,
        mark=at position .7 with {\arrow[#1]{Stealth[sep=-3pt]}}
      }}},
  R/.style ={draw=lightgray, thick,densely dotted,edge node={node[above=-7pt] {\color{gray} $R$ }},anchor=south,pos=0.5,postaction={decoration={
        markings,
        mark=at position .7 with {\arrow[#1]{Stealth[sep=-3pt]}}
      },decorate}},
we/.style args ={#1}{draw=lightgray, thick,densely dotted,edge node={node[above=-7pt] {\color{gray} #1 }},anchor=south,pos=0.5},
IEg/.style ={draw=lightgray, thick,densely dotted},
  S/.style ={draw=lightgray, thick,densely dotted,edge node={node[above=-7pt] {\color{gray}$S$}},anchor=south,pos=0.5},  
  T/.style ={draw=lightgray, thick,densely dotted,edge node={node[above=-7pt] {\color{gray}$T$}},anchor=south,pos=0.5,postaction={decoration={
        markings,
        mark=at position .7 with {\arrow[#1]{Stealth[sep=-3pt]}}
      },decorate}},
  rpd/.style ={draw=lightgray,thick,densely dotted,edge node={node[above=-7pt] {\color{red}$\partial$}},anchor=south,pos=0.5,postaction={decoration={
        markings,
        mark=at position .7 with {\arrow[#1]{Stealth[sep=-3pt]}}
      },decorate}},
  pd/.style ={draw=lightgray,thick,densely dotted,edge node={node[above=-7pt] {\color{gray}$\partial$}},anchor=south,pos=0.5,postaction={decoration={
        markings,
        mark=at position .7 with {\arrow[#1]{Stealth[sep=-3pt]}}
      },decorate}},
  pdr/.style ={g,draw=lightgray,thick,densely dotted,edge node={node[above=-7pt] {\color{red}$\partial$}},anchor=south,pos=0.5},
  pdkr/.style args={#1}{g,draw=lightgray,thick,densely dotted,edge node={node[above=-7pt] { \color{gray}$\partial^{#1}$\color{red}$\partial$}},anchor=south,pos=0.5},
  pdh/.style ={g,draw=lightgray,thick,densely dotted,edge node={node[above=-7pt] {\color{gray}$\widehat\partial$}},anchor=south,pos=0.5},
  pdk/.style args={#1}{g,draw=lightgray,thick,densely dotted,edge node={node[above=-7pt] {\color{gray}$\partial^{#1}$}},anchor=south,pos=0.5},
  eq/.style = {double,postaction={decoration={name=none}}},
  inl/.style args={#1}{initial,initial where=left, initial text=#1,initial distance=10pt},
  pdn/.style args={#1}{g,draw=lightgray,thick,densely dotted,edge node={node[above=-7pt] {\color{gray}$\partial^{#1}$}},anchor=south,pos=0.5},
  inr/.style args={#1}{initial,initial where=right, initial text=#1,initial distance=10pt},
  int/.style args={#1}{initial,initial where=above, initial text=#1,initial distance=10pt},
  inb/.style args={#1}{initial,initial where=below, initial text=#1,initial distance=10pt},
  lpf/.style args={#1}{initial,initial where=left, initial text=$\vp\vf$,initial distance=10pt},
  tpf/.style args={#1}{initial,initial where=above, initial text=$\vp\vf$,initial distance=10pt},
  bpf/.style args={#1}{initial,initial where=below, initial text=$\vp\vf$,initial distance=10pt},
  rpf/.style args={#1}{initial,initial where=right, initial text=$\vp\vf$,initial distance=10pt},
  l1/.style args={#1}{initial,initial where=left, initial text=$\bm 1$,initial distance=10pt},
  t1/.style args={#1}{initial,initial where=above, initial text=$\bm 1$,initial distance=10pt},
  b1/.style args={#1}{initial,initial where=below, initial text=$\bm 1$,initial distance=10pt},
  r1/.style args={#1}{initial,initial where=right, initial text=$\bm 1$,initial distance=10pt},
  lm/.style args={#1}{initial,initial where=left, initial text=$\vm$,initial distance=10pt},
  tm/.style args={#1}{initial,initial where=above, initial text=$\vm$,initial distance=10pt},
  bm/.style args={#1}{initial,initial where=below, initial text=$\vm$,initial distance=10pt},
  rm/.style args={#1}{initial,initial where=right, initial text=$\vm$,initial distance=10pt},
  inr/.style args={#1}{initial,initial where=right, initial text=#1,initial distance=10pt},
  B/.style args={#1}{thick,draw=lightgray,decorate,decoration={snake,amplitude=.4mm,segment length=.8mm,post length=1.4mm},edge node={node[above=-7pt] {\color{gray} #1 }},anchor=south,pos=0.5},
  br/.style = {bend right},
  b0/.style = {bend left=0},
  bl/.style = {bend left},
  glb/.style = {looseness=20,in =220, out=320},
  gll/.style = {looseness=20,in =130, out=230},
  glr/.style = {looseness=20,in =310, out=50},
  glt/.style = {looseness=20,in =40, out=140},
  gm/.style={postaction={decorate,decoration={
        markings,
        mark=at position .3 with {\arrow[#1]{Diamond[open,sep=-3pt,width=5pt]}}
      }}},
}

\newcommand\sGraph[1]{
\begin{tikzpicture}[grow=right,baseline={([yshift=-2pt]current bounding box.center)},font=\footnotesize,>=Stealth]
  \graph[simple necklace layout,components go down left aligned,nodes={draw,circle,as=,minimum size=3pt,inner sep=0pt,fill},node distance = 20pt,component sep=25pt,edges=g]{ #1  };
  \end{tikzpicture}
}

\tikzset{circle split part fill/.style  args={#1}{%
 alias=tmp@name, 
  postaction={%
    insert path={
     \pgfextra{%
     \pgfpointdiff{\pgfpointanchor{\pgf@node@name}{center}}%
                  {\pgfpointanchor{\pgf@node@name}{east}}%
     \pgfmathsetmacro\insiderad{\pgf@x}
      \fill[white,fill opacity=0] (\pgf@node@name.base) ([xshift=-\pgflinewidth]\pgf@node@name.east) arc
                          (0:180:\insiderad-\pgflinewidth)--cycle;
      \fill[fill=white,preaction={fill, white},pattern=#1] (\pgf@node@name.base) ([xshift=\pgflinewidth]\pgf@node@name.west)  arc
                           (180:360:\insiderad-\pgflinewidth)--cycle;   
      \draw[line width=0.4pt] (\pgf@node@name.base) ([xshift=\pgflinewidth]\pgf@node@name.west)  arc
                           (180:360:\insiderad-\pgflinewidth)--cycle;                                
         }}}}}  
\makeatother  
\makeatletter
\tikzset{my loop/.style =  {to path={
  \pgfextra{}
  [looseness=6,min distance=4mm]
  \tikz@to@curve@path},font=\sffamily\small
  }}  
\makeatletter

\csedef{pat0}{}
\csedef{pat1}{north east lines}
\csedef{pat2}{crosshatch dots}
\csedef{pat3}{crosshatch}
\csedef{pat4}{grid}

\definecolor{col0}{HTML}{FFFFFF}
\definecolor{col1}{HTML}{A2B969}
\definecolor{col2}{HTML}{EBCB38}
\definecolor{col3}{HTML}{0D95BC}
\definecolor{col4}{HTML}{063951}
\definecolor{col5}{HTML}{F36F13}
\definecolor{col6}{HTML}{C13018}
\definecolor{lightgray}{HTML}{CCCCCC}

\newcommand\csum[1]{%
\sum_{\forcsvlist{\createColorCircle@item}{#1}}
}

\newcommand\createColorCircle@item[1]{
\StrDel{#1}{h}[\colNum]
\newif\ifhalf
\IfSubStr{#1}{h}{\halftrue}{\halffalse}
{\color{col\colNum}\ifhalf\circ\else\bullet\fi}
}

\newcommand\plotLambda[1]{
  \def\mgraphspecs{}
  \foreach \ll [count = \countl] in {#1} {
    \csedef{firstCol}{col0}
    \def\graphspecs{}
    \StrCount{\ll}{,}[\graphlen]
    \ifnum\graphlen>1 
    \foreach \node [count = \g] in \ll {
      \StrDel{\node}{m}[\nodenumber]
      \StrDel{\nodenumber}{.}[\nodenumber]
      \global\csedef{tempCol}{col\nodenumber}
      \global\csdef{tempPat}{\col{1}}
      \ifnum\g=1
        \global\csedef{firstCol}{\tempCol}
        \IfSubStr{\node}{.}{\global\csedef{open}{y}}{\global\csedef{open}{n}}
        \IfStrEqCase{\open}{
          {y}{\xappto\graphspecs{\countl\g[as=,draw=none] }}
          {n}{\xappto\graphspecs{\countl\g[as=,preaction={fill, white},fill=\tempCol,pattern=\csname pat\nodenumber\endcsname] }}
        }
      \else
        \IfBeginWith{\node}{m}{\global\csedef{secondArrow}{{<[sep=-3pt,length=8pt]}}}{\global\csedef{secondArrow}{}}%
        \ifnum\g=2
          \IfStrEqCase{\open}{
            {y}{\xappto\graphspecs{ --[dotted,thick] \countl\g[as=,preaction={fill, white},fill=\tempCol,pattern=\csname pat\nodenumber\endcsname] }}
            {n}{\xappto\graphspecs{ --[\firstArrow-\secondArrow] \countl\g[as=,preaction={fill, white},fill=\tempCol,pattern=\csname pat\nodenumber\endcsname] }}
          }
        \else
          \xappto\graphspecs{ --[\firstArrow-\secondArrow] \countl\g[as=,preaction={fill, white},fill=\tempCol,pattern=\csname pat\nodenumber\endcsname] }%
        \fi
      \fi
      \IfEndWith{\node}{m}{\global\csedef{firstArrow}{{>[sep=-3pt,length=8pt]}}}{\global\csedef{firstArrow}{}}
      \global\csedef{prevTempCol}{\tempCol} 
    }
    \IfStrEqCase{\open}{
      {n}{\IfBeginWith{\ll}{m}{\global\csedef{secondArrow}{{<[sep=-3pt,length=8pt]}}}{\global\csedef{secondArrow}{}}
          \xappto\graphspecs{ --[\firstArrow-\secondArrow,decorate,decoration={snake,amplitude=.3mm,segment length=.6mm}] \countl1; } }
      {y}{\xappto\graphspecs{ --[dotted,thick] \countl0[draw=none,as=] ; } }
    }
    \else
      \ifnum\graphlen=0
        \StrDel{\ll}{m}[\nodenumber]
        \csedef{tempCol}{col\nodenumber}
        \IfSubStr{\ll}{m}{\csedef{firstArrow}{{>[sep=-3pt,length=8pt]}}}{\csedef{firstArrow}{}}
        \def\graphspecs{ \countl1[as=,preaction={fill, white},fill = \tempCol,pattern=\csname pat\nodenumber\endcsname] --[my loop, decorate,decoration={snake,amplitude=.3mm,segment length=.6mm},\firstArrow-] \countl1; }
      \else
        \IfSubStr{\ll}{.}{
          \StrDel{\ll}{.}[\nodenumber]
          \StrDel{\nodenumber}{,}[\nodenumber]
          \csedef{tempCol}{col\nodenumber}
          \def\graphspecs{ \countl0[draw=none,as=,orient = left] --[dotted,thick] \countl1[as=,fill = \tempCol,pattern=\csname pat\nodenumber\endcsname] --[dotted,thick] \countl2[draw=none,as=,nudge down=10pt]; }
        }{
          \StrBefore{\ll}{,}[\nodeOne]
          \StrBehind{\ll}{,}[\nodeTwo]
          \StrDel{\nodeOne}{m}[\nodeOneNumber]
          \StrDel{\nodeTwo}{m}[\nodeTwoNumber]
          \def\tempColOne{col\nodeOneNumber}
          \def\tempColTwo{col\nodeTwoNumber}
          \IfEndWith{\nodeOne}{m}{\global\csedef{firstArrowOne}{{>[sep=-3pt,length=8pt]}}}{\global\csedef{firstArrowOne}{}}
          \IfEndWith{\nodeTwo}{m}{\global\csedef{firstArrowTwo}{{<[sep=-3pt,length=8pt]}}}{\global\csedef{firstArrowTwo}{}}
          \IfBeginWith{\nodeOne}{m}{\global\csedef{secondArrowOne}{{>[sep=-3pt,length=8pt]}}}{\global\csedef{secondArrowOne}{}}
          \IfBeginWith{\nodeTwo}{m}{\global\csedef{secondArrowTwo}{{<[sep=-3pt,length=8pt]}}}{\global\csedef{secondArrowTwo}{}}
          \def\graphspecs{ \countl1[as=,preaction={fill, white},fill = \tempColOne,pattern=\csname pat\nodeOneNumber\endcsname] --[bend right,decorate,decoration={snake,amplitude=.3mm,segment length=.6mm}, \firstArrowOne-\secondArrowTwo ] \countl2[as=,preaction={fill, white},fill = \tempColTwo,pattern=\csname pat\nodeTwoNumber\endcsname]; \countl1 --[bend left,\secondArrowOne-\firstArrowTwo] \countl2;}
        }
      \fi
    \fi
    \xappto\mgraphspecs{ \graphspecs }
  }
  \xdef\mgraphspecs{\noexpand\graph[simple necklace layout,componentwise,component packing=skyline,components go right center aligned,orient=0,nodes=G]{ \mgraphspecs }}
  \begin{tikzpicture}[baseline={([yshift=-2pt]current bounding box.center)},font=\tiny,>=Stealth, node distance = 15pt,node sep=12pt,component sep=5pt]
    \mgraphspecs;
  \end{tikzpicture}%
}

\newcommand\pB[1]{
  \StrDel{#1}{h}[\patNum]
  \IfSubStr{#1}{h}{
  \begin{tikzpicture}[baseline={([yshift=-2pt]current bounding box.center)},font=\tiny,>=Stealth, node distance = 0pt,node sep=1pt,component sep=1pt]
   \graph[simple necklace layout,componentwise,component packing=skyline,components go right center aligned,orient=0,nodes=G] { 1[Gf,rotate=270,circle split part fill={\csname pat\patNum\endcsname},as=,minimum size=6pt]; };
  \end{tikzpicture}
  }{
  \begin{tikzpicture}[baseline={([yshift=-2pt]current bounding box.center)},font=\tiny,>=Stealth, node distance = 0pt,node sep=1pt,component sep=1pt]
   \graph[simple necklace layout,componentwise,component packing=skyline,components go right center aligned,orient=0,nodes=G] { 1[pattern=\csname pat\patNum\endcsname,as=,minimum size=6pt]; };
  \end{tikzpicture}
  }
}
\newcommand\plotlLambda[1]{
  \def\mgraphspecs{}
  \foreach \ll [count = \countl] in {#1} {
    \csedef{firstCol}{col0}
    \csedef{openEnd}{no}
    \def\graphspecs{}
    \StrCount{\ll}{,}[\graphlen]
    \foreach \node [count = \g] in \ll {
      \StrDel{\node}{m}[\nodenumber]
      \csedef{type}{n}
      \csedef{marking}{no}
      \IfSubStr{\node}{x}{\csedef{type}{x}}{}
      \IfSubStr{\node}{c}{\csedef{type}{c}}{}
      \IfSubStr{\node}{d}{\csedef{type}{d}}{}
      \IfSubStr{\node}{.}{\csedef{type}{open}}{}
      \IfBeginWith{\node}{m}{\csedef{firstMarking}{yes}\csedef{marking}{yes}}{\csedef{firstMarking}{no}}
      \IfEndWith{\node}{m}{\csedef{secondMarking}{yes}\csedef{marking}{yes}}{\csedef{secondMarking}{no}}
      \StrDel{\nodenumber}{.}[\nodenumber]
      \StrDel{\nodenumber}{c}[\nodenumber]
      \StrDel{\nodenumber}{d}[\nodenumber]
      \StrDel{\nodenumber}{x}[\nodenumber]
      \global\csedef{tempCol}{col\nodenumber}
      \ifnum\g=1
        \global\csedef{firstCol}{\tempCol}
        \IfStrEq{\type}{open}{
          \xappto\graphspecs{ \countl0[xx,as=,grow right,draw=none,fill=white] --[dotted,thick] \countl1[fill=white,preaction={fill, white},pattern=\csname pat\nodenumber\endcsname,as=]}
        }{
        \IfSubStr{\node}{x}{
          \xappto\graphspecs{ \countl0[xx,as=,grow right,fill=lightgray,draw=none] --[color=lightgray]\countl1[Gf,circle split part fill={\csname pat\nodenumber\endcsname},as=]}}
        {
          \IfStrEq{\marking}{yes}
          {
            \xappto\graphspecs{ \countl0[xx,as=,grow right] --[-{<[fill=lightgray,color=black,sep=-3pt,length=8pt]}] \countl1[fill=white,preaction={fill, white},pattern=\csname pat\nodenumber\endcsname,as=]}
          }{
            \xappto\graphspecs{ \countl0[xx,as=,grow right] -- \countl1[fill=white,preaction={fill, white},pattern=\csname pat\nodenumber\endcsname,as=]}
          }
        }}
      \else
        \IfStrEq{\firstMarking}{yes}{\csedef{secondArrow}{{<[sep=-3pt,length=8pt]}}}{\csedef{secondArrow}{}}
        \IfStrEqCase{\type}{
          {n}{\xappto\graphspecs{ --[\firstArrow-\secondArrow] \countl\g[fill=white,preaction={fill, white},pattern=\csname pat\nodenumber\endcsname,as=] }}%
          {c}{\xappto\graphspecs{ --[\firstArrow-\secondArrow] \countl\g[Gc=\csname pat\nodenumber\endcsname,rotate=180,circle split part fill={\csname pat\nodenumber\endcsname},as=] }}%
          {d}{\xappto\graphspecs{ --[\firstArrow-\secondArrow] \countl\g[Gd=\csname pat\nodenumber\endcsname,rotate=180,circle split part fill={\csname pat\nodenumber\endcsname},as=] }}%
          {open}{\global\csedef{openEnd}{yes}}%
        }
      \fi
      \IfStrEq{\secondMarking}{yes}{
        \global\csedef{firstArrow}{{>[sep=-3pt,length=8pt]}}%
      }{
        \global\csedef{firstArrow}{}%
      }
      \global\csedef{prevTempCol}{\tempCol} 
    }
    \IfStrEqCase{\openEnd}{
      {yes}{\xappto\mgraphspecs{ \graphspecs --[dotted,thick] \countl17[Gend,as=,draw=none]; }}%
      {no}{\xappto\mgraphspecs{ \graphspecs -- \countl17[Gend,as=]; }}%
    }%
  }
  \xdef\mgraphspecs{\noexpand\graph[tree layout,componentwise,component packing=skyline,components go down left aligned,nodes=G]{ \mgraphspecs }}
  \begin{tikzpicture}[grow=right,baseline={([yshift=-2pt]current bounding box.center)},font=\tiny,>=Stealth, node distance = 5pt,node sep=12pt,component sep=5pt]
  \mgraphspecs;
  \end{tikzpicture}%
}